\documentclass[11pt]{amsart}

\usepackage{amssymb}
\usepackage{amsmath}
\usepackage{amsthm}
\usepackage{graphicx}
\usepackage{macros}
\usepackage[usenames,dvipsnames]{xcolor}

\usepackage{subcaption}
\usepackage{tikz}
\usepackage{float}
\usepackage{color}
\usepackage{hyperref}
\hypersetup{
  colorlinks,
  citecolor= Brown,
  linkcolor= link,
  urlcolor= link
}

\usepackage{appendix}
\usepackage{array}
\usepackage{footnote}
\usepackage{booktabs}
\usepackage{hhline}
\makesavenoteenv{tabular}

\definecolor{link}{rgb}{0.18,0.25,0.63}
\usepackage[top=3cm,bottom=3cm,left=2.5cm,right=2.5cm]{geometry}

\numberwithin{equation}{section}

\DeclareMathOperator*{\esssup}{ess\,sup}
\DeclareMathOperator*{\essinf}{ess\,inf}
\DeclareMathOperator*{\supp}{\mathrm{supp}}

\makeatletter
\g@addto@macro{\endabstract}{\@setabstract}
\newcommand{\authorfootnotes}{\renewcommand\thefootnote{\@fnsymbol\c@footnote}}%
\makeatother

\begin{document}
\definecolor{link}{rgb}{0,0,0}

 \begin{center}
 \large
  \textbf{  ANALYTICAL ASPECTS OF SPATIALLY ADAPTED TOTAL VARIATION REGULARISATION}. \par \bigskip \bigskip
  
   \normalsize
  \textsc{Michael Hinterm\"uller} \textsuperscript{$\dagger$ $\ddagger$}, \textsc{Konstantinos Papafitsoros} \textsuperscript{$\dagger$},
  \textsc{and} \textsc{Carlos N. Rautenberg} \textsuperscript{$\ddagger$} \par \bigskip 

\let\thefootnote\relax\footnote{
\textsuperscript{$\dagger$}Weierstrass Institute for Applied Analysis and Stochastics (WIAS), Mohrenstrasse 39, 10117, Berlin, Germany
}
\let\thefootnote\relax\footnote{
\textsuperscript{$\ddagger$}Institute for Mathematics, Humboldt University of Berlin, Unter den Linden 6, 10099, Berlin, Germany}

\let\thefootnote\relax\footnote{
\hspace{3.2pt}Emails: \href{mailto:Michael.Hintermueller@wias-berlin.de}{\nolinkurl{Michael.Hintermueller@wias-berlin.de}},
 \href{mailto: Kostas.Papafitsoros@wias-berlin.de}{\nolinkurl{Kostas.Papafitsoros@wias-berlin.de}},}
 \let\thefootnote\relax\footnote{
 \hspace{37pt}\href{mailto:carlos.rautenberg@math.hu-berlin.de}{\nolinkurl{Carlos.Rautenberg@math.hu-berlin.de}}
 }

\end{center}

\definecolor{link}{rgb}{0.18,0.25,0.63}
\begin{abstract}
In this paper we study the structure of solutions of the one dimensional weighted total variation regularisation problem, motivated by its application in signal recovery tasks. We study in depth the relationship between the weight function and the creation of new discontinuities in the solution. A partial semigroup property relating the weight function and the solution is shown and analytic solutions for simply data functions are  computed. We  prove that the weighted total variation minimisation problem is well-posed even in the case of vanishing weight function, despite the lack of coercivity. This is based on the fact that the total variation of the solution is bounded by the total variation of the data, a result that it also shown here. Finally the relationship to the corresponding weighted fidelity problem is explored, showing that the two problems can produce completely different solutions even for very simple data functions.\vspace{0.1cm}

 \end{abstract}

\section{Introduction}

A general task in mathematical signal reconstruction is to recover as best as possible a signal $u_{0}$, given a corrupted version $f$, 
which is generated by the following degradation process:
\begin{equation}\label{lbl:intro:general}
f=Tu_{0}+\eta.
\end{equation}
Here $T$ denotes a bounded, linear operator and $\eta$ is a random noise component. The mapping $T$ might be related to blurring, downscaling, Fourier or wavelet transform, among several others.
 The problem aforementioned reconstruction problem \eqref{lbl:intro:general} is typically ill-posed and variational regularisation methods are often employed for its solution. A specific, very successful regularisation model is given by total variation minimisation as introduced in the seminal work by Rudin, Osher and Fatemi \cite{rudin1992nonlinear}. In that paper, the authors considered the case $T=id$, i.e., the denoising task, and they proposed to recover an approximation $u$ of $u_{0}$ by solving the discrete version of the  minimisation problem
\begin{equation}\label{lbl:intro:rof}
\min_{u\in\bv(\om)} \frac{1}{2} \int_{\om} (f-u)^{2}dx+\alpha|Du|(\om),
\end{equation}
where $\alpha$ is a positive scalar  and $|Du|(\om)$ denotes the total variation of the function $u$. Here $\om$ represents a bounded, open domain with Lipschitz boundary. In image reconstruction tasks, $\om$ is typically a rectangle on which the image is defined. 

Ever since, total variation minimisation has been employed for a variety of image restoration tasks mainly due its edge-preserving ability. This stems from the fact that the minimisation  \eqref{lbl:intro:rof}, is performed over the space of functions of bounded variation $\bv(\om)$. We note that an element of $\bv(\om)$ may exhibit jump discontinuities. One disadvantage of the model \eqref{lbl:intro:rof}, on the other hand, is the promotion of piecewise constant structures in the solution $u$, a phenomenon known as staircasing effect. To overcome this, higher order extensions of the total variation have been proposed in the literature. Here, we only mention the review paper \cite{S15}, as well as the references collected in the introduction of \cite{papafitsorosphd}.

Another drawback of \eqref{lbl:intro:rof} originates from the fact that the regularisation strength is uniform over the entire image domain,  due to the regularisation parameter $\alpha$ being a scalar quantity only. This is particularly disadvantageous when the noise level or the amount of corruption in general, is not distributed uniformly throughout the image. Regularisation of uniform strength  is also undesirable when both fine scales details, e.g. texture, and large homogeneous areas are present in an image. In that case, one ideally should strongly regularise in  the smooth parts of the image and to a lesser degree in fine detailed areas in order for these details to be better preserved. Therefore, the introduction of spatially distributed weights in the minimising functional in \eqref{lbl:intro:rof} has been considered in the literature. This weight can be either introduced in the first term of \eqref{lbl:intro:rof}, i.e., the so-called fidelity or data fitting term, or be incorporated into the total variation functional. For the denoising case, this leads to the following two models
\begin{equation}\label{lbl:intro:weight_fid}
\min_{u\in\bv(\om)} \frac{1}{2}\int_{\om} w(f-u)^{2}dx+|Du|(\om),
\end{equation}
where $w\in L^{\infty}(\om)$ with $w\ge 0$, and
\begin{equation}\label{lbl:intro:weighted_rof}
\min_{u\in\bv(\om)} \frac{1}{2}\int_{\om} (f-u)^{2}dx+\int_{\om}\alpha(x)d|Du|,
\end{equation}
where $\alpha\in C(\overline{\om})$ with $\alpha>0$.
 Here $\alpha$ and $w$ are the two weight functions that determine locally the strength of the regularisation and the fidelity term respectively.  

Versions of the weighted fidelity model \eqref{lbl:intro:weight_fid} have been considered in \cite{Dong2010, Dong2010b} for Gaussian denoising and deblurring image restoration problems as well as in \cite{hint_rincon} for reconstructing images that have been corrupted by impulse noise. In these works, the weight function $w$ is selected based on local statistical estimators and the statistics of the extremes. An adaptation of this idea to TGV (total generalised variation) \cite{tgv}, a higher order extension of the total variation, can be found in \cite{bredies2013spatially}.
 A different statistical approach with variance estimators is considered in \cite{Almansa2007}.
Variants of \eqref{lbl:intro:weight_fid} are also studied in \cite{frick2012, Hotz2012543} using techniques based on a statistical multiresolution criterion. The model \eqref{lbl:intro:weight_fid} is also considered in \cite{bertalmio_caselles} where a piecewise constant weight function  is determined using a pre-segmentation of the image. 

The weighted total variation model \eqref{lbl:intro:weighted_rof} has been considered recently for image restoration purposes in \cite{hint_rau, hint_rau_tao_langer}. In these papers, the choice of the weight function $\alpha$ is done via a bilevel optimisation approach, see also \cite{carola_spatial_2, carola_spatial_1}. Moreover, apart from the classical denoising and deblurring tasks, the fact that the fidelity term in \eqref{lbl:intro:weighted_rof} appears without weights, allows the authors of \cite{hint_rau, hint_rau_tao_langer}  to consider problems also in Fourier and wavelet domains, e.g., Fourier and wavelet impainting, something which highlights an advantage of the model \eqref{lbl:intro:weighted_rof} over \eqref{lbl:intro:weight_fid}.  We also mention that recently, a weighted $\tv$ regularisation for vortex density models was studied in \cite{novaga_weighted}. 

While the analysis of the regularisation properties of the scalar total variation regularisation \eqref{lbl:intro:rof} has  received a considerable amount of attention in the literature \cite{ Acar94analysisof, ring2000structural, meyer2001oscillating, strong2003edge, chanL1, grasmair2007equivalence, caselles2007discontinuity,  Allard1, Allard2, Allard3, duvalL1, valkonen_jump_1, Jalalzai2015jmiv, poon_TV_geometric} this is not the case for the models \eqref{lbl:intro:weight_fid} and \eqref{lbl:intro:weighted_rof}. We note however two analytical contributions towards the weighted total variation model \eqref{lbl:intro:weighted_rof}. Specifically, in \cite{jalalzai2014discontinuities}, the author showed 
that the set of the jump discontinuities of the solution $u$ of \eqref{lbl:intro:weighted_rof} is essentially contained in the set that consists of the jump discontinuity points of the data $f$ and the jump discontinuity points of the gradient of the weight function $\alpha$. This result shows that the solution $u$ can potentially have jump discontinuities at points where the data function $f$ is continuous. Hence, if the weight function is smooth,  no new discontinuities are created. In the scalar parameter case, this was first shown in  \cite{caselles2007discontinuity} and subsequently in \cite{valkonen_jump_1} using a different technique. In \cite{novaga_weighted} the authors show, among others, that the maximal level set of the solution $u$ is flat and has positive measure, as it is also the case for the scalar total variation regularisation \cite{Jalalzai2015jmiv}.

However, there are still several open questions regarding the models \eqref{lbl:intro:weight_fid} and \eqref{lbl:intro:weighted_rof}. For instance, concerning the weighted total variation model \eqref{lbl:intro:weighted_rof}, one is interested to understand, under which specific conditions new discontinuities are created, and how  these are related to the weight function $\alpha$. It is also important to examine in what degree  the structure of solutions of the weighted total variation minimisation resembles the one of the solutions of the standard scalar minimisation problem. Finally, it is of importance to understand  the similarities and the differences of the two weighted models  \eqref{lbl:intro:weight_fid} and \eqref{lbl:intro:weighted_rof}.

In view of this, the purpose of the present paper is to answer the questions raised above as well as related ones, thus filling in that knowledge gap in the literature. We do that by a extended fine scale analysis of the regularisation properties of the one dimensional versions of the problems \eqref{lbl:intro:weight_fid} and \eqref{lbl:intro:weighted_rof}. We should note, however, that the majority of our results concern the weighted total variation model \eqref{lbl:intro:weighted_rof}, since, as we will see in the following sections,  it is the one that exhibits a greater variety of interesting properties.

\subsection*{Summary of the results and organisation of the paper}

For the reader's convenience we provide here a short summary of our results, stating as well the sections of the paper that each of these belongs to. The results are put into perspective with the literature in the corresponding sections.
\subsubsection*{Structure of solutions-creation of new discontinuities} 
After fixing the notation and recalling some preliminary facts  in Section \ref{sec:not_prel},  we study in Sections \ref{sec:optimality}--\ref{sec:discontinuities} the weighted total variation problem \eqref{lbl:intro:weighted_rof}  and the conditions under which new discontinuities are created in its solution $u$. We give a simple proof of a refined version of the result in \cite{jalalzai2014discontinuities} in Proposition \ref{lbl:jump_set_incl}, showing that new jump discontinuities can potentially be created at the points where the weight function $\alpha$ is not differentiable. Note that in the case where $\alpha'\in\bv(\om)$ these are exactly the set of jump discontinuity points of $\alpha'$, i.e., the set of points $x\in \om$ such that $|D\alpha'|(\{x\})>0$. In fact, we show that in order for a new discontinuity to be created at $x$, there must hold $D\alpha'(\{x\})>0$, i.e., the derivative of $\alpha'$ must have a positive jump, Proposition \ref{lbl:lipschitz_alpha}. In contrast, if $D\alpha'(\{x\})<0$ then a plateau is created for the solution $u$ around $x$. Furthermore, we show that in every point $x\in\om$, the following estimate holds
\begin{equation}\label{intro:jump_estimate}
|Du|(\{x\})\le |Df|(\{x\})+|D\alpha'|(\{x\}),
\end{equation}
see Propositions \ref{lbl:lipschitz_alpha}, \ref{lbl:correct_jump} as well as Corollary \ref{lbl:lipschitz_alpha_ii}. 
Moreover, in the weighted case, the jump of $u$ can have different direction from the one of the data $f$, something that does not occur in the scalar case. We show however that the jumps  of $f$ and $u$ at a point,  have the same direction when $\alpha$ is differentiable but are not necessarily  \emph{nested}, see Proposition \ref{lbl:correct_jump} and the numerical examples of Figure \ref{fig:counterexamples}. Finally it is shown that if $\alpha'$ is large enough in an area, then $u$ is constant there, Proposition \ref{lbl:large_gradient_plateau}. Thus, it is not only high values of $\alpha$ that can produce flat areas as someone might expect, but also high values of $\alpha'$.

\subsubsection*{A partial semigroup property}
In Section  \ref{sec:semigroup} we show that, denoting by $S_{\alpha}(f)$ the solution of \eqref{lbl:intro:weighted_rof} with data $f$ and weight $\alpha$, it holds
\[S_{\alpha_{1}+\alpha_{2}}(f)=S_{\alpha_{2}}(S_{\alpha_{1}}(f)),\]
provided $\alpha_{2}$ is a scalar. This is shown in Proposition \ref{lbl:semigroup} and we call this property \emph{partial semigroup property}. On the other hand, one can easily construct counterexamples where this property fails even in the case where $\alpha_{1}$ is scalar and $\alpha_{2}$ is not, see Figures \ref{fig:sg_spikes} and \ref{fig:sg2}, i.e., unlike the full scalar case, this partial semigroup property is not  commutative, something perhaps surprising.  
 
 \subsubsection*{Analytic solutions} 
 In Section \ref{sec:exact} we compute some analytic solutions for simple data and weight functions. In particular, we take as data a family of affine functions and as  weight functions, a family of absolute value type functions. The formulae of the solutions are summarised  in Proposition \ref{lbl:exact_affine_abs}, also  depicted in Figure \ref{fig:affine_abs}. Note that this is the first example, where the creation of new discontinuities is computed analytically. 
 
\subsubsection*{A bound on the total variation of the solution}
In Section \ref{sec:boundTV} we show that  for the solution of the weighted total variation minimisation problem \eqref{lbl:intro:weighted_rof}, the following estimate holds
\begin{equation}\label{intro:boundtv}
|Du|(\om)\le |Df|(\om).
\end{equation}
Unlike the scalar case, the proof of \eqref{intro:boundtv} is quite involved  and uses some fine scale analysis.
We do that initially for differentiable weight $\alpha$ in Theorem \ref{lbl:TVu_less_TVf} and then for continuous one in Theorem \ref{lbl:TVu_less_TVf_2}.

\subsubsection*{Vanishing weight function $\alpha$} 
Provided that $f\in\bv(\om)$,  we show in Section \ref{sec:alpha_zero} the existence of solutions for \eqref{lbl:intro:weighted_rof} even when $\alpha\ge 0$, despite the lack of coercivity of the minimising functional, see Theorem \ref{lbl:well_alpha_zero}. Letting $\alpha$ having zero values, can allow an exact recovery of piecewise constant functions, as we show in Proposition  \ref{lbl:exact_pc}. 

\subsubsection*{Relationship of the models \eqref{lbl:intro:weight_fid} and \eqref{lbl:intro:weighted_rof}}
In Section \ref{sec:weight_fid}, we show that the structure of the solutions of the weighted fidelity problem \eqref{lbl:intro:weight_fid} is simpler and resembles more the one of the scalar case. We prove that no new discontinuities are created, provided that $w>0$, Proposition \ref{lbl:weight_fid_discont}. Moreover, by considering the same family of simple affine data functions for which we computed analytic solutions for the problem \eqref{lbl:intro:weighted_rof}, we see that the solutions here are much simpler, see Proposition \ref{lbl:affine_exact_fid}. Interestingly, for these specific data functions,  the sets of solutions of the problems  \eqref{lbl:intro:weight_fid} and \eqref{lbl:intro:weighted_rof} are totally different, regardless of the choice of weight functions $\alpha$ and $w$. In fact, the only common solutions that they have are the ones that can  be also obtained by the standard scalar total variation minimisation, see Proposition \ref{lbl:alpha_cannot} and   Figure \ref{fig:venn}. This shows how different can the models \eqref{lbl:intro:weight_fid} and \eqref{lbl:intro:weighted_rof} be, even for very simple data functions.

\section{Notation and Preliminaries}\label{sec:not_prel}

Functions of bounded variation play a central role in this paper. Standard references are the books \cite{AmbrosioBV, attouch2014variational, EvansGariepy, giusti1984minimal}. Here we follow the notation of \cite{AmbrosioBV}. Let $\om\subseteq \RR^{d}$ be a open set, $d\in\NN$. 
 Given a finite Radon measure $\mu\in\cM(\om)$ we denote by $|\mu|$ its total variation measure and by $\sgn(\mu)$ the unique $L^{1}(\om,|\mu|)$ function such that $\mu=\sgn(u)|\mu|$. That is to say $\sgn(\mu)$ is the Radon-Nikod\'ym derivative
$\sgn(\mu)=\frac{d\mu}{d|\mu|}$,
which is equal to $1$ $|\mu|$--almost everywhere.
 A function $u\in L^{1}(\om)$ is said to be a \emph{function of bounded variation} if its distributional derivative is represented by a $\RR^{d}$-valued finite Radon measure, denoted by $Du$.
Equivalently, $u$ is a function of bounded variation if its \emph{total variation} $\tv(u)$ is finite, where
\[\tv(u):=\sup \left \{\int_{\om}u\,\di v\,dx:\; v\in C_{c}^{1}(\om,\RR^{d}),\; \|v\|_{\infty}\le 1 \right \},\]
 and in that case it can be shown that $\tv(u)=|Du|(\om)$. The space of functions of bounded variation is denoted by $\bv(\om)$ and is a Banach space under the norm $\|u\|_{\bv(\om)}=\|u\|_{L^{1}(\om)}+|Du|(\om)$. The measure $Du$ can be decomposed into the absolutely continuous and the singular part with respect to the Lebesgue measure $\mathcal{L}$, $D^{a}u$ and $D^{s}u$ respectively, i.e.,
 $Du=D^{a}u+D^{s}u$.
 
In this paper,  emphasis is given on the functions of bounded variation of one variable and in particular on the notion of \emph{good representatives}. In order to define these, let $\om=(a,b)$ be a bounded open interval in $\RR$. For a function $u:\om \to \RR$, the \emph{pointwise variation} of $u$ in $\om$ is defined as 
\[\mathrm{pV}(u,\om)=\sup \left \{\sum_{i=1}^{n-1} |u(x_{i+1})-u(x_{i})|:\;n\ge 2,\; a<x_{1}<\cdots<x_{n}<b \right \},\]
and the \emph{essential variation} as
\[\mathrm{eV}(u,\om)=\inf \big\{\mathrm{pV}(v,\om): v=u,\; \cL-\text{a.e.}\text{ in }\om \big \}. \]
It turns out that when $u\in \bv(\om)$ then $|Du|(\om)=\mathrm{eV}(u,\om)$ and in fact the infimum in the definition of $\mathrm{eV}(u,\om)$ is attained. The functions in the equivalence class of $u$ that attain this infimum are called good representatives of $u$. That is to say $\tilde{u}$ is a good representative of $u$ if $\tilde{u}=u$ Lebesgue--almost everywhere and
\[\mathrm{pV(\tilde{u})}=\mathrm{eV}(u)=|Du|(\om).\]
 We denote by $J_{u}$ the at most countable set of atoms of $Du$ (jump set of $u$), i.e., $J_{u}=\{x\in \om:\; |Du|(\{x\})\ne 0\}$. If $Du(\{x\})> 0$ we say that $u$ has a positive jump at $x$, whereas if $Du(\{x\})< 0$ we say that $u$ has a negative jump at $x$.
It can be shown that there exists a unique $c\in \RR$ such that the functions 
\[
u^{l}(x):=c+Du((a,x)),\quad
u^{r}(x):=c+Du((a,x]),
\]
are good representatives of $u$. Note that $u^{l}$ and $u^{r}$ are left and right continuous respectively. The following equalities also hold
\[
u^{l}(x)= \lim_{\delta\shortdownarrow 0} \int_{x-\delta}^{x} u(t)dt,\quad 
u^{r}(x)= \lim_{\delta\shortdownarrow 0} \int_{x}^{x+\delta} u(t)dt,\quad \text{for all } x\in\om.
\]
 Any other function $\tilde{u}:\om\to \RR$ is a good representative of $u$ if and only if 
\[\tilde{u}(x)\in \left \{\theta u^{l}(x)+(1-\theta)u^{r}(x):\;\theta\in [0,1] \right \}.\] 
As a result, the following functions are also good representatives
\[
\overline{u}(x):=\max(u^{l}(x),u^{r}(x)),\quad
\underline{u}(x):=\min(u^{l}(x),u^{r}(x)).
\]
The right and the left limits of any good representative $\tilde{u}$ exist at any point of $x\in\om$ and 
$
\tilde{u}(x_{+})=u^{r}(x)$, 
$\tilde{u}(x_{-})=u^{l}(x).
$
Every good representative of $u$ is continuous at the complement of the jump set of $u$ i.e., in the set $\{x\in\om:\; |Du|(\{x\})=0\}$. 

We denote by $u'$ the density of $D^{a}u$ with respect to the Lebesgue measure, i.e., 
\[Du=u'\mathcal{L}+D^{s}u.\]
If $u\in W^{1,1}(\om)$ then $u'$ is the standard weak derivative of $u$.

Recall some basic notions from convex analysis. If $X$, $X^{\ast}$ are two vector spaces placed in duality and $F:X\to \RR\cup \{+\infty\}$ then $F^{\ast}$ denotes the convex conjugate of $F$
\[F^{\ast}(x^{\ast}):=\sup _{x\in X} \langle x^{\ast}, x \rangle -F(x).\]
The subdifferential of $F$ is denoted as usual by $\partial F$. Given $A\subseteq X$ then $\mathcal{I}_{A}$ denotes the indicator function of $A$
\[\mathcal{I}(x)=
\begin{cases}
0, & \text{ if }x\in A,\\
+\infty, & \text{ if }x\notin A.
\end{cases}
\]

We finally note that whenever we write \emph{total variation regularisation} or \emph{total variation minimisation} we always mean the total variation denoising problem with $L^{2}$ fidelity term.

\section{Weighted total variation with strictly positive weight function $\alpha$}\label{sec:main_weighted_rof}
The problem we are considering here is the one dimensional weighted total variation regularisation problem with $L^{2}$ fidelity term, i.e.,
\begin{equation}\label{weighted_rof}
\min_{u\in\bv(\om)} \frac{1}{2}\int_{\om}(f-u)^{2}dx+\int_{\om}\alpha(x)d|Du|,
\end{equation}
where $\om=(a,b)$, $f\in\bv(\om)$ and $\alpha\in C(\overline{\om})$ with $\alpha>0$.
Thus, there exist constants $0<c_{\alpha}\le C_{\alpha}<\infty$ such that
\[0<c_{\alpha}\le \alpha(x)\le C_{\alpha}<\infty,\quad \text{ for all }x\in\om.\]
The well-posedness (existence and uniqueness) of \eqref{weighted_rof} in all dimensions, i.e., when $\om\subseteq \RR^{d}$, is proven via the direct method of calculus of variations taking advantage of the fact that a strictly positive weight function $\alpha$ provides the necessary coercivity to the weighted $\tv$ functional, see \cite{hint_rau} for details. Among others, the authors in \cite{hint_rau}  prove that for the anisotropic version of weighted $\tv$ it holds
\begin{equation}\label{weighted_tv_dual_Hdiv}
\int_{\om}\alpha(x)d|Du|=\sup \left\{\int_{\om}u\,\di v\,dx:\; v\in H_{0}(\om, \di),\;|v_{i}(x)|\le \alpha (x), \text{ a.e. }i=1,\ldots,d \right \}.
\end{equation}
Using also appropriate density arguments, the isotropic version of \eqref{weighted_tv_dual_Hdiv} reads
\begin{equation}\label{weighted_tv_dual}
\int_{\om}\alpha(x)d|Du|=\sup \left\{\int_{\om}u\, \di v\,dx:\; v\in C_{c}^{1}(\om,\RR^{d}),\;|v(x)|\le \alpha (x), \text{ for every }x\in\om\right \},
\end{equation}
where in the expression above $|\cdot|$ denotes the Euclidean norm in $\RR^{d}$. It is then clear that the weighted $\tv$ is lower semicontinuous with respect to the strong convergence in $L^{1}$.

\subsection{Optimality conditions}\label{sec:optimality}
We now proceed to the derivation of the optimality conditions for the minimisation problem \eqref{weighted_rof}. This is done via the Fenchel-Rockafellar duality theory, see for instance \cite{ekeland1976convex}. 

We start with some useful  definitions. For a finite Radon measure $\mu\in\cM(\om)$ we define
\begin{equation}\label{Sgn}
\Sgn(\mu):=\left\{v\in L^{\infty}(\om)\cap L^{\infty}(\om,\mu):\; \|v\|_{\infty}\le 1,\;v=\sgn(\mu),\; |\mu|-\text{a.e.} \right\},
\end{equation}
i.e., the set of all the functions $v$ that are $\mu$-almost everywhere equal to $\frac{d\mu}{d|\mu|}$ with the extra property that their absolute values is less than $1$, Lebesgue--almost everywhere. The definition \eqref{Sgn} originates from \cite{BrediesL1}. 
For a function $\alpha\in C(\overline{\om})$, we also define
\begin{equation}\label{aSgn}
\alpha(x) \Sgn(\mu):=\left\{v\in L^{\infty}(\om)\cap L^{\infty}(\om,\mu):\; v=\alpha \tilde{v} \text{ for some } \tilde{v}\in \Sgn(\mu) \right\}.
\end{equation}
Notice that we  slightly abuse the notation in the definition \eqref{aSgn} where we  denote the set by ``$\alpha(x) \Sgn(\mu)$'' instead of ``$\alpha\Sgn(\mu)$'' in order to stress the fact that $\alpha$ is not necessarily a constant function.

The following proposition is an extension of \cite[Lemma 3.5]{BrediesL1} to the weighted case.

\newtheorem{alpha_sgn}{Lemma}[section]
\begin{alpha_sgn}[Subdifferential of the weighted Radon norm]\label{lbl:alpha_sgn}
Let $\alpha\in C(\overline{\om})$. Consider the map $\|\cdot\|_{\cM,\alpha}:\cM (\om)\to \RR$ where
\[\|\mu\|_{\cM,\alpha}=\int_{\om}\alpha(x)d|\mu|,\quad \mu\in\cM(\om).\]
Then for every $\mu\in\cM(\om)$
\[\partial \|\cdot\|_{\cM,\alpha} (\mu) \cap C_{0}(\om)=\alpha(x) \Sgn(\mu)\cap C_{0}(\om),\]
\end{alpha_sgn}

\begin{proof}
Fix $\mu\in\cM(\om)$ and let $v\in \partial \|\cdot\|_{\cM,\alpha} (\mu) \cap C_{0}(\om)$. Then  
\begin{align}
\int_{\om}\alpha(x)d|\mu|+\int_{\om}v(x)d(\nu-\mu)&\le \int_{\om}\alpha(x) d|\nu| \quad \text {for every }\nu\in \cM(\om)\Rightarrow\label{asgn_1}\\
					      \int_{\om}v(x)d(\nu-\mu)&\le \int_{\om} \alpha (x)d|\nu-\mu| \quad \text {for every }\nu\in \cM(\om) \Rightarrow\notag\\
		\int_{\om}v(x)d\nu&\le \int_{\om} \alpha (x)d|\nu| \quad \text {for every }\nu\in \cM(\om). \label{asgn_2}			      
\end{align}
From the  inequality \eqref{asgn_2} we deduce that
\begin{equation}\label{asgn_equi_1}
|v(x)|\le \alpha (x)\quad \text{for every }x\in \om.
\end{equation}
Observe that it also holds
\begin{equation}\label{asgn_equi_2}
\int_{\om}v(x)d\mu=\int_{\om}\alpha(x)d|\mu|.
\end{equation}
Indeed, just consider \eqref{asgn_1} with $\nu=0$ and $\nu=2\mu$. One can readily check that if  a function $v\in C_{0}(\om)$ satisfies \eqref{asgn_equi_1}--\eqref{asgn_equi_2} then $v\in  \partial \|\cdot\|_{\cM,\alpha} (\mu) \cap C_{0}(\om)$. Then it just suffices to check that a function $v\in C_{0}(\om)$  satisfies \eqref{asgn_equi_1}--\eqref{asgn_equi_2} if and only if $v\in \alpha(x) \Sgn(\mu)\cap C_{0}(\om)$. The ``if'' implication is immediate from the definition of $\alpha(x) \Sgn(\mu)$. For the ``only if'' part, by considering the polar decomposition $\mu=\sgn(\mu)|\mu|$ we have
\[
\int_{\om}v(x)d\mu=\int_{\om}\alpha(x)d|\mu| \Rightarrow\\
\int_{\om}(v(x)\sgn(\mu)(x)-\alpha(x))d|\mu|=0,
\]
which, with the help of \eqref{asgn_equi_1}, implies that
\[v(x)\sgn(\mu)(x)=\alpha(x)\quad \text{ for }|\mu|\text{-almost every }x \;\Rightarrow\]
\[v(x)=\sgn(\mu)(x)\alpha(x)\quad \text{ for }|\mu|\text{-almost every }x.\]
Thus, $v\in \alpha(x)\Sgn(\mu)\cap C_{0}(\om)$ and the proof is complete.
\end{proof}

We define now the predual problem of \eqref{weighted_rof}:
\begin{equation}\label{predual}
-\min \left \{\frac{1}{2} \int_{\om} (v')^{2}dx +\int_{\om}fv' dx:\; v\in H_{0}^{1}(\om),\;|v(x)|\le \alpha(x), \text{ for every } x\in \om \right \}.
\end{equation}
The fact that the minimum in \eqref{predual} is attained by a unique $H_{0}^{1}$ function, can be shown  easily using standard techniques. In order to be convinced that \eqref{predual} is indeed the predual of \eqref{weighted_rof} define
\begin{align*}
\Lambda&: H_{0}^{1}(\om)\to L^{2}(\om)\quad \text{with}\quad \Lambda(v)=v',\\
G&: L^{2}(\om)\to\RR \quad \text{with}\quad G(\psi)=\frac{1}{2}\int_{\om} \psi^{2}dx+\int_{\om} f\psi\, dx,\\
F&: H_{0}^{1}(\om)\to \overline{\RR} \quad \text{with}\quad F(v)=\mathcal{I}_{\{|\cdot (x)|\le \alpha (x),\,\forall x\in\om\}}.
\end{align*}
Then it is easy to verify that the problem \eqref{predual} is equivalent to 
\begin{equation}\label{predual_FG}
-\min_{v\in H_{0}^{1}(\om)} F(v)+G(\Lambda v).
\end{equation}
Now the dual problem of \eqref{predual_FG} is defined as \cite{ekeland1976convex} 
\begin{equation}\label{dual_FG}
\min_{u\in L^{2}(\om)^{\ast}}  F^{\ast} (-\Lambda ^{\ast} u) +G^{\ast}(u).
\end{equation}
After a few computations the problem \eqref{dual_FG} can be shown to be equivalent with our main problem \eqref{weighted_rof}. The proof follows closely the analogue proofs in \cite{ring2000structural}, \cite{BrediesL1} and \cite{Papafitsoros_Bredies} for the corresponding $L^{2}$--$\tv$ (scalar case), $L^{1}$--$\tgv$ and $L^{2}$--$\tgv$ minimisations and thus we omit it. We note here that the derivation of the predual problem of \eqref{weighted_rof} in higher dimensions is  more involved, see \cite{hint_rau}.
The solutions of the problems \eqref{predual_FG} and \eqref{dual_FG} are linked through the optimality conditions:
\begin{align*}
v&\in \partial F^{\ast} (-\Lambda^{\ast} u),\\
\Lambda v&\in \partial G^{\ast}(u),
\end{align*}
which, after a few calculations, can be reformulated as
\begin{align*}
v'&=f-u,\\
-v&\in \alpha(x) \Sgn(Du).
\end{align*}
Summarising, the following proposition holds.
\newtheorem{optimality}[alpha_sgn]{Proposition}
\begin{optimality}[Optimality conditions for weighted $\tv$ minimisation]\label{lbl:optimality}
Let $\om=(a,b)$, $f\in\bv(\om)$ and $\alpha\in C(\overline{\om})$ with $\alpha>0$. A function $u\in\bv(\om)$ is the solution to the minimisation problem
\[\min_{u\in\bv(\om)} \frac{1}{2}\int_{\om}(f-u)^{2}dx+\int_{\om}\alpha(x)d|Du|,\]
if and only if there exists a function $v\in H_{0}^{1}(\om)$ such that 
\begin{align}
v'&=f-u,\label{opt1}\\
-v&\in \alpha(x) \Sgn(Du).\label{opt2}
\end{align}
\end{optimality}

Observe here that Proposition \ref{lbl:optimality} still holds when $f\in L^{2}(\om)$. This is useful in the context of image denoising, where  $f$ is a noisy, perhaps strongly oscillating function, modelled as an element outside $\bv(\om)$.  Here, in contrast, we assume that $f\in\bv(\om)$, since in this study, we are more interested in the structural properties of weighted $\tv$ minimisation than addressing the entire reconstruction problem. Observe that since we are in dimension one, this also implies that we have more than $H_{0}^{1}$ regularity for the function $v$. Indeed, $v'\in\bv(\om) \subseteq L^{\infty}(\om)$ and in particular $v$ is a Lipschitz function.

\subsection{Structure of solutions -- creation of new discontinuities}\label{sec:discontinuities}

One can already notice a basic difference between the scalar and the weighted total variation regularisation. Indeed, when $\alpha(x)=\alpha\in\RR$ for every $x\in\om$, the optimality conditions \eqref{opt1}--\eqref{opt2} imply that when $\overline{f}<\underline{u}$ (or $\underline{f}>\overline{u}$) then $Du=0$ there. That is to say, the solution $u$ is constant in the areas where it is not equal to the data $f$, a well-known characteristic of total variation minimisation \cite{ring2000structural}. In the weighted case, however, the optimality conditions \eqref{opt1}--\eqref{opt2} do not enforce such a behaviour. In this section, using a series of propositions and numerical examples we highlight the differences between the scalar and the weighted case as far as the structure of solutions is concerned. Particular emphasis is given on the discontinuities of the solution $u$. Recall here that one of the few analytical results concerning the weighted $\tv$ regularisation is that of Jalalzai \cite{jalalzai2014discontinuities}. There, the author shows that given $\om\subseteq \RR^{d}$ open, bounded with Lipschitz boundary, data $f\in \bv(\om)\cap L^{\infty}(\om)$, and a bounded, Lipschitz continuous weight function $\alpha$ with the extra property that $\nabla \alpha \in \bv(\om)$, then 
\begin{equation}\label{jump_jal}
J_{u}\subseteq J_{f}\cup J_{\nabla \alpha},
\end{equation}
up to $\mathcal{H}^{d-1}$ negligible set. Here $\mathcal{H}^{d-1}$ denotes the $(d-1)$-dimensional Hausdorff measure. This result shows that new jump discontinuities can potentially appear in the solution $u$ at points where the derivative of the weight function also has a jump. This is in strong contrast to the scalar $\tv$ minimisation where  the discontinuities of the solution can only occur in points where the data $f$ is discontinuous \cite{caselles2007discontinuity, valkonen_jump_1}. Note that this also true in the weighted case when $\alpha\in C^{1}(\om)$ since then $J_{\nabla \alpha}=\emptyset$. 

Here we investigate in detail, the creation of new discontinuities in the one dimensional regime. We will show with analytical and numerical results that at least in dimension one, the inclusion \eqref{jump_jal} is sharp. In order to develop an intuition for this phenomenon, we start with a simple proof of \eqref{jump_jal} in the one dimensional case. Note that we do not assume here that $\alpha$ is Lipschitz continuous with $\alpha'\in \bv(\om)$. 

\newtheorem{jump_set_incl}[alpha_sgn]{Proposition}
\begin{jump_set_incl}\label{lbl:jump_set_incl}
Let $u\in\mathrm{BV}(\om)$ be a solution to \eqref{weighted_rof} and let $x\in\Omega$ such that $\alpha$ is differentiable at $x$ and $|Df|(\{x\})=0$, i.e., $x\notin J_{f}$. Then $x\notin J_{u}$.
\end{jump_set_incl}

\begin{proof}
Suppose, towards contradiction, that $x\in J_{u}$, i.e., $|Du|(\{x\})>0$. Without loss of generality we  assume that $Du(\{x\})>0$ since the case $Du(\{x\})<0$ is treated analogously. Hence, we have
\begin{equation}\label{left_bigger_right}
u^{l}(x)<u^{r}(x).
\end{equation}
Since $|Df|(\{x\})=0$ we have that any good representative $\tilde{f}$ of $f$ is continuous at $x$. Using \eqref{left_bigger_right}, the continuity of $\tilde{f}$, the left and right continuity of $u^{l}(x)$ and $u^{r}(x)$, respectively, we have that there exist a small enough $\epsilon>0$ and two constants $m<M$ such that
\begin{equation}\label{sup_less_inf}
\sup_{t\in (x,x+\epsilon)} \tilde{f}(t)-u^{r}(t)\le m<M \le \inf_{t\in (x-\epsilon,x)} \tilde{f}(t)-u^{l}(t).
\end{equation}
With the help of \eqref{opt1}, the above inequalities are translated into
\begin{equation}\label{esssup_less_essinf}
\esssup_{t\in (x,x+\epsilon) } v'(t)\le m<M \le \essinf_{t\in (x-\epsilon,x)} v'(t).
\end{equation}
Since $Du(\{x\})>0$, condition \eqref{opt2} dictates that 
\[v(x)=-\alpha(x).\]
Using now the fundamental theorem of calculus along with \eqref{esssup_less_essinf} we get that for every $t\in (x,x+\epsilon)$
\begin{align*}
v(t)&= -\alpha(x)+\int_{x}^{t}v'(t)dt\\
     &\le -\alpha(x)+m(t-x),
\end{align*}
and for every $t\in (x-\epsilon,x)$
\begin{align*}
v(t)&= -\alpha(x)+\int_{t}^{x}-v'(t)dt\\
     &\le  -\alpha(x) +M(t-x).
\end{align*}
Using the fact that $-\alpha(t)\le v(t)$ for every  $t\in\om$ and condition \eqref{opt2}, we further calculate
\begin{equation}\label{left_limit}
\lim_{t\to x-} \frac{\alpha(x)-\alpha(t)}{x-t}\le \frac{\alpha(x)+v(t)}{x-t}\le \frac{M(t-x)}{x-t}= -M
\end{equation}
and 
\begin{equation}\label{right_limit}
\lim_{t\to x+} \frac{\alpha(t)-\alpha(x)}{t-x}\ge \frac{-v(t)-\alpha(x)}{t-x}\ge \frac{-m(t-x)}{t-x}=-m.
\end{equation}
The inequalities \eqref{left_limit}--\eqref{right_limit} contradict the differentiability of $\alpha$ at $x$ and thus the proof is complete.
\end{proof}

Even though it is now clear that non-differentiablity of $\alpha$ can potentially lead to the creation of new discontinuities, as the next proposition shows this is not always the case. In particular, we show in what follows that if $\alpha$ has an upward spike at a point $x$, then the solution $u$ of \eqref{weighted_rof} is constant in a neighbourhood of $x$; see Figure \ref{fig:spikeup} for an illustration.

\newtheorem{alpha_spike_up}[alpha_sgn]{Proposition}
\begin{alpha_spike_up}\label{lbl:alpha_spike_up}
Let  $\alpha\in C(\overline{\om})$ and $x\in\om$ such that $\alpha\in C(\overline{\om})$ is differentiable in a neighbourhood of $x$ (but not at $x$) with
\[\lim_{t\to x-} \alpha'(t)=+\infty\quad \text{and}\quad \lim_{t\to x+}\alpha'(t)=-\infty.\]
Then, if $u$ is the solution of \eqref{weighted_rof} with weight function $\alpha$ and some given data $f\in\bv(\om)$, then there exists an $\epsilon>0$ such that $|Du|((x-\epsilon,x+\epsilon))=0$, i.e., $u$ is constant in $(x-\epsilon,x+\epsilon)$.
\end{alpha_spike_up}

\begin{proof}
We show first that there exists an $\epsilon>0$ such that $|Du|((x,x+\epsilon))=0$. Indeed otherwise, using the condition \eqref{opt2}, we can assume without loss of generality, that there exists a decreasing sequence $(t_{n})_{n\in \NN}$ such that $t_{n}\downarrow x$ with $x<t_{n}$ and 
\[v(t_{n})=-\alpha(t_{n}),\quad \text{for every } n\in\NN.\]
But then, using the mean value theorem, we have for some $t_{n+1}<\xi_{n}<t_{n}$
\begin{align*}
\frac{|v(t_{n})-v(t_{n+1})|}{|t_{n}-t_{n+1}|}&=\frac{|\alpha(t_{n})-\alpha(t_{n+1})|}{|t_{n}-t_{n+1}|}\\
									    &=|\alpha'(\xi_{n})|.
\end{align*}
Since $|\alpha'(\xi_{n})|\to \infty$, the equality above implies that $v$ is not Lipschitz, a contradiction. Similarly we get $|Du|((x-\epsilon,x))=0$ for a small enough $\epsilon>0$. Finally notice that it also holds that $|Du|(\{x\})=0$. Otherwise, again from condition \eqref{opt2}, we would have that $v(x)=-\alpha(x)$ (or $v(x)=\alpha(x)$, with a similar proof) and using also the fact that $v\ge -\alpha$, we have for $t>x$
\begin{align}
\frac{v(t)-v(x)}{t-x}&\ge \frac{-\alpha(t)+\alpha(x)}{t-x}\to +\infty \quad \text{as }t\to x+,
\end{align}
again contradicting the fact that $v$ is Lipschitz. Hence, for a small enough $\epsilon>0$ we have
\[|Du|((x-\epsilon,x+\epsilon))=|Du|((x-\epsilon,x))+|Du|(\{x\})+|Du|((x,x+\epsilon))=0.\]
\end{proof}

Observe that it is not essential to assume that  $\alpha$ is differentiable at a set of the type $(x-\delta,x)\cup (x,x+\delta)$ for small enough $\delta>0$. For example it would be enough to assume that $\alpha$ is concave at each of the intervals $(x-\delta,x)$ and $(x,x+\delta)$ and its graph does not satisfy the cone property at $x$. 

\begin{figure}[t]
\centering
\includegraphics[width=0.6\textwidth]{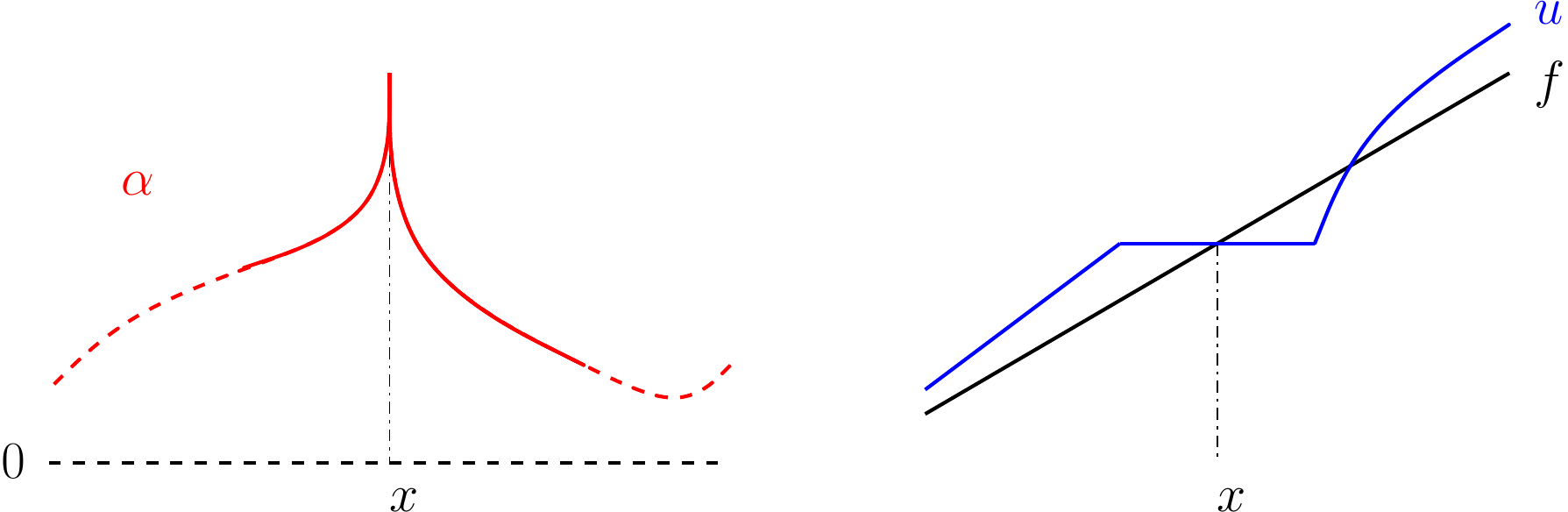}
\caption{Illustration of Proposition \ref{lbl:alpha_spike_up}: When the weight function $\alpha$ has an upward spike at a point $x$ (left plot) then the solution $u$ of \eqref{weighted_rof} is constant at an neighbourhood of $x$ (right plot).}
\label{fig:spikeup}
\end{figure}

After examining the case where $\alpha$ has an upward spike, it is natural  to ask what happens if $\alpha$ exhibits a downward spike. The following proposition provides some intuition.

\newtheorem{alpha_spike_down}[alpha_sgn]{Proposition}
\begin{alpha_spike_down}\label{lbl:alpha_spike_down}
Let $f\in\bv(\om)$ such that $f$ is continuous and strictly increasing. Suppose that $\alpha\in C(\overline{\om})$
is differentiable everywhere in $\om$ apart from a point $x$ and
\[\lim_{t\to x-} \alpha'(t)=-\infty\quad \text{and}\quad \lim_{t\to x+}\alpha'(t)=+\infty\]
with $\alpha$ attaining its minimum at $x$.  Then, if $u$ is the solution of \eqref{weighted_rof} for the weight function $\alpha$ and data $f$,  it has either a jump discontinuity at $x$ or it is constant up to the boundary of $\om$.
\end{alpha_spike_down}
\begin{proof}
Similarly to the proof of Proposition \ref{lbl:alpha_spike_up} we can deduce that there exists an $\epsilon>0$ such that $|Du|((x-\epsilon,x)\cup(x,x+\epsilon))=0$, i.e., $u$ will be constant in each of the intervals $(x-\epsilon,x)$, $(x,x+\epsilon)$. Suppose now that $u$ does not have a jump discontinuity at $x$, i.e., $Du(\{x\})=0$, and thus $u$ is constant in $(x-\epsilon, x+\epsilon)$, say equal to $c$. \\[3pt]
\emph{Case 1}:  $u(x)<f(x)$.\\
In this case we claim that $u$ is constant, equal to $c$, in $[x-\epsilon,b)$. Suppose this is not true. Then note first that since $f$ is strictly increasing, it is easily checked that $u$ will be increasing as well. Recall from Proposition \ref{lbl:jump_set_incl}, that $u$ will be continuous on $(x,b)$ since $\alpha$ is differentiable there. Now choose  $t_{0}\in [x+\epsilon,b)$ such that
\[d:=Du(x,t_{0})\le \frac{f(x)-u(x)}{2},\]
with $d$ being strictly positive. Notice that this can be done since $u$ is increasing in $[x+\epsilon,b)$ and not just equal to a constant.
 Define $\tilde{u}$ to be the following function:
\[
\tilde{u}(t)=
\begin{cases}
u(t),\;\;&t\in (a,x),\\
u(x)+d, & t\in [x,t_{0}),\\
u(t), & t\in [t_{0},b),
\end{cases}
\]
see also Figure \ref{fig:tilde_u} for an illustration.
In other words, $\tilde{u}$ has all the variation of $u$ in $(x+\epsilon,t_{0})$ concentrated in $x$. Note that 
\[\int_{\om}(f-\tilde{u})^{2}dx<\int_{\om}(f-u)^{2}dx,\]
and since $\alpha$ has a minimum at $x$ we also have 
\[\int_{\om}\alpha (x)d|D\tilde{u}|\le\int_{\om}\alpha (x)d|Du|,\]
hence $u$ is not optimal which is a contradiction.
\begin{figure}
\centering
\includegraphics[width=0.7\textwidth]{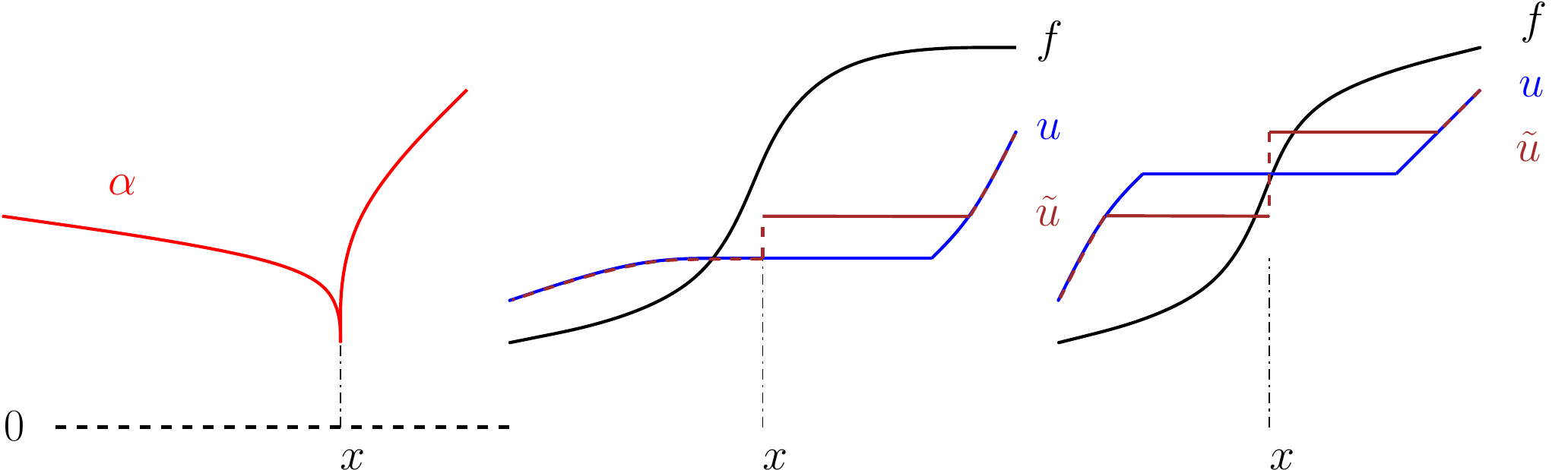}
\caption{The function $\tilde{u}$ from the proof of Proposition \ref{lbl:alpha_spike_down}. Shifting the variation from areas which is costly into a single point where it is less costly. This favours the creation of a new discontinuity point.}
\label{fig:tilde_u}
\end{figure}\\[3pt]
\emph{Case 2}:  $u(x)>f(x)$.\\
This case is treated similarly to \emph{Case 1}. If $u$ does not have a jump discontinuity at $x$, then by similar arguments we  conclude that $u$ will be constant on an interval of the type $(a,x+\epsilon)$.\\[3pt]
\emph{Case 3}:  $u(x)=f(x)$.\\
The arguments are similar to the previous cases; see also the third graph in Figure \ref{fig:tilde_u}. We just have to make sure that by choosing a small enough jump at $x$ for the function $\tilde{u}$ we can achieve a better $L^{2}$ distance from $f$. This can be done, for instance, by choosing $d$ smaller than $f(x+\frac{\epsilon}{3})-u(x+\frac{\epsilon}{3})$.

\end{proof}

\newtheorem{remark_Du_pos}[alpha_sgn]{Remark}
\begin{remark_Du_pos}
Note that for the type of data (increasing) of the Proposition \ref{lbl:alpha_spike_down}  the potential jump discontinuity  at $x$ can only be positive, i.e., $Du(\{x\})>0$. Indeed, it can be easily checked that if $Du(\{x\})<0$ then the function $u$ would be not optimal. 
 \end{remark_Du_pos}
 
Summarising the findings so far, we can say  that whenever the weight function $\alpha$ has a \emph{spike} at a point $x$, no matter whether this \emph{spike} is upward or downward,  the solution will always be constant at  each one of the intervals $(x-\epsilon,x)$ and $(x,x+\epsilon)$ for a small enough $\epsilon>0$. If the spike is upward, then the solution $u$ will  be constant in the whole interval $(x-\epsilon,x+\epsilon)$. If the spike is downward then the solution $u$ will be either constant in $(x-\epsilon,x+\epsilon)$ or piecewise constant with a jump discontinuity at $x$. 
In order to be convinced that the second alternative can indeed occur, think of the following corollary of Proposition \ref{lbl:alpha_spike_down}. Suppose that $f$ is a strictly increasing, continuous function with a graph which is symmetric with respect to $(\frac{b-a}{2},f(\frac{b-a}{2}))$ and $\alpha$ is a similarly symmetric function with a downward spike at $\frac{b-a}{2}$, e.g., 
$\alpha(x)=\sqrt{\left |x-\frac{b-a}{2} \right |}$.
 Then unless $u$ is a constant function, it will always have a jump discontinuity at $x=\frac{b-a}{2}$.  

In fact new discontinuities can be created even with more regular weight function, i.e., when $\alpha'\in\bv(\om)$. While we will come back to this with specific examples in Section \ref{sec:exact}, the following proposition provides conditions on when this can indeed occur and it establishes a connection between the jump size of $\alpha'$ and the jump size of $u$. Note that for such a function $\alpha$ we have
\begin{equation}\label{alpha_BV2}
\{x\in\om:\; \alpha \text{ is not differentiable at }x\}=J_{\alpha'}.
\end{equation}

\newtheorem{lipschitz_alpha}[alpha_sgn]{Proposition}
\begin{lipschitz_alpha}\label{lbl:lipschitz_alpha}
Let $f\in\bv(\om)$ with $f$ being continuous at a point $x\in\Omega$. Let $\alpha\in C(\overline{\om})$ be a weight function with $\alpha'\in\bv(\om)$ such that $|D\alpha'|(\{x\})>0$. Let $u$ solve \eqref{weighted_rof} with data $f$ and weight function $\alpha$. Then the following hold true:
\begin{enumerate}
\item If $D\alpha'(\{x\})<0$, then $|Du|((x-\epsilon,x+\epsilon))=0$, for a small enough $\epsilon>0$.
\item If $D\alpha'(\{x\})>0$, then $u$ has potentially a jump discontinuity at $x$ with
\begin{equation}\label{jump_u_less_jump_agrad}
|Du|(\{x\})\le D\alpha'(\{x\}).
\end{equation}
 In the particular case where there exists an $\epsilon>0$ such that $(x-\epsilon,x+\epsilon)\subseteq \mathrm{supp}(|Du|)$ then $u$ has a jump discontinuity at $x$ and 
\begin{equation}\label{jump_u_equal_jump_agrad}
|Du|(\{x\})= D\alpha'(\{x\}).
\end{equation}
\end{enumerate}
\end{lipschitz_alpha}

\begin{proof}
$(i)$ We start with the first case. We show first that $|Du|(\{x\})=0$. Suppose towards contradiction that $|Du|(\{x\})>0$  and assume without loss of generality that $Du(\{x\})>0$. We claim that $Dv'(\{x\})>0$.
Indeed we have 
\begin{align*}
v(t)&=v(x)+\int_{x}^{t}v'(s)ds,\quad -\alpha(t)=-\alpha(x)-\int_{x}^{t}\alpha'(s)ds,\quad \text{for all }x\le t,\\
v(t)&=v(x)-\int_{t}^{x}v'(s)ds,\quad -\alpha(t)=-\alpha(x)+\int_{t}^{x}\alpha'(s)ds,\quad \text{for all }t\le x.
\end{align*}
From condition \eqref{opt2} we have that $v(x)=-\alpha(x)$ and also $v(t)\ge -\alpha(t)$ for every $t\in\om$. Thus we can write 
\begin{align}
\int_{x}^{t}v'(s)ds&\ge -\int_{x}^{t}\alpha'(s)ds,\quad \text{for all }x\le t,\label{va1}\\
-\int_{t}^{x}v'(s)ds&\ge \int_{t}^{x}\alpha'(s)ds,\quad \text{for all }t\le x.\label{va2}
\end{align}
Since $-D\alpha'(\{x\})>0$ there exist a small enough $\epsilon>0$ and two constants $m<M$ such that
\begin{equation}\label{jump_minus_alpha}
\esssup_{s\in (x-\epsilon,x)} -\alpha'(s)<m<M<\essinf_{s\in(x,x+\epsilon)} -\alpha'(s).
\end{equation}
In combination with \eqref{va1}--\eqref{va2}, this implies that for all $\delta<\epsilon$
\begin{equation}\label{v_left_right}
\frac{1}{\delta}\int_{x-\delta}^{x}v'(s)ds<m<M<\frac{1}{\delta}\int_{x}^{x+\delta}v'(s)ds.
\end{equation}
Taking the limit in \eqref{v_left_right} as $\delta\to 0$ we get
\[(v')^{l}(x)\le m<M \le (v')^{r}(x).\]
This implies that $Dv'(\{x\})> 0$. However from condition \eqref{opt1} we have that
\[Df(\{x\})=Du(\{x\})+Dv'(\{x\})>0,\]
which contradicts the continuity of $f$ at $x$ and hence $|Du|(\{x\})=0$. We now claim that not only $|Du|(\{x\})=0$ but there exists a small enough $\epsilon>0$ such that $|Du|((x-\epsilon,x+\epsilon))=0$. If that was not the case, using condition \eqref{opt2}, we can find a sequence  $(t_{n})_{n\in\NN}$ with $t_{n}\to x$ such that $|v(t_{n})|=\alpha(t_{n})$. Without loss of generality, we can assume that $v(t_{n})=-\alpha(t_{n})$ for all $n\in\NN$. The proof is similar if we assume $v(t_{n})=\alpha(t_{n})$. 
From the continuity of $v$ and $\alpha$, this implies that $v(x)=-\alpha(x)$. Then by simply following again the steps above, we can derive again $Dv'(x)>0$ and 
\[Df(\{x\})=Dv'(\{x\})>0,\] 
which contradicts again the continuity of $f$ at $x$.\\
$(ii)$ Suppose now that $D\alpha' (\{x\})>0$. If $u$ does not have a jump discontinuity at $x$ then \eqref{jump_u_less_jump_agrad} holds trivially. Thus assume that $Du(\{x\})>0$. Working similarly to case $(i)$, we arrive again at \eqref{va1}--\eqref{va2}. Notice also that since we assumed that $Du(\{x\})>0$ we have from \eqref{opt1}
\[D(-v')(\{x\})=Du(\{x\})>0,\]
which means that $(-v')^{l}(x)<(-v')^{r}(x)$ and thus for all $\delta>0$ that are small enough we have
\begin{equation}\label{v_left_right_ii}
\frac{1}{\delta}\int_{x-\delta}^{x}-v'(s)ds<\frac{1}{\delta}\int_{x}^{x+\delta}-v'(s)ds.
\end{equation}
Inequality \eqref{v_left_right_ii} together with \eqref{va1}--\eqref{va2} gives
\[\frac{1}{\delta}\int_{x-\delta}^{x}\alpha'(s)ds\le\frac{1}{\delta} \int_{x-\delta}^{x}-v'(s)ds<\frac{1}{\delta}\int_{x}^{x+\delta}-v'(s)ds\le \frac{1}{\delta}\int_{x}^{x+\delta}\alpha'(s)ds,\]
for all $\delta>0$ small enough. Taking the limit $\delta\to 0$ in the expression above we end up with
\[(\alpha')^{l}(x)\le (-v')^{l}(x)<(-v')^{r}(x)\le (\alpha')^{r}(x),\]
and thus
\[Du(\{x\})=D(-v')(\{x\})\le D\alpha' (\{x\}).\]
Assuming $Du(\{x\})<0$, by working similarly we derive
\[Du(\{x\})\ge -D\alpha'(\{x\}),\]
and thus generally \eqref{jump_u_less_jump_agrad} holds.

For the second part of $(ii)$, note first that since $(x-\epsilon,x+\epsilon)\subseteq \mathrm{supp}(|Du|)$ and $v$ is continuous, from \eqref{opt2} it follows that there exists a sufficiently small $\delta>0$ such that $v=1$ everywhere in $(x-\delta,x+\delta)$ (or $v=-1$ everywhere in $(x-\delta,x+\delta)$). 
As a result, from condition \eqref{opt2}  we get that $v=-\alpha$ or $v=\alpha$ in  $(x-\delta,x+\delta)$ and condition \eqref{opt1} imposes there
\[-\alpha'=f-u\quad \text{or}\quad \alpha'=f-u.\]
Since $|Df|(\{x\})=0$ from the above we get that
\[Du(\{x\})=D\alpha'(\{x\})\quad \text{or}\quad  Du(\{x\})=-D\alpha'(\{x\}).\]
\end{proof}

By performing similar steps to the ones in the proof of Proposition \ref{lbl:lipschitz_alpha}, the following result can be shown.

\newtheorem{lipschitz_alpha_ii}[alpha_sgn]{Corollary}
\begin{lipschitz_alpha_ii}\label{lbl:lipschitz_alpha_ii}
Suppose that $x$ is a jump discontinuity point for the data $f$ and  $D\alpha'(\{x\})>0$. Then, the following estimate holds:
\begin{equation}\label{jump_u_less_jump_agrad_jump_f}
|Du|(\{x\})\le |Df|(\{x\})+D\alpha'(\{x\}).
\end{equation}
\end{lipschitz_alpha_ii}

\begin{proof}
We briefly sketch the proof. Suppose without loss of generality that $Df(\{x\})>0$. \\
$\bullet$ If $D(-v')(\{x\})>0$, then we follow the steps of the proof above starting from \eqref{v_left_right_ii} and we derive $D(-v')(\{x\})\le D\alpha'(x)$. Then from \eqref{opt1} we get
\[Du(\{x\})=Df(\{x\})+D(-v')(\{x\})\le Df(\{x\})+D\alpha'(\{x\}).\]
$\bullet$ If $-Df(\{x\})\le D(-v')(\{x\})<0$, then obviously
\[Du(\{x\})=Df(\{x\})+D(-v')(\{x\})<Df(\{x\})\le Df(\{x\})+D\alpha'(\{x\}).\]
$\bullet$ Lastly if $D(-v')(\{x\})<-Df(\{x\})$ then it follows that $Du(\{x\})<0$. Then following exactly the steps of $(ii)$ in the proof of Proposition \ref{lbl:lipschitz_alpha} (only the signs are reversed) we end up to 
\[Dv'(\{x\})\le D\alpha' (\{x\}),\]
and thus in this case
\[0>Du(\{x\})=Df(\{x\})+D(-v')(\{x\})\ge Df(\{x\})-D\alpha'(\{x\}).\]
\end{proof}

We would like now to prove that if $\alpha$ is  differentiable at a point $x$, then $|Du|(\{x\})\le |Df|(\{x\})$. Notice that we cannot derive this straightforwardly from Corollary \ref{lbl:lipschitz_alpha_ii} as there we use the fact that $D\alpha'(\{x\})>0$. However, this can easily be shown independently as the next proposition shows. 

\newtheorem{correct_jump}[alpha_sgn]{Proposition}
\begin{correct_jump}\label{lbl:correct_jump}
Let $u$ solve the weighted $\tv$ minimisation problem with data $f$ and weight function $\alpha\in C(\overline{\om})$ with $\alpha'\in\bv(\om)$ and  $\alpha>0$. 
Then  if $|D\alpha'|(\{x\})=0$, we have
\begin{equation}\label{eq:correct_jump}
|Du|(\{x\})\le |Df|(\{x\}).
\end{equation}
Moreover, the jumps of $u$ and $f$ have the same direction.
\end{correct_jump}

\begin{proof}
If $|Df|(\{x\})=0$ we have nothing to prove since by Proposition \ref{lbl:jump_set_incl} we have that $|Du|(\{x\})=0$ as well.
Thus,  suppose that $Df(\{x\})>0$. The case $Df(\{x\})<0$ is treated similarly. We first exclude the case $Du(\{x\})<0$. Suppose towards contradiction that this holds. From the left and right continuity properties of $f$ and $u$ and \eqref{opt1} we have that there exists an $\epsilon>0$ and some real numbers $m<M$ such that
\[\esssup_{t\in (x-\epsilon,x)}v'(t)\le m<M\le \essinf_{t\in(x,x+\epsilon)} v'(t), \]
Bearing in mind that $v(x)=\alpha(x)>0$ and the fact that
\begin{align*}
v(t)&=v(x)+\int_{x}^{t}v'(s)ds,\quad x\le t,\\
v(t)&=v(x)-\int_{t}^{x}v'(s)ds,\quad t\le x,
\end{align*}
together with $v<\alpha$, we deduce that 
\begin{align*}
\alpha(x)-\alpha(t)&\ge M(x-t),\quad x\le t,\\
\alpha(x)-\alpha(t)&\le m(x-t),\quad t\le x,
\end{align*}
which contradicts the fact that $\alpha$ is differentiable at $x$. Hence $Du(\{x\})>0$ and it now remains to prove \eqref{eq:correct_jump}. Notice first of all that by arguing similarly as above we can exclude the cases 
\[\underline{u}(x)<\underline{f}(x)<\overline{f}(x)<\overline{u}(x),\quad \underline{u}(x)\le\underline{f}(x)<\overline{f}(x)<\overline{u}(x)\quad \text{and} \quad \underline{u}(x)<\underline{f}(x)<\overline{f}(x)\le\overline{u}(x).\]
We thus focus on  the cases 
\begin{align}
&\underline{f}(x)<\underline{u}(x)\le \overline{f}(x)<\overline{u}(x),\label{cj_case1}\\
&\underline{f}(x)<\overline{f}(x)<\underline{u}(x)<\overline{u}(x),\label{cj_case2}\\
&\underline{u}(x)<\underline{f}(x)\le \overline{u}(x)< \overline{f}(x),\label{cj_case3}\\
&\underline{u}(x)<\overline{u}(x)<\underline{f}(x)<\overline{f}(x).\label{cj_case4}
\end{align}
and we will show that when these happen then \eqref{eq:correct_jump} must hold.

We argue for \eqref{cj_case1}  since \eqref{cj_case2}, \eqref{cj_case3} and \eqref{cj_case4} can be treated similarly.  Assume that \eqref{eq:correct_jump} does not hold. This means that 
\[\overline{f}(x)-\underline{f}(x)<\overline{u}(x)-\underline{u}(x).\]
Arguing in the same way as before, this implies that  there exists an $\epsilon>0$ and some real numbers $m<M$ such that
\[\esssup_{t\in (x,x+\epsilon)} v'(t)\le m<M\le \essinf_{t\in (x-\epsilon,x)} v'(t)<0.\]
This, together with the fact that $v(x)=-\alpha(x)$ and $v\le -\alpha$ contradicts again the differentiability of $\alpha$ at $x$.
\end{proof}

Recall that in the standard scalar $\tv$ minimisation we always have at a jump point $x$ of $u$,
\begin{equation}\label{eq:correct_jumpTV}
\underline{f}(x)\le \underline{u}(x)<\overline{u}(x)\le \overline{f}(x).
\end{equation}
Moreover, the jumps of $u$ and $f$ having the same directions, i.e.,
\[f^{l}(x)\le u^{l}(x)<u^{r}(x)\le f^{r}(x)\quad \text{or} \quad f^{r}(x)\le u^{r}(x)<u^{l}(x)\le f^{l}(x).\]

We now summarise our findings so far. Given $f\in\bv(\om)$ $\alpha\in C(\overline{\om})$ with $\alpha'\in\bv(\om)$,  we have shown analytically  the following:
\begin{enumerate}
\item If $|D\alpha'|(\{x\})=0$ and $|Df|(\{x\})=0$ then $|Du|(\{x\})=0$; see Proposition \ref{lbl:jump_set_incl} and \eqref{alpha_BV2}.
\item If $D\alpha' (\{x\})<0$ then a \emph{plateau} is created for $u$ around $x$; see Proposition \ref{lbl:lipschitz_alpha}.
\item The estimate $|Du|(\{x\})\le |Df|(\{x\})+|D\alpha'|(\{x\})$ holds  in every point $x\in\om$.
\item If $|Df|(\{x\})=0$, $D\alpha' (\{x\})>0$ and $(x-\epsilon,x+\epsilon)\subseteq \mathrm{supp}(|Du|)$, then $|Du|(\{x\})=D\alpha'(x)$; see Proposition \ref{lbl:lipschitz_alpha}.
\item If $f$ and $u$ jump at $x$ in different directions then $|Du|(\{x\})\le ||Df|(\{x\})-D\alpha'(\{x\})|$; see Corollary \ref{lbl:lipschitz_alpha_ii}.
\item If $|D\alpha'|(\{x\})=0$ and $u$ and $f$ jump at $x$, then their jumps have the same direction; see Proposition \ref{lbl:correct_jump}.
\end{enumerate}

Despite these first analytical results, several questions still need to be addressed. For instance, one wonders whether  $D\alpha'(\{x\})>0$, always creates a jump discontinuity for $u$ at $x$. Furthermore, we note that \eqref{eq:correct_jumpTV} is related to a loss of contrast in mathematical image processing. Here, one consequently is interested in understanding, whether such an effect still is possible in the weighted case, provided the weight is smooth.
Table \ref{question_table} summarises these and further questions. In Figures \ref{lbl:counterex_1}--\ref{lbl:counterex_6} we provide numerical examples for all the cases discussed in Table \ref{question_table}.

{\footnotesize
\setlength\extrarowheight{3pt}

\begin {table}
\begin{tabular}{| c | c | c | c |}
\hline
Case 			  				& What is proved analytically		& Is it possible...		& Answer/Figure\\[3pt]\hline\hline
$D\alpha'(\{x\})>0$   & 	$|Du|(\{x\})\le |Df|(\{x\})+|D\alpha'|(\{x\})$			& for $u$ to remain continuous? & Yes, Fig. \ref{lbl:counterex_1}		\\[3pt]\hline
$D\alpha'(\{x\})>0$ $\&$ $Df(\{x\})=0$  & 	$|Du|(\{x\})\le |D\alpha'|(\{x\})$	& to have ``$<$'' ?& Yes, Fig. \ref{lbl:counterex_2}		\\[3pt]\hline
$D\alpha'(\{x\})>0$ $\&$ $Df(\{x\})=0$  & 	$|Du|(\{x\})\le|D\alpha'|(\{x\})$			& to have ``$=$'' ?& Yes, Fig. \ref{lbl:counterex_3}		\\[3pt]\hline
$D\alpha'(\{x\})=0$ $\&$ $Df(\{x\})>0$  &   $|Du|(\{x\})\le |Df|(\{x\})$       & 	$f^{l}(x)<f^{r}(x)<u^{l}(x)<u^{r}(x)$ ?& Yes,  Fig. \ref{lbl:counterex_4} \\[3pt]\hline
$D\alpha'(\{x\})>0$ $\&$ $Df(\{x\})>0$  & $|Du|(\{x\})\le |Df|(\{x\})+|D\alpha'|(\{x\})$ & $u^{l}(x)< f^{l}(x)<f^{r}(x)<u^{r}(x)$ ? & Yes,  Fig. \ref{lbl:counterex_5}\\[3pt]\hline
$D\alpha'(\{x\})>0$ $\&$ $Df(\{x\})>0$  & $|Du|(\{x\})\le ||Df|(\{x\})-D\alpha'(\{x\})|$
\setcounter{footnote}{0}\footnotemark
 & $u^{l}(x)<f^{r}(x)<f^{l}(x)<u^{r}(x)$ ? & Yes, Fig. \ref{lbl:counterex_6} \\[3pt]\hline

\end{tabular}
\vspace{10pt}
\caption{Summary of the questions that are answered with numerical examples in Figures \ref{lbl:counterex_1}--\ref{lbl:counterex_6}}
\label{question_table}
\end{table}
}
\begin{figure}
	\begin{subfigure}[t]{0.48\textwidth}
		\begin{tabular}{@{}cc@{}}
			\includegraphics[width=.48\textwidth]{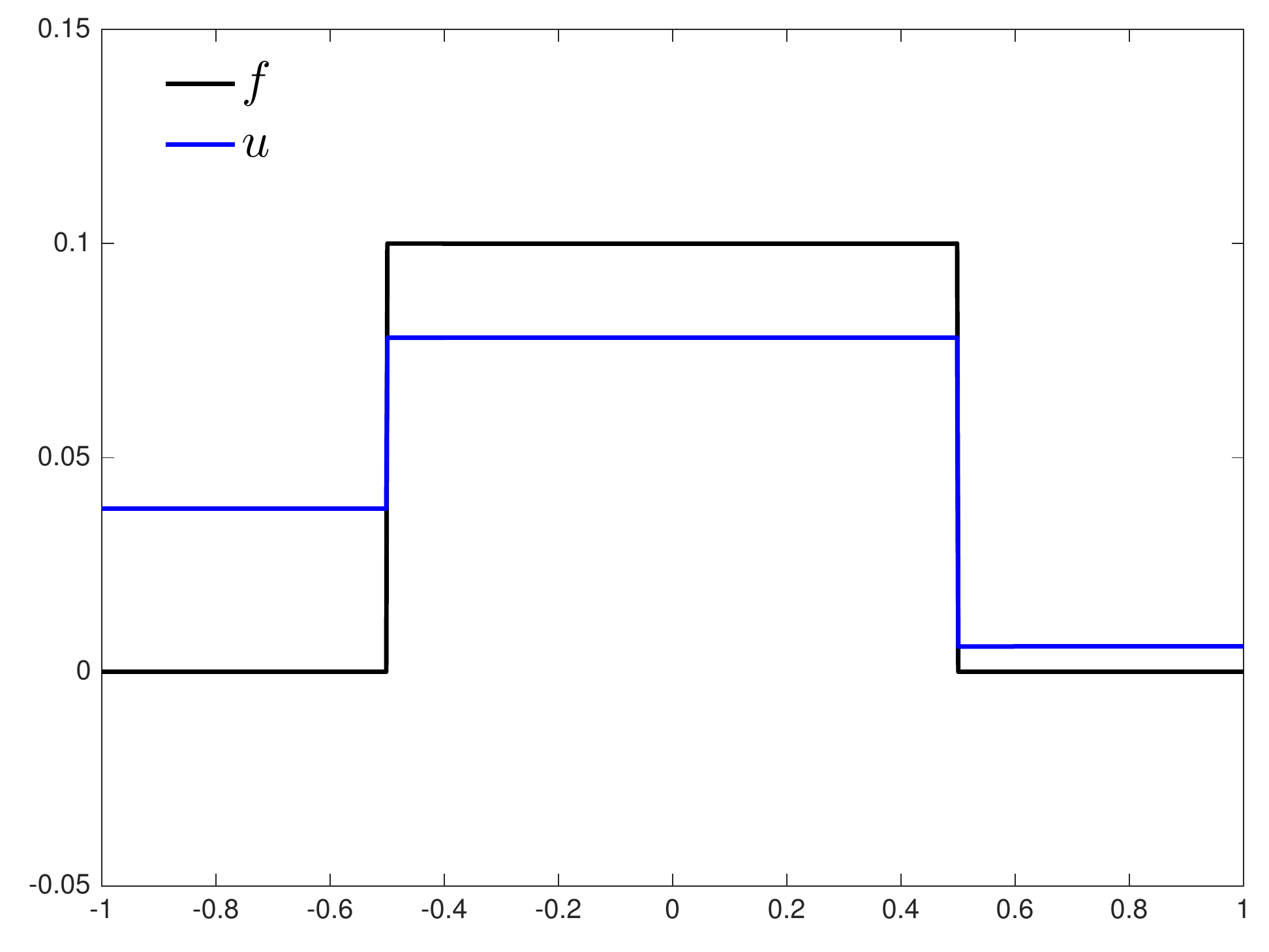} &
  			 \includegraphics[width=.48\textwidth]{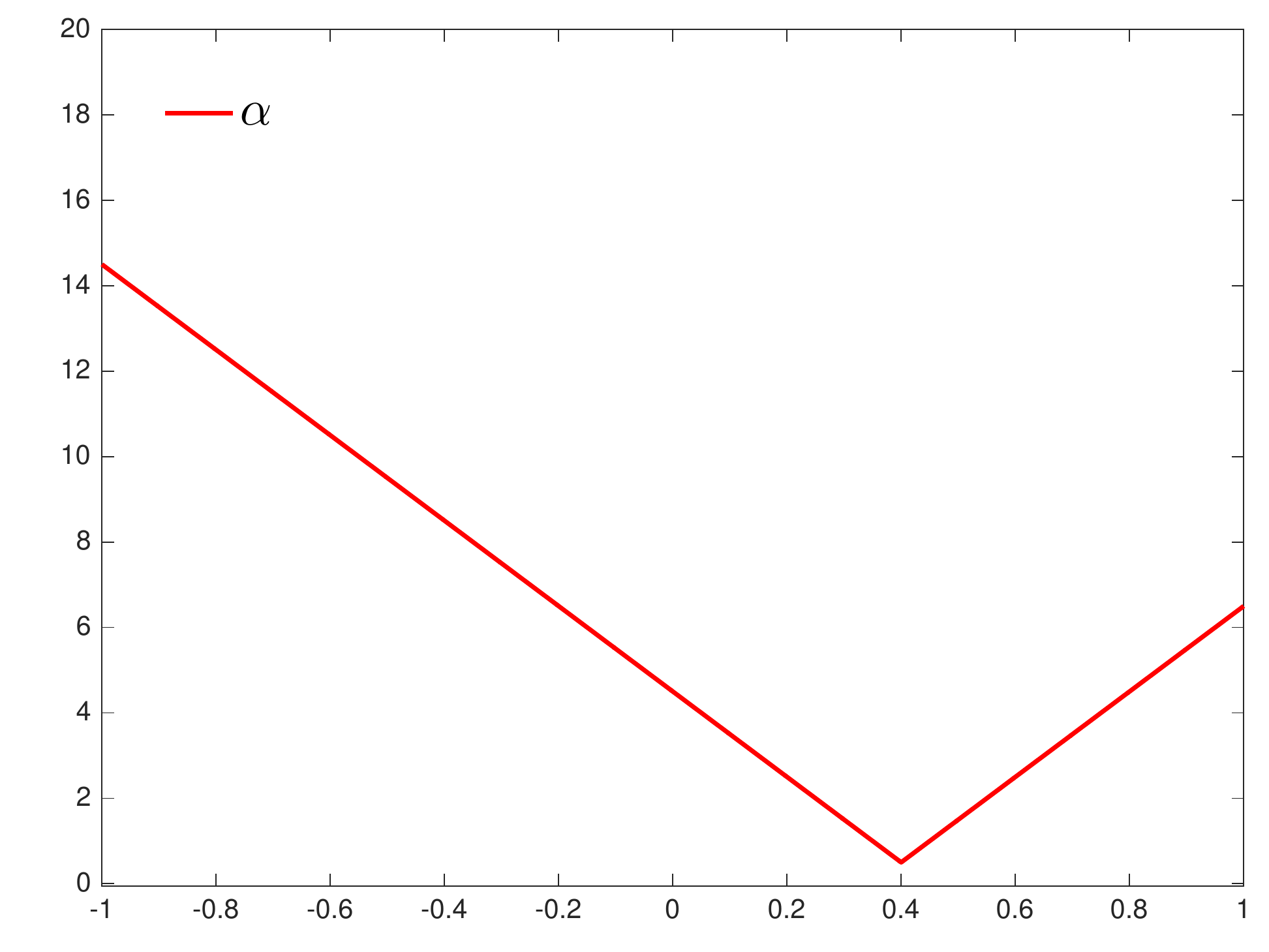} 
		\end{tabular}
		\captionsetup{width=.85\linewidth}
		\caption{Case where even though $\alpha'$ has a positive jump, the solution $u$ remains continuous at 			this point}
		\label{lbl:counterex_1}
	\end{subfigure}\hspace{0.5cm}
	\begin{subfigure}[t]{0.48\textwidth}
		\begin{tabular}{@{}cc@{}}
			\includegraphics[width=.48\textwidth]{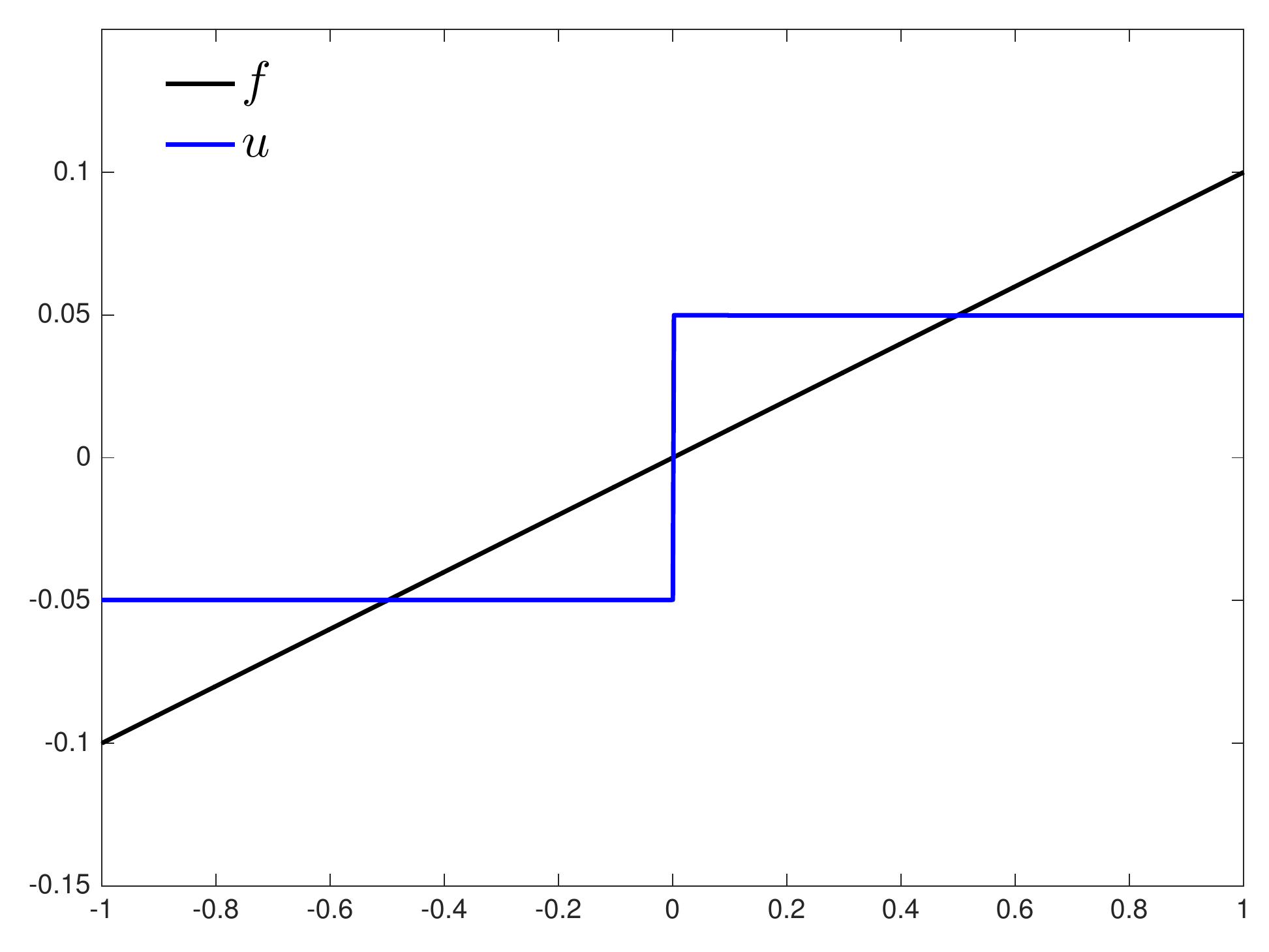} &
  			 \includegraphics[width=.48\textwidth]{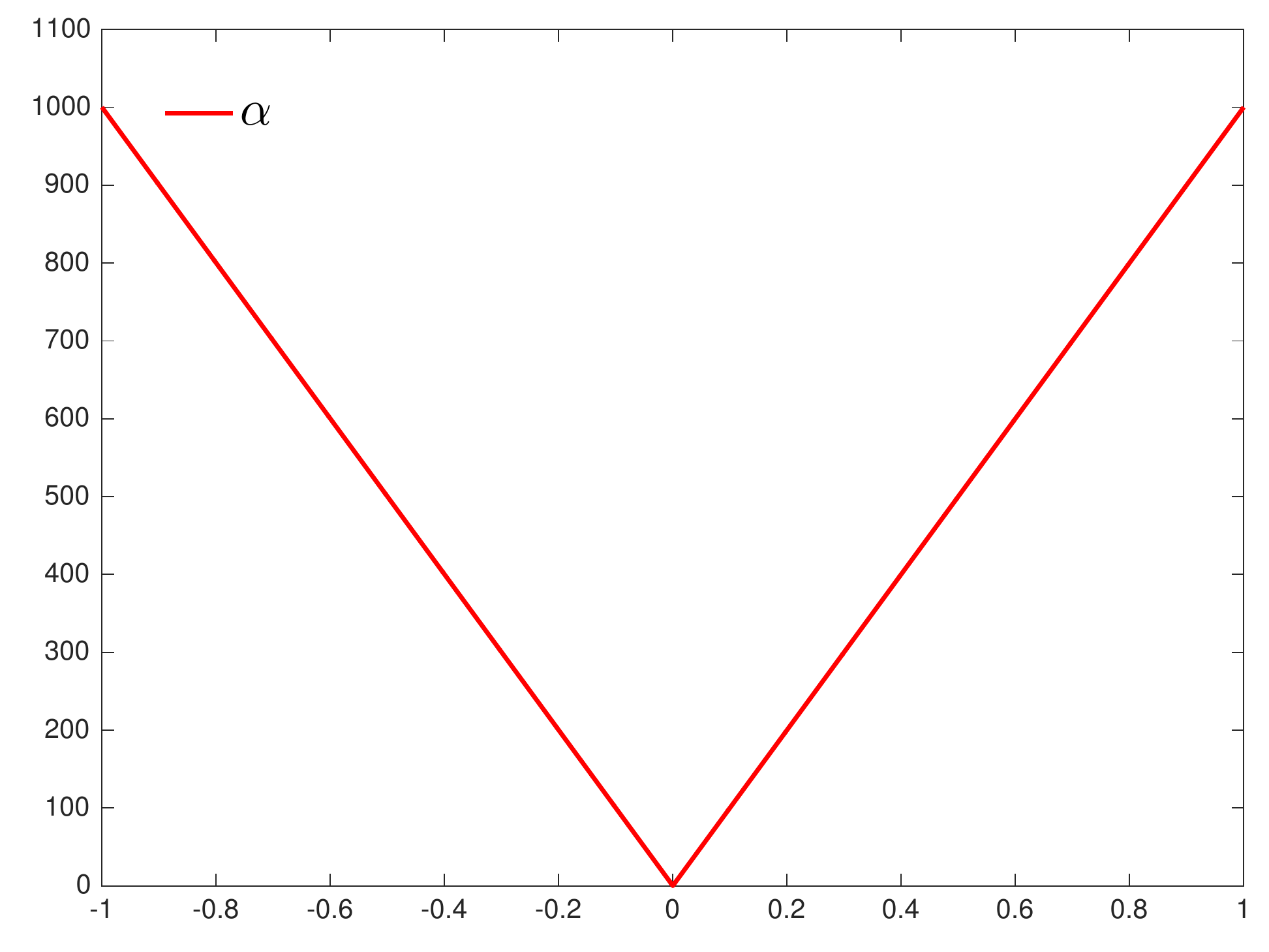} 
		\end{tabular}
		\captionsetup{width=.85\linewidth}
		\caption{Case where $|Du|(\{x\})<D\alpha'(\{x\})$. Note \\that $x=0$ is not an interior point of  $\mathrm{supp}(|Du|)$}
		\label{lbl:counterex_2}
	\end{subfigure}\vspace{0.2cm}
	
	\begin{subfigure}[t]{0.48\textwidth}
		\begin{tabular}{@{}cc@{}}
			\includegraphics[width=.48\textwidth]{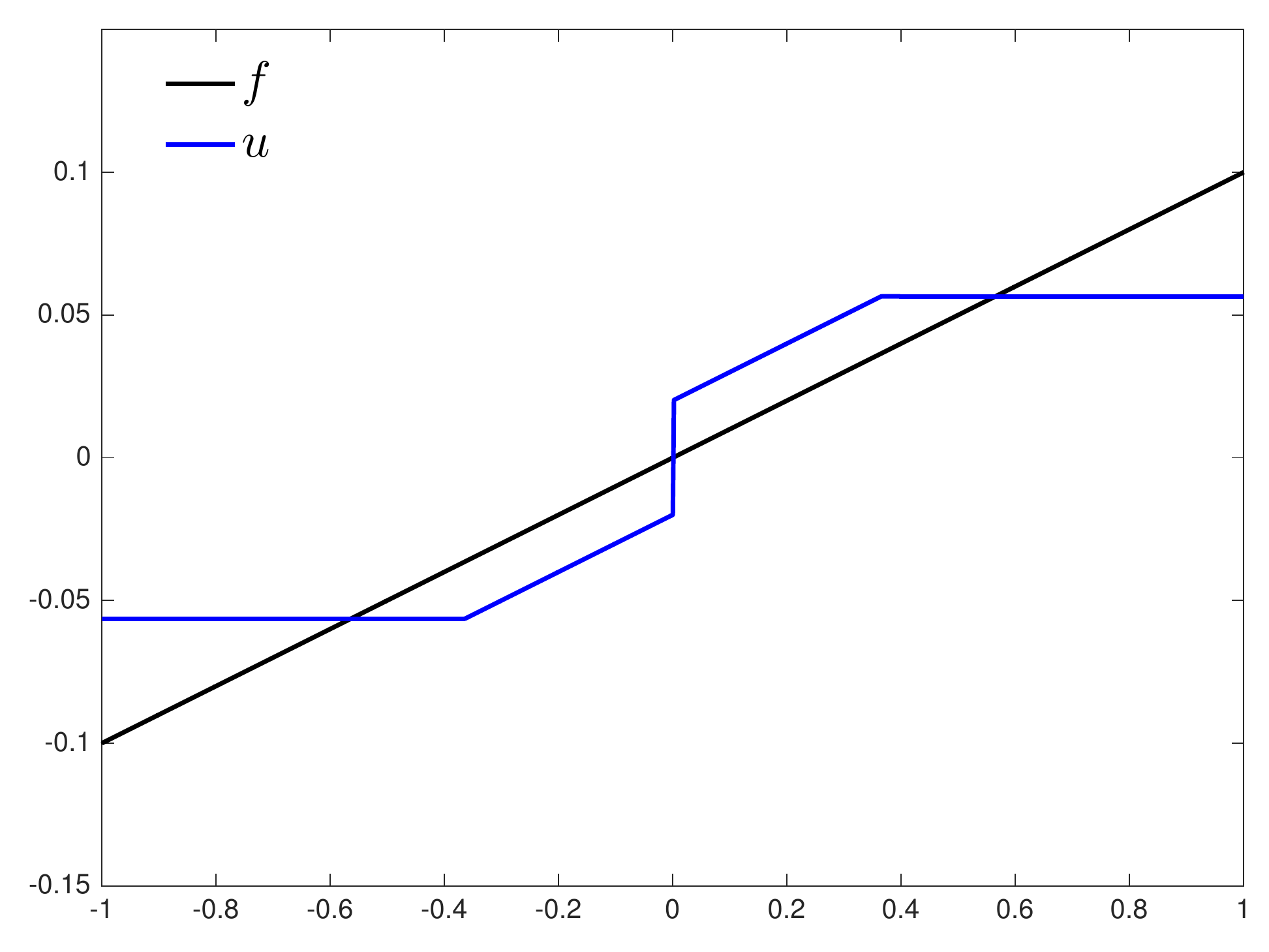} &
  			 \includegraphics[width=.48\textwidth]{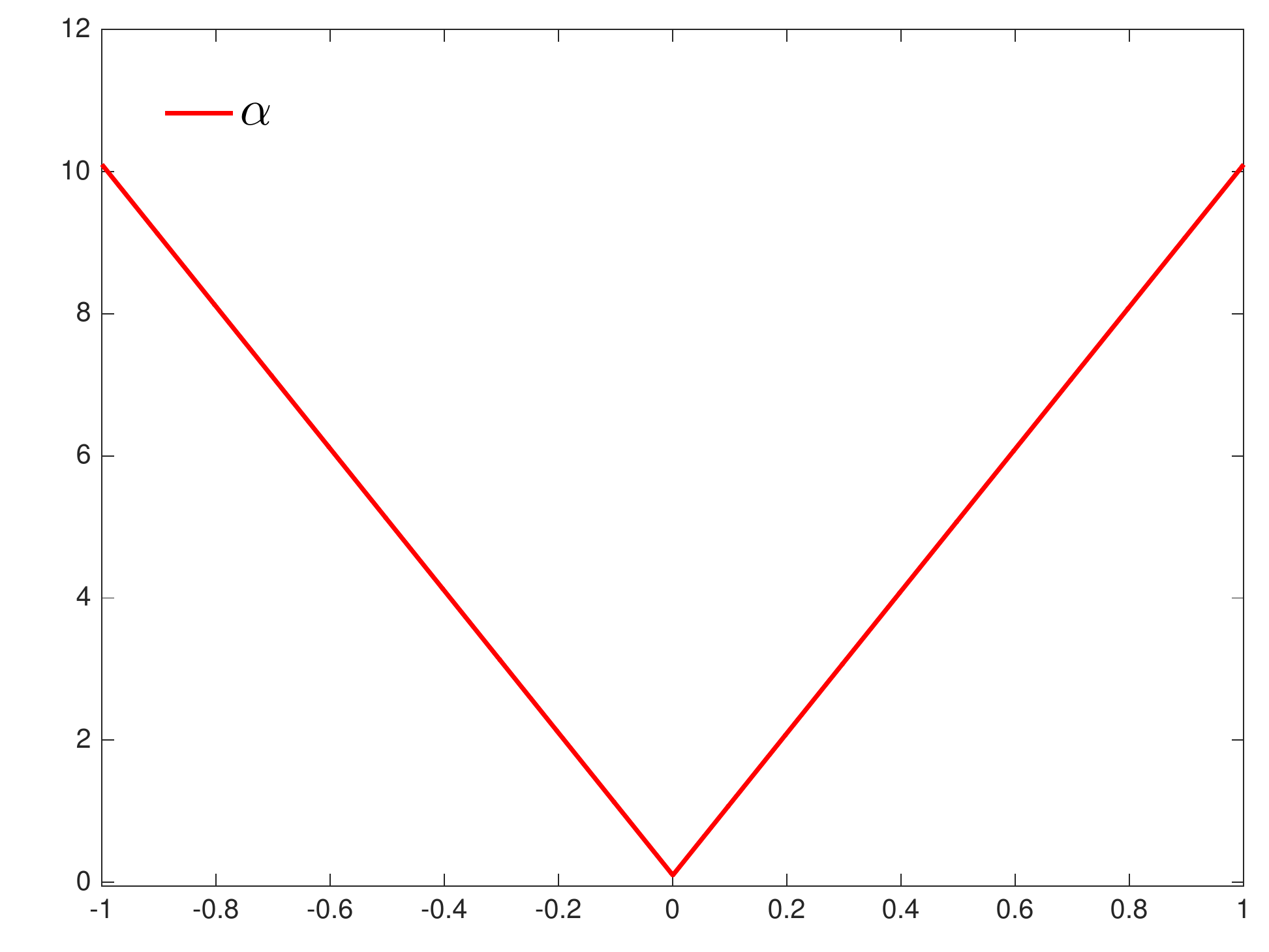} 
		\end{tabular}
		\captionsetup{width=.85\linewidth}
		\caption{Case where $|Du|(\{x\})=D\alpha'(\{x\})$. Note that this is predicted by Proposition 					\ref{lbl:lipschitz_alpha}, since $(-\epsilon,\epsilon)\subseteq\mathrm{supp}(|Du|)$}
		\label{lbl:counterex_3}
	\end{subfigure}\hspace{0.5cm}
	\begin{subfigure}[t]{0.48\textwidth}
		\begin{tabular}{@{}cc@{}}
			\includegraphics[width=.48\textwidth]{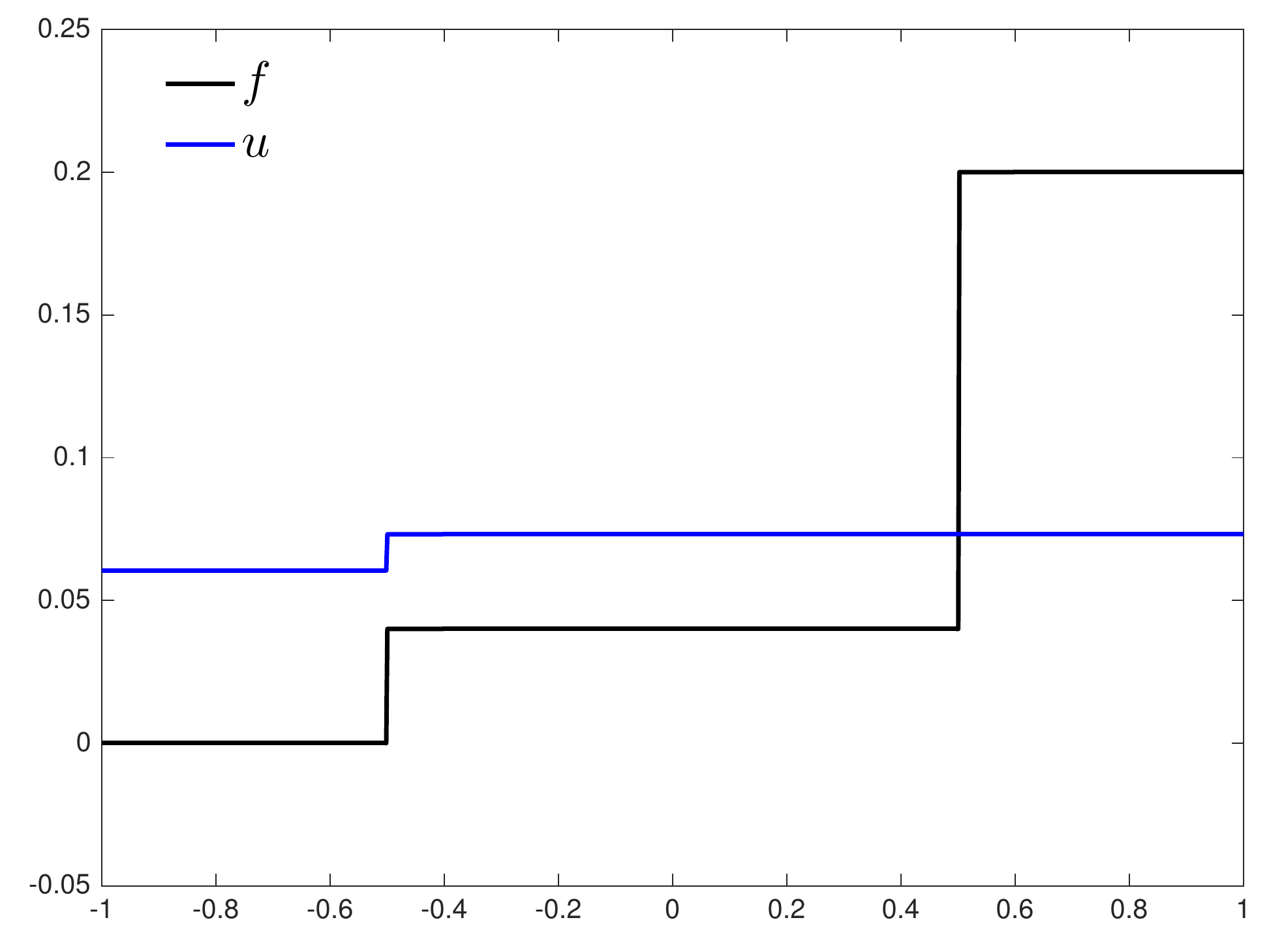} &
  			 \includegraphics[width=.48\textwidth]{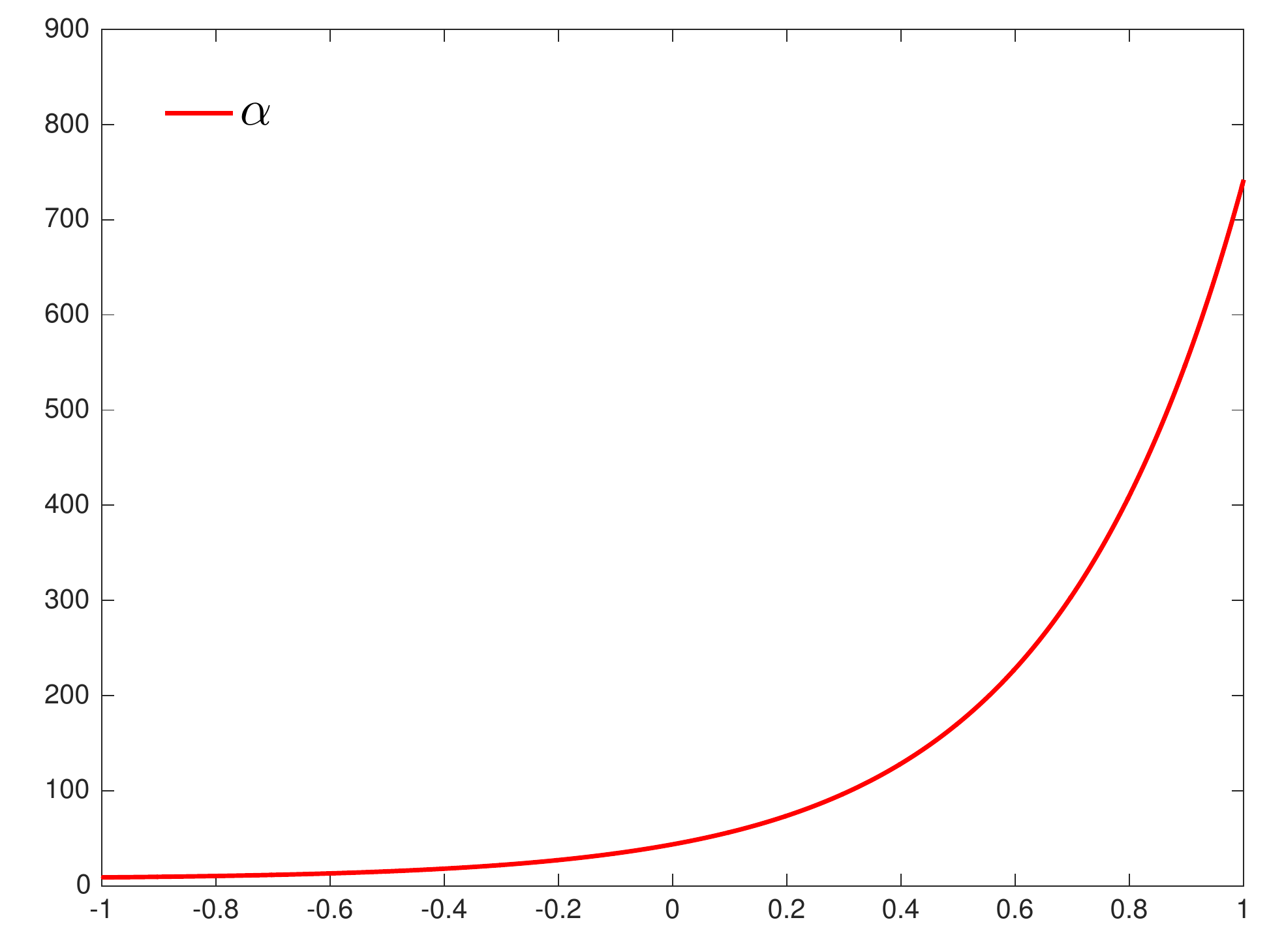} 
		\end{tabular}
		\captionsetup{width=.85\linewidth}
		\caption{Case where even though the weight function $\alpha$ is smooth, the jump of $u$ is above the 		jump of $f$, i.e., $f^{l}(x)<f^{r}(x)<u^{l}(x)<u^{r}(x)$}
		\label{lbl:counterex_4}
	\end{subfigure}\vspace{0.2cm}
	
	\begin{subfigure}[t]{0.48\textwidth}
		\begin{tabular}{@{}cc@{}}
			\includegraphics[width=.48\textwidth]{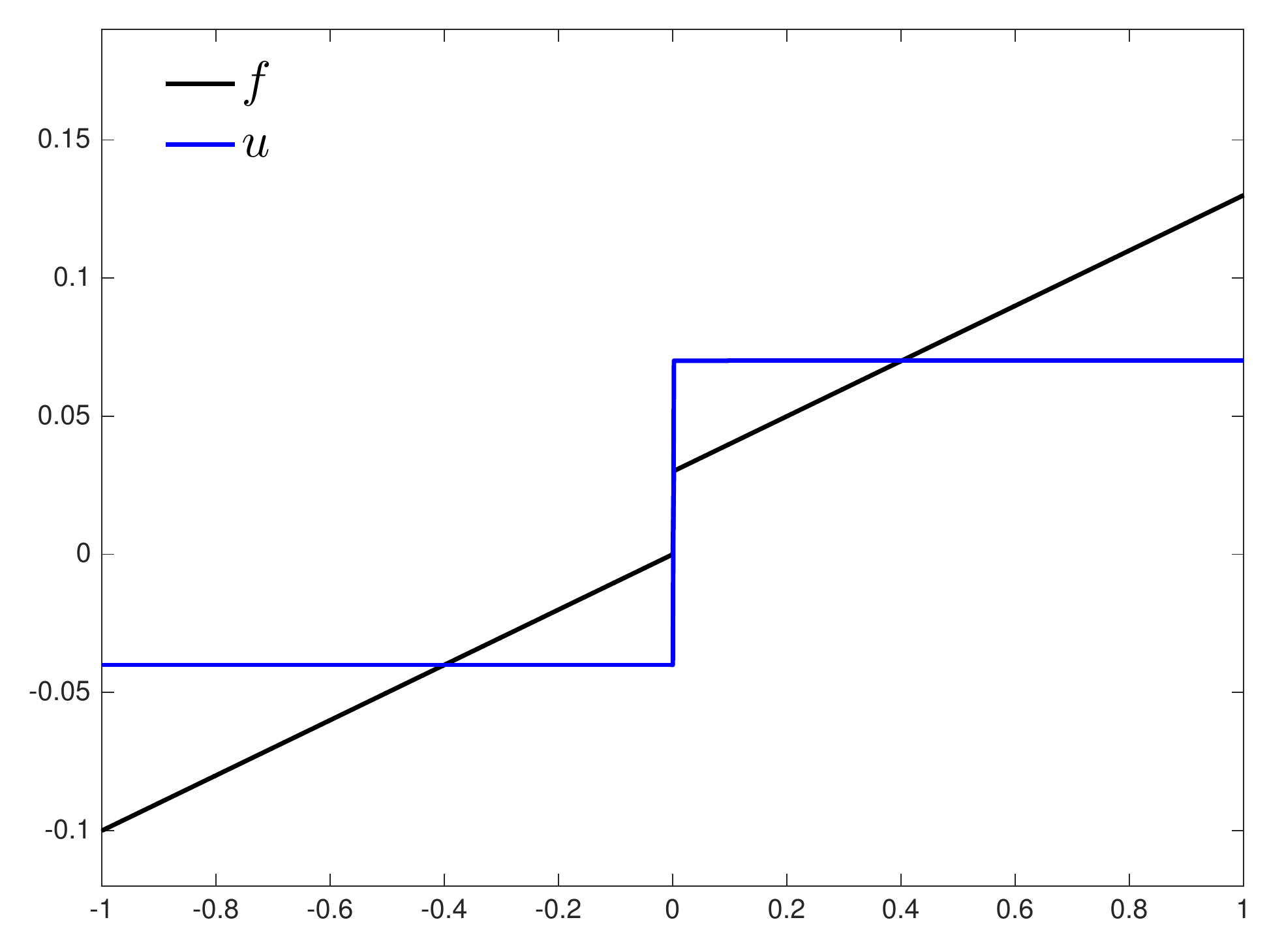} &
  			 \includegraphics[width=.48\textwidth]{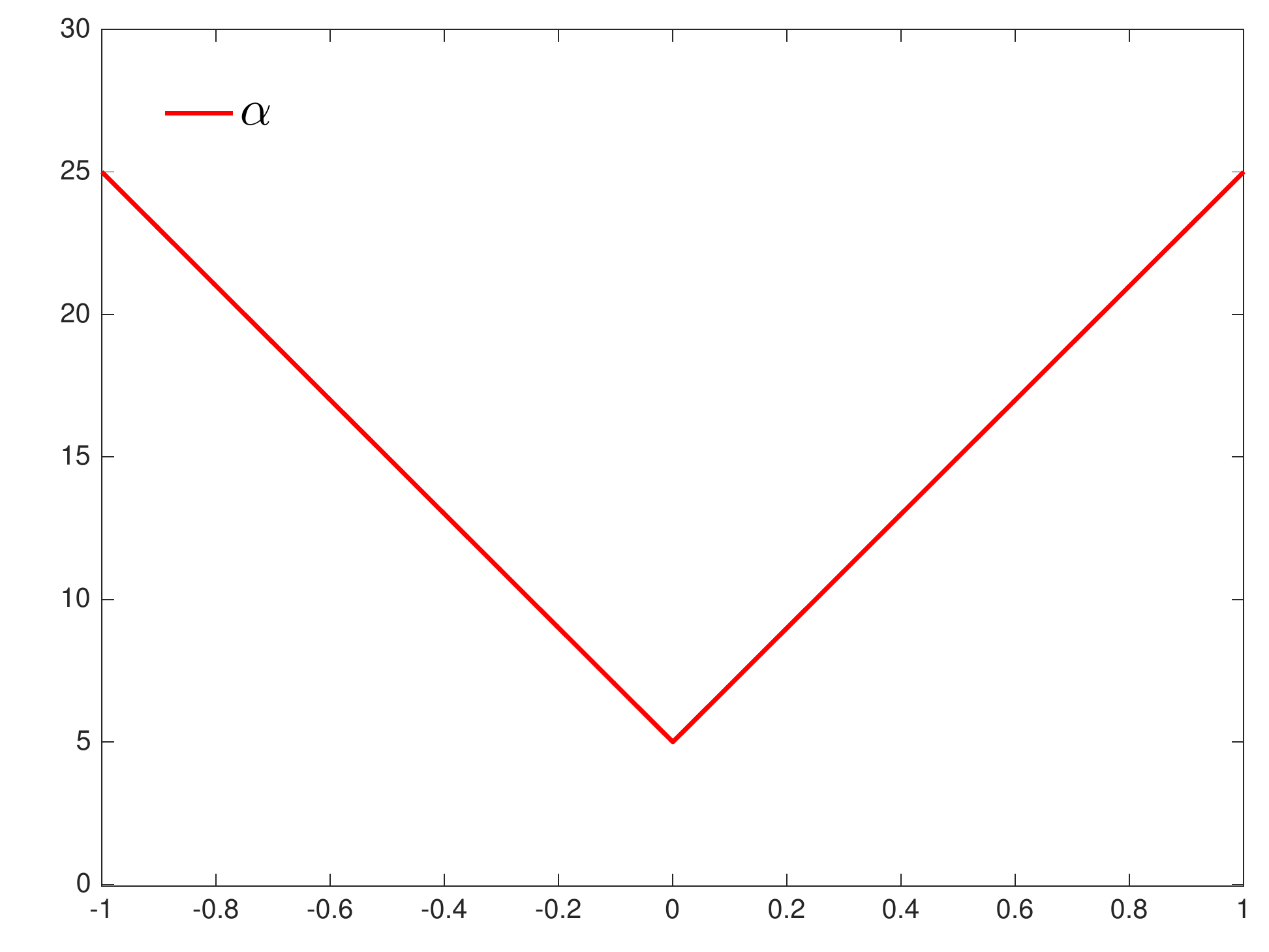} 
		\end{tabular}
		\captionsetup{width=.85\linewidth}
		\caption{Case where the jump of $u$ is larger than the jump of $f$ and has the same direction, i.e., $u^{l}(x)< f^{l}(x)<f^{r}(x)<u^{r}(x)$}
		\label{lbl:counterex_5}
	\end{subfigure}\hspace{0.5cm}
		\begin{subfigure}[t]{0.48\textwidth}
		\begin{tabular}{@{}cc@{}}
			\includegraphics[width=.48\textwidth]{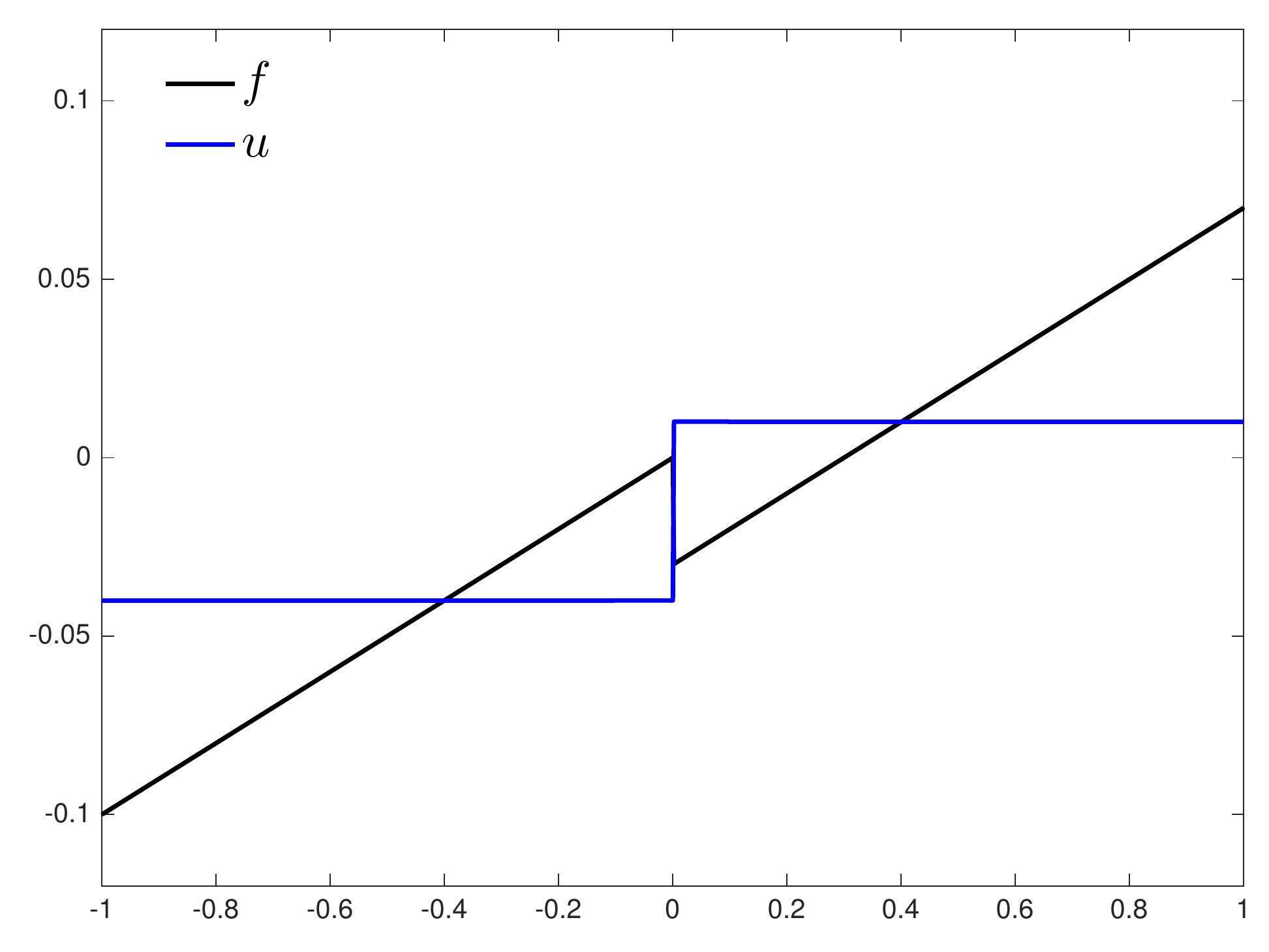} &
  			 \includegraphics[width=.48\textwidth]{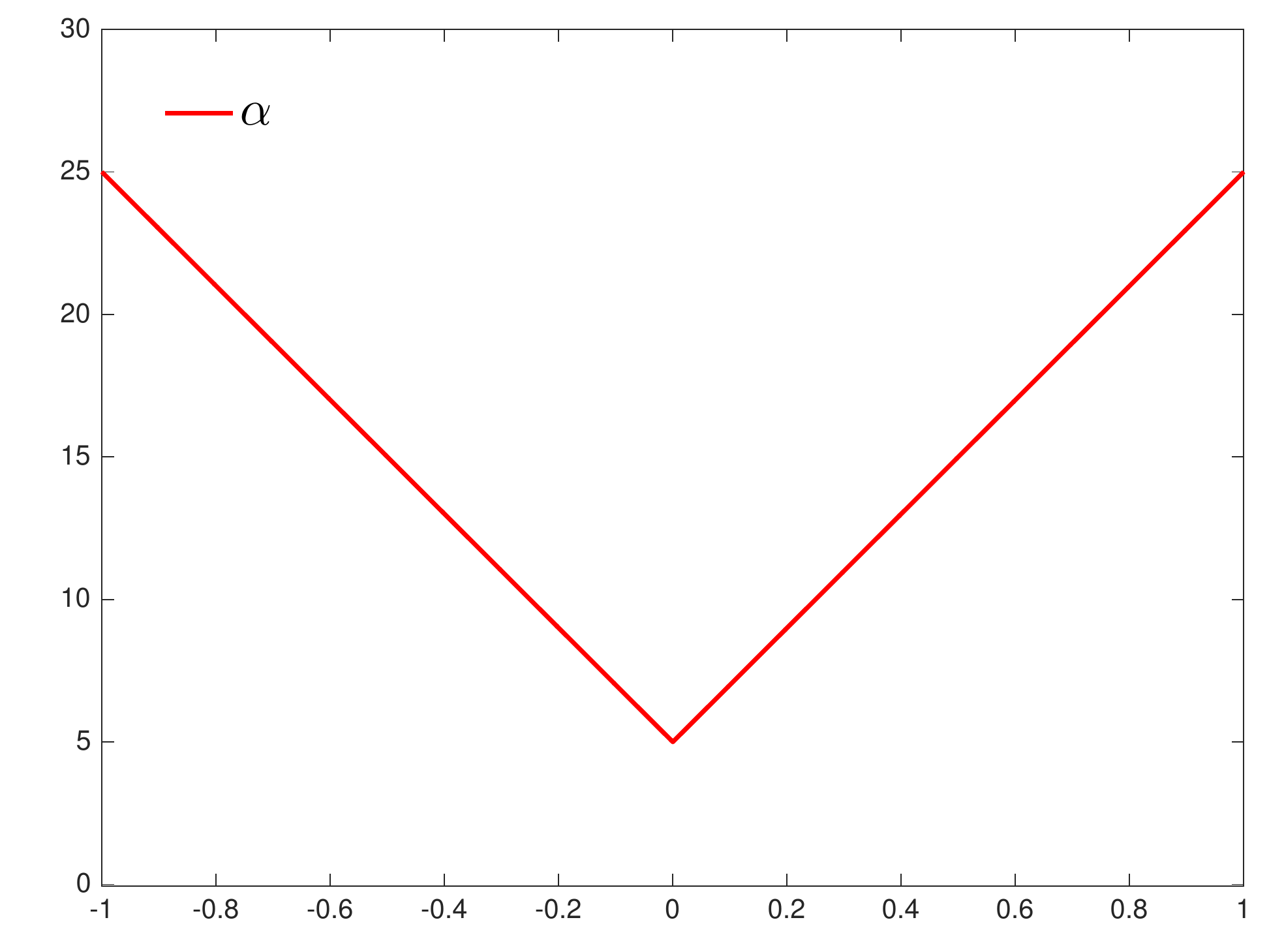} 
		\end{tabular}
		\captionsetup{width=.85\linewidth}
		\caption{Case where the jump of $u$ is larger than the jump of $f$ and has the opposite direction, i.e., $u^{l}(x)<f^{r}(x)<f^{l}(x)<u^{r}(x)$}
		\label{lbl:counterex_6}
	\end{subfigure}
\caption{Series of numerical examples that confirm that certain cases, as these are described in Table \ref{question_table}, can indeed occur using weighted $\tv$ regularisation}	
\label{fig:counterexamples}
\end{figure}

In Figure \ref{lbl:counterex_1} we have an example where both the data $f$ and the solution $u$ are continuous at a point $x$ despite the fact that $D\alpha'(\{x\})>0$. Note also that a trivial example here would also be the case where $\alpha$ is so large that $u$ would be a constant.

In Figure \ref{lbl:counterex_2} we depict an example where a new discontinuity is created for the solution $u$ at the point $x=0$ where the data function $f$ is continuous. Note that in this specific example, the creation of this discontinuity is guaranteed to happen since $f$ is continuous and strictly increasing. This was shown in Proposition \ref{lbl:alpha_spike_down} for data functions with a downward spike, but it can be easily extended  to an absolute value-type function as we have here.  Note that the estimate $|Du|(\{x\})\le D\alpha'(\{x\})$ holds here with strict inequality. Observe that the jump of $\alpha'$ is very large at the point $0$ in contrast to the jump of $u$ there. In fact there is a further upper bound for $Du(\{x\})$ which is independent of $\alpha$; see Theorem \ref{lbl:TVu_less_TVf} of Section \ref{sec:boundTV} and Theorem \ref{lbl:TVu_less_TVf_2} of Section \ref{sec:alpha_zero}. As we mention in the introduction, in these theorems it is shown that $|Du|(\om)\le |Df|(\om)$. \footnotetext{ when the jumps of $u$ and $f$ have opposite directions}
  
On the contrary, in Figure \ref{lbl:counterex_3} we have an example where the estimate $|Du|(\{x\})\le D\alpha'(\{x\})$ holds with an equality. We use the same $f$ as in Figure \ref{lbl:counterex_3} and a similar weight function $\alpha$ with a small jump of $\alpha'$ at $x=0$. Note that here it holds that $(-\epsilon,\epsilon)\subseteq\mathrm{supp}(|Du|)$ for some small $\epsilon>0$ and, as Proposition \ref{lbl:lipschitz_alpha}  predicts, we have $|Du|(\{x\})= D\alpha'(\{x\})$.

In Figure \ref{lbl:counterex_4} we encounter another, perhaps unexpected situation. Even though we are using a smooth weight function $\alpha$ and, as Proposition \ref{lbl:jump_set_incl} states, the jumps of $u$ should occur at the same points where $f$ has jumps,  condition \eqref{eq:correct_jumpTV} is violated. Indeed, here we have that the whole jump of $u$ is above that of $f$, i.e., $f^{l}(x)<f^{r}(x)<u^{l}(x)<u^{r}(x)$. Nevertheless, there still holds $|Du|(\{x\})\le |Df|(\{x\})$ in accordance with Proposition \ref{lbl:correct_jump}.

In Figure \ref{lbl:counterex_5} we have an example where the size of the jump of $u$ is larger than the one of $f$ since we have that $\alpha'$ also jumps there. Note that in this example the jump directions are the same.

On the other hand, in Figure \ref{lbl:counterex_6} we have a similar situation, i.e., the jump of $u$ is larger than the one of $f$, but also the orientation of the jumps is different as $Du(\{0\})>0$, while  $Df(\{0\})<0$.

Let us note here that situations where condition \eqref{eq:correct_jumpTV} is violated, like the examples in Figures \ref{lbl:counterex_4} and \ref{lbl:counterex_6} in which the jump is moving upwards or changing orientations,  shows that the weighted $\tv$ has fundamentally different regularising properties when compared  to the scalar $\tv$. Moreover, this is also true in comparison with other popular regularisers like Huber-$\tv$ or $\tgv$. Indeed, for (the scalar parameter versions of) one dimensional Huber-$\tv$ and $\tgv$, condition \eqref{eq:correct_jumpTV} always holds \cite{BrediesL1, Papafitsoros_Bredies, journal_tvlp}.

Finally we note that even though in Figure \ref{fig:counterexamples} we present only numerical results, one can easily calculate analytically the explicit forms of the solutions depicted there. Indeed, in Section \ref{sec:exact} we do that exhaustively for the examples in Figures \ref{lbl:counterex_2}--\ref{lbl:counterex_3}.

Before we finish this section, we mention two last results. The first one says that the solution $u$ is constant in areas where the weight function $\alpha$ has  high gradient. This is essentially a corollary of the fact that the solution $u$ is bounded by the data $f$ with respect to the uniform norm.

\newtheorem{large_gradient_plateau}[alpha_sgn]{Proposition}
\begin{large_gradient_plateau}\label{lbl:large_gradient_plateau}
Let  $f\in\bv(\om)$ and $\alpha$ be a weight function, differentiable in an open interval $I\subseteq \Omega$, such that 
\begin{equation}\label{bound_weight_gradient}
|\alpha'(x)|>2\|f\|_{\infty},\quad \text{for all }x\in I.
\end{equation}
Let  $u$ be the solution of \eqref{weighted_rof} with data $f$ and weight $\alpha$. Then $u$ is constant in $I$.
\end{large_gradient_plateau}

\begin{proof}
One can easily observe that for the solution $u$ the following maximum principle holds true:
\begin{equation}\label{maximum_principle}
\essinf_{x\in\om} f(x)\le \essinf_{x\in\om} u(x) \le \esssup_{x\in\om} u(x) \le \esssup_{x\in\om} f(x).
\end{equation}
This stems from the fact that the function 
\[u_{\inf f}^{\sup f}(x):=\max(\min(\esssup_{t\in\om} f(t), \tilde{u}(x)), \essinf_{t\in\om} f(t)),\]
has always equal or less energy than $u$, where here $\tilde{u}$ is any good representative of $u$. Using a simple translation argument, we can assume that 
\begin{align*}
\esssup_{x\in\om} f(x)&=\|f\|_{\infty},\\
\essinf_{x\in\om} f(x)&=-\|f\|_{\infty}.
\end{align*}
This together with \eqref{maximum_principle} implies that $\|f-u\|_{\infty}\le 2\|f\|_{\infty}$ and, thus, for the dual variable $v$ it must hold 
\begin{equation}\label{bound_v_prime}
\|v'\|_{\infty}\le 2\|f\|_{\infty}.
\end{equation}
Suppose now  that $u$ is not constant in $I$. Then from condition \eqref{opt2} we have that there exists a point $x\in I$ such that $|v(x)|=\alpha(x)$ and, without loss of generality, we assume that $v(x)=\alpha(x)$. We can also assume without loss of generality that $\alpha'<-2\|f\|_{\infty}$ in $I$. Using the  fundamental theorem of calculus, the mean value theorem, the fact that $v(x)=\alpha(x)$ and $v\le\alpha$, we get for every $t>x$ with $t\in I$
\begin{equation}\label{bound_v_prime_average}
\frac{1}{t-x}\int_{x}^{t}v'(s)ds\le \frac{\alpha(x)-\alpha(t)}{x-t}<-2\|f\|_{\infty}.
\end{equation}
We claim now that for every $\epsilon>0$ small enough there exists a set $A\subseteq (x,x+\epsilon)$ of positive Lebesgue
 measure such that $v'<-2\|f\|_{\infty}$ in $A$, which of course contradicts \eqref{bound_v_prime}. Indeed if there was a small enough $\epsilon>0$ such that $v'\ge -2\|f\|_{\infty}$ almost everywhere on $(x,x+\epsilon)\subseteq I$ then we would get for every $x<t<x+\epsilon$
\[-2\|f\|_{\infty}=\frac{1}{t-x}\int_{x}^{t}-2\|f\|_{\infty}\,ds\le \frac{1}{t-x}\int_{x}^{t}v'(s)ds,\]
which contradicts \eqref{bound_v_prime_average}.
\end{proof}

The last result we state is that, as in the scalar case, solutions of the weighted $\mathrm{TV}$ minimisation are constant near the boundary of $\om$. The proof is very simple and stems from the facts that $\alpha$ is bounded away from zero as well as $v\in H_{0}^{1}(\om)$.

\newtheorem{boundary_constant}[alpha_sgn]{Proposition}
\begin{boundary_constant}\label{lbl:boundary_constant}
Let $\om=(a,b)$, $f\in\bv(\om)$ and $\alpha\in C(\overline{\om})$ with $\alpha>0$. Then there exist $a<x_{1}<x_{2}<b$ such that the solution $u$ of \eqref{weighted_rof} with data $f$ and weight $\alpha$, is constant in $(a,x_{1})$ and $(x_{2},b)$.
\end{boundary_constant}
\begin{proof}
Suppose that there does not exist a $x_{1}\in (a,b)$ such that $u$ is constant in $(a,x_{1})$. Then from condition \eqref{opt2} it follows that there exists a sequence $(a_{n})_{n\in\NN}\in (a,b)$, converging to $a$ and 
\[|v(a_{n})|=\alpha(a_{n}), \quad \text{for every }n\in\NN.\]
From the continuity of $v$ and the fact that it belongs to $H_{0}^{1}(\om)$ we have
\[0=\lim_{n\to\infty}|v(a_{n})|=\lim_{n\to\infty}\alpha(a_{n})\ge \min_{x\in(a,b)} \alpha(x)>0,\]
which is a contradiction. Similarly we deal with the right part of the boundary of $\om$.
\end{proof}

\subsection{A partial semigroup property}\label{sec:semigroup}
For this section it is convenient to introduce the following notation 
\[S_{\alpha(x)}(f):=\underset{u\in\bv(\om)}{\operatorname{argmin}}\;\frac{1}{2}\int_{\om}(f-u)^{2}dx+\int_{\om}\alpha(x)d|Du|,\]
i.e., $S_{\alpha(x)}(f)$ denotes the solution of \eqref{weighted_rof} with data $f\in\bv(\om)$ and weight function $\alpha\in C(\overline{\om})$. Again we slightly abuse the notation by writing $S_{\alpha(x)}(f)$ when $\alpha$ is not a constant function and $S_{\alpha}(f)$ when it is. It is a well-known fact that the following \emph{semigroup} property holds for the one dimensional scalar total variation problem \cite{scherzer2009variational} 
\begin{equation}\label{semigroup_scalar}
S_{\alpha_{1}+\alpha_{2}}(f)=S_{\alpha_{2}}\left (S_{\alpha_{1}}(f) \right )=S_{\alpha_{1}}\left (S_{\alpha_{2}}(f) \right ),\quad \alpha_{1}, \alpha_{2}>0.
\end{equation}
In other words, one can obtain the solution of $\tv$ regularisation with scalar parameter $\alpha_{1}+\alpha_{2}$ and data $f$ by applying  $\tv$ regularisation with parameter $\alpha_{2}$ to the result which is obtained by applying $\tv$ regularisation with parameter $\alpha_{1}$ to the data $f$. The result remains the same if we apply $\tv$ regularisation first with parameter $\alpha_{2}$ and then with $\alpha_{1}$. Here we examine  whether this property holds true in the weighted $\tv$ regularisation or not. The next proposition states that this is indeed the case when the second regularisation parameter is a scalar. 

\newtheorem{semigroup}[alpha_sgn]{Proposition}
\begin{semigroup}[Partial semigroup property]\label{lbl:semigroup}
Let $f\in\bv(\om)$, $\alpha_{1}\in C(\overline{\om})$ with $\alpha_{1}>0$, and $\alpha_{2}>0$ be a scalar.
Then
\begin{equation}\label{smg}
S_{\alpha_{1}(x)+\alpha_{2}}(f)=S_{\alpha_{2}}\left (S_{\alpha_{1}(x)}(f) \right ).
\end{equation}
\end{semigroup}
\begin{proof}
Let $u_{1}:=S_{\alpha_{1}(x)}(f) $ and $u_{2}:=S_{\alpha_{2}}\left (S_{\alpha_{1}(x)}(f) \right )$.
The optimality conditions for the corresponding minimisation problems read
\begin{alignat*}{3}
v_{1}'&=f-u_{1}, 					    &\qquad\qquad v_{2}'&=u_{1}-u_{2},\\
-v_{1}&\in\alpha_{1}(x)\Sgn(Du_{1}), &-v_{2}&\in\alpha_{2}\Sgn(Du_{2}),
\end{alignat*}
where both $v_{1},v_{2}\in H_{0}^{1}(\om)$. Defining $v_{12}=v_{1}+v_{2}$,  we have $v_{12}\in H_{0}^{1}(\om)$ and 
\[v_{12}'=f-u_{2}.\]
Thus, for \eqref{smg} it suffices to show that
\begin{equation}
-v_{1}-v_{2}\in (\alpha_{1}(x)+\alpha_{2})\sgn(Du_{2}).
\end{equation}
Since $|v_{1}(x)|\le \alpha_{1}(x)$ for Lebesgue--almost every $x\in\om$ and $\|v_{2}\|_{\infty}\le \alpha_{2}$, it obviously holds that $|v_{12}(x)|\le \alpha_{1}(x)+\alpha_{2}$ for Lebesgue--almost every $x\in\om$. Consequently, it suffices to check that
\begin{equation}\label{v12_toprove}
-v_{12}=(\alpha_{1}(x)+\alpha_{2})\sgn(Du_{2}),\quad |Du_{2}|-\text{a.e.}
\end{equation}
In what follows we work with the continuous representatives of $\sgn(Du_{1})$ and $\sgn(Du_{2})$. Let $x\in\om$ with $x\in \supp(|Du_{2}|)$ and $\sgn(Du_{2})(x)=1$, then we will show that $\sgn(Du_{1})(x)=1$ as well (similarly we proceed with $-1$). Note that in this case (any good representative of) $u_{2}$ is increasing near $x$ in the sense that there exists a small enough $\epsilon>0$ such that
\[u_{2}(s)\le u_{2}(s')<u_{2}(t)\le u_{2}(t'),\quad \text{for every } x-\epsilon<s<s'<x<t<t'<x+\epsilon. \]
Indeed this comes from the fact that
\[u_{2}^{r}(s)-u_{2}^{r}(t)=\int_{(s,t]} \sgn(Du_{2})(y)d|Du_{2}|,\quad s<x<t,\]
combined with the continuity of $\sgn(Du_{2})$ and the fact that $\sgn(Du_{2})(x)=1$. We claim then, that $x\in \supp(|Du_{1}|)$ as well. Indeed, if this is not true then $|Du_{1}|((x-\epsilon,x+\epsilon))=0$ for a small  $\epsilon>0$, and thus the same would hold for $u_{2}$ (which follows directly when using the optimality conditions). Since $x\in \supp(|Du_{1}|)$, then either $\sgn(Du_{1})(x)=1$ or $\sgn(Du_{1})(x)=-1$. However, the latter case can be excluded and thus 
$\sgn(Du_{1})(x)=1$. Indeed, from the optimality conditions, and by considering different cases depending on whether $x$ it is a jump point or not, it follows that  the data $u_{1}$ and the solution $u_{2}$ cannot be decreasing and increasing respectively near $x$. 
 As a result, if $\sgn(Du_{2})(x)=1$ it follows that $\sgn(Du_{1})(x)=1$. From the optimality conditions we obtain
\[-v_{2}(x)=\alpha_{2}\quad \text{and}\quad -v_{1}(x)=\alpha_{1}(x),\]
respectively. Hence $-v_{12}(x)=\alpha_{1}(x)+\alpha_{2}$. Similarly if $\sgn(Du_{2})(x)=1$ then 
$-v_{12}(x)=-\alpha_{1}(x)-\alpha_{2}$ and thus \eqref{v12_toprove} holds.
\end{proof}

It is natural to ask whether \eqref{smg} holds when both $\alpha_{1}, \alpha_{2}$ are non-scalar continuous weight functions. Propositions \ref{lbl:alpha_spike_up} and \ref{lbl:alpha_spike_down} indicate that this is not necessarily true. 
%
\begin{figure}[t!]
\begin{center}
\begin{subfigure}[t]{0.3\textwidth}
	\centering
	\includegraphics[width=0.9\textwidth]{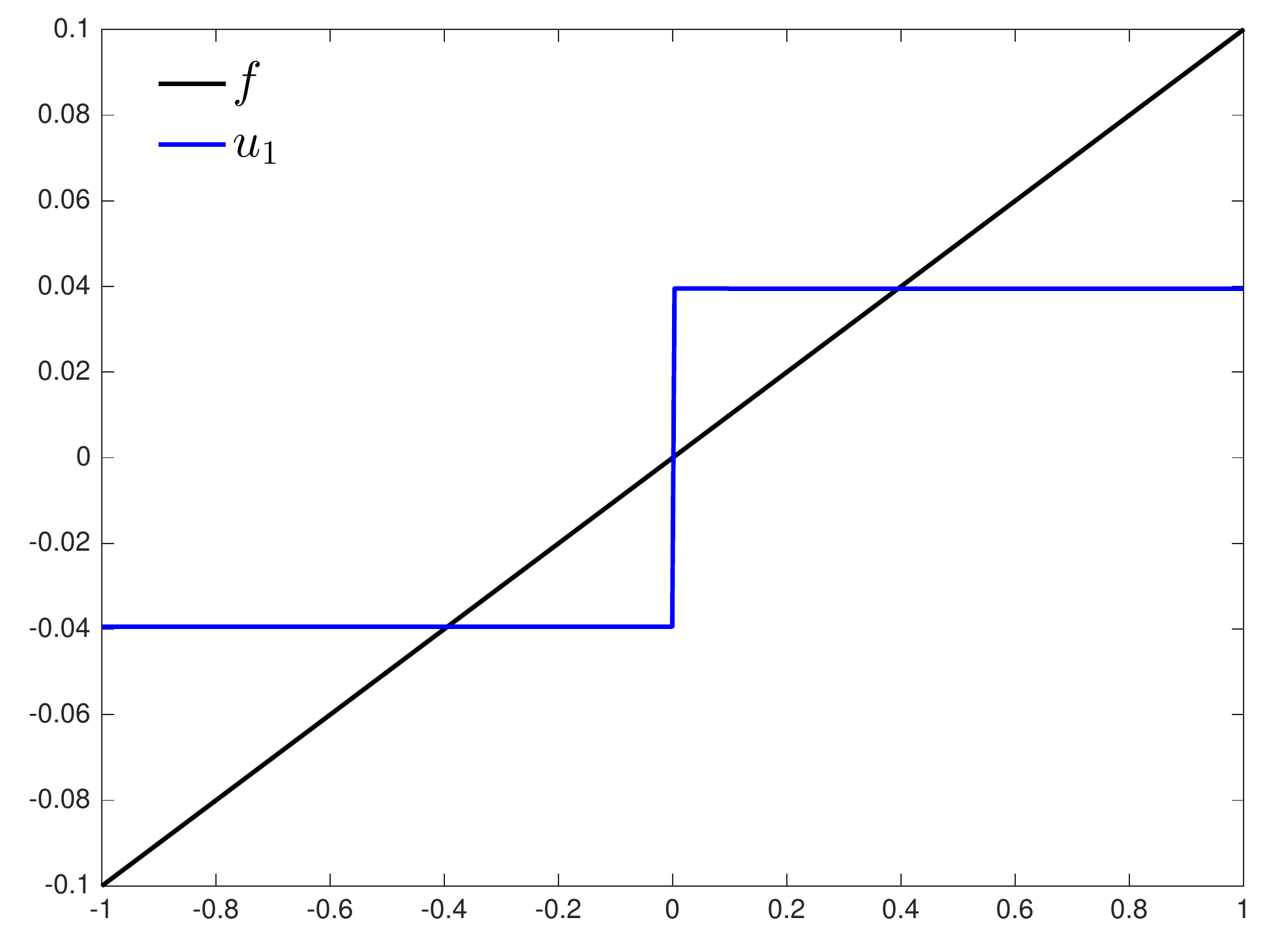}
	\label{fig:sg_spikes_f_u1}
\end{subfigure}
\begin{subfigure}[t]{0.3\textwidth}
	\centering
	\includegraphics[width=0.9\textwidth]{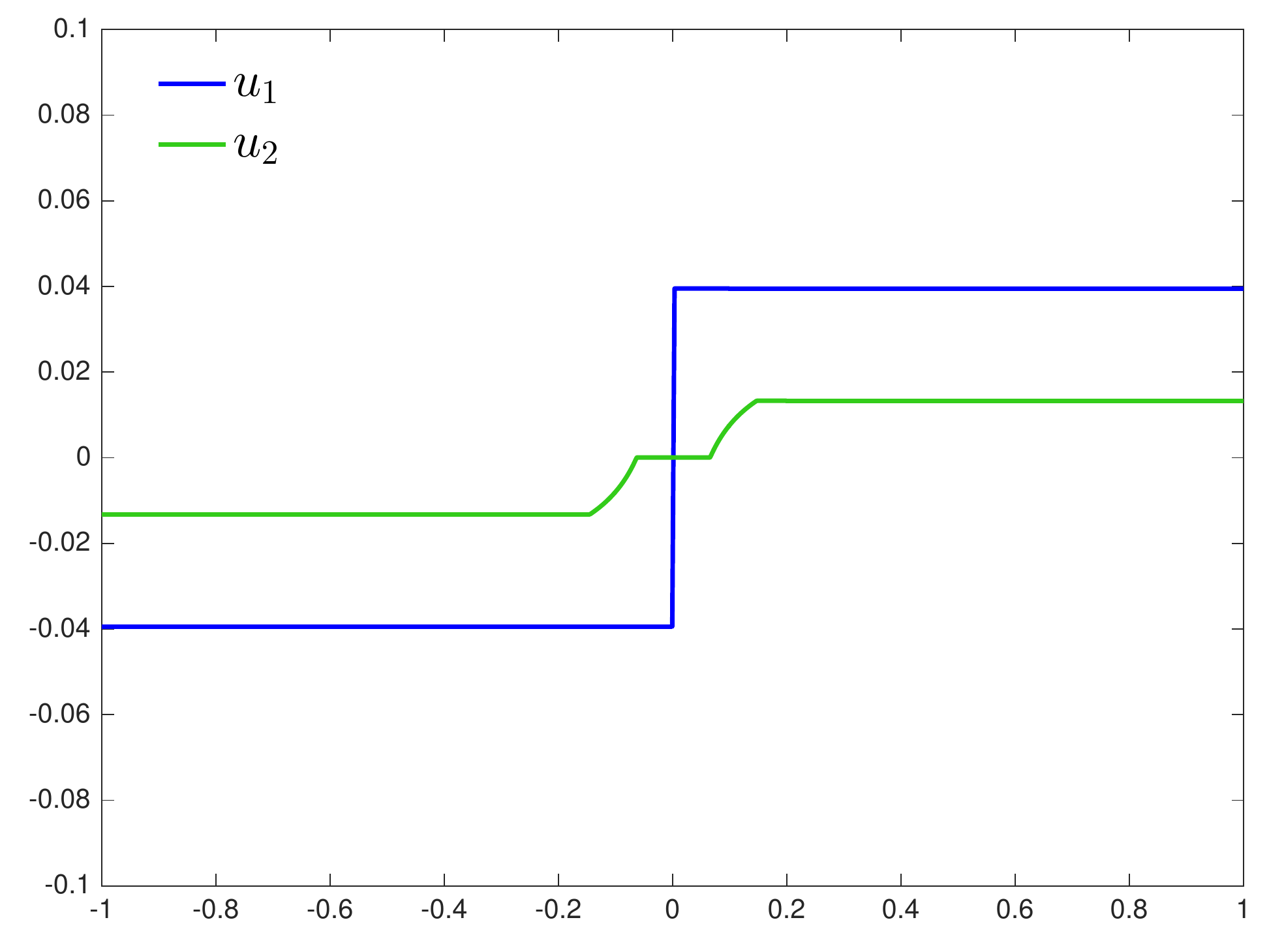}
	\label{fig:sg_spikes_u1_u2}
\end{subfigure}
\begin{subfigure}[t]{0.3\textwidth}
	\centering
	\includegraphics[width=0.9\textwidth]{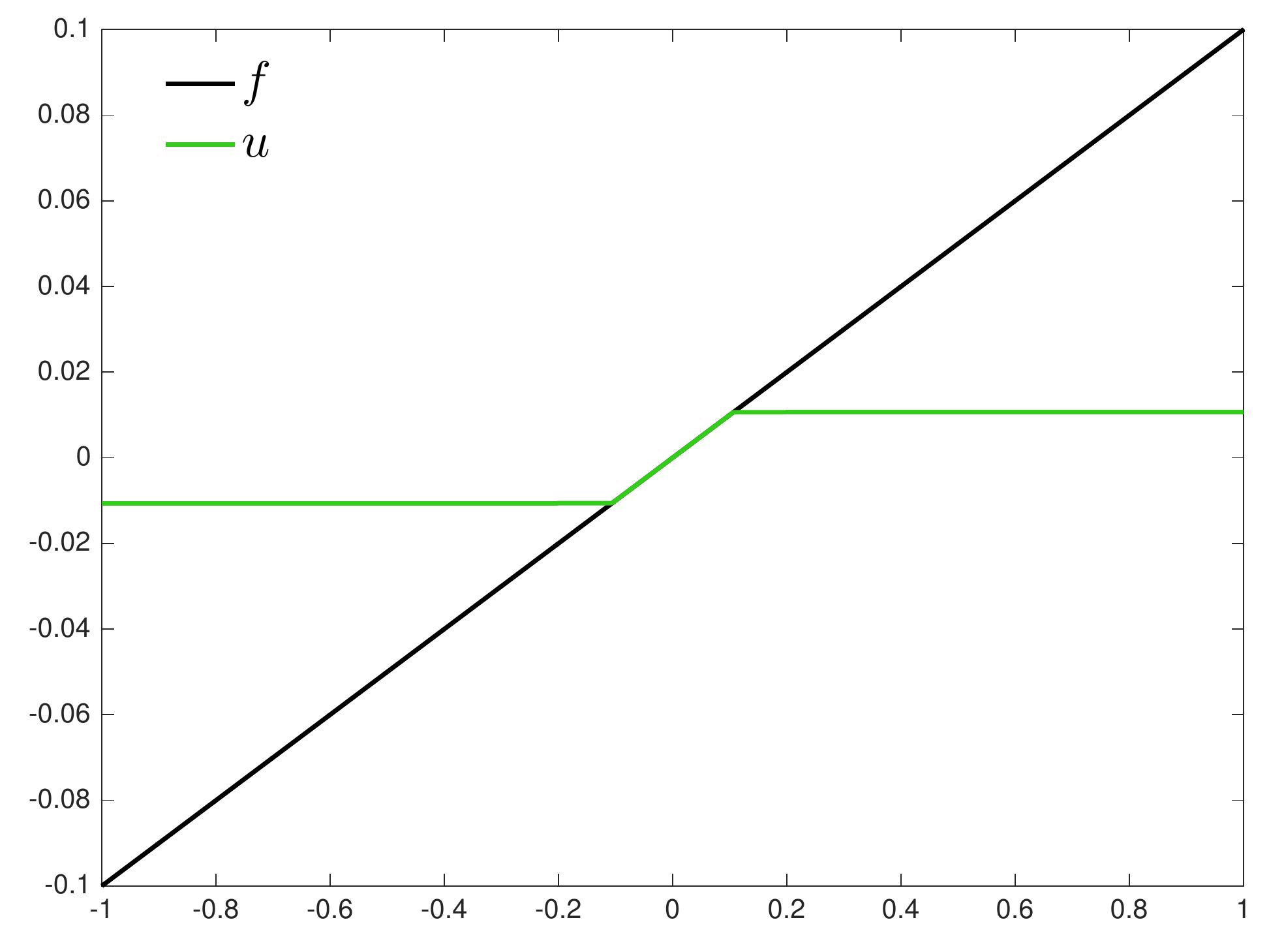}
	\label{fig:sg_spikes_f_u}
\end{subfigure}

\hspace{-0.37cm}
\begin{subfigure}[t]{0.3\textwidth}
	\centering
	\captionsetup{width=.9\linewidth}
	\includegraphics[width=0.9\textwidth]{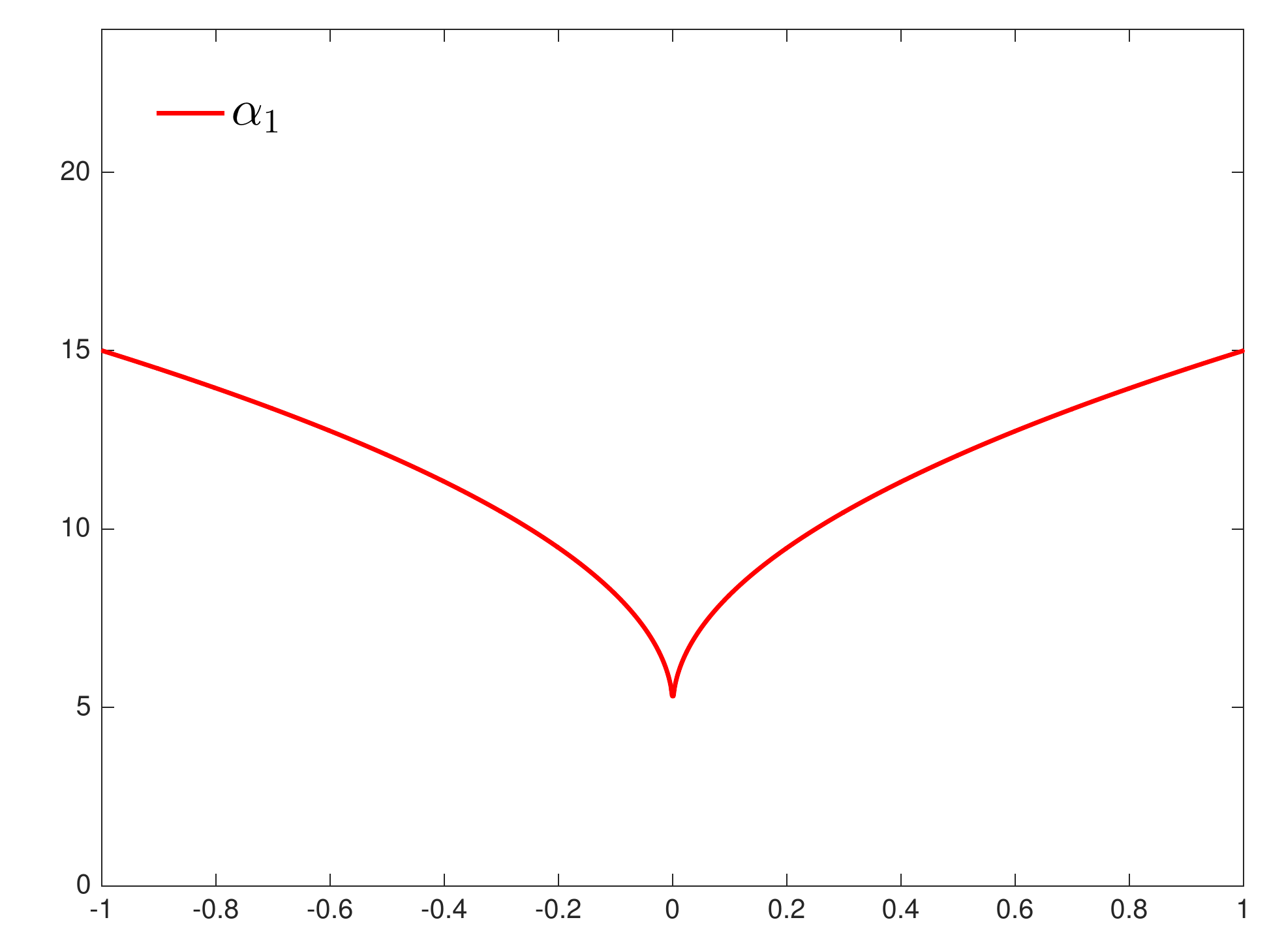}
	\caption{Regularisation of the initial data with a downward  spike weight function $\alpha_{1}$. Creation of a new discontinuity}
	\label{fig:sg_spikes_f_u1}
\end{subfigure}
\begin{subfigure}[t]{0.3\textwidth}
	\centering
	\captionsetup{width=.9\linewidth}
	\includegraphics[width=0.9\textwidth]{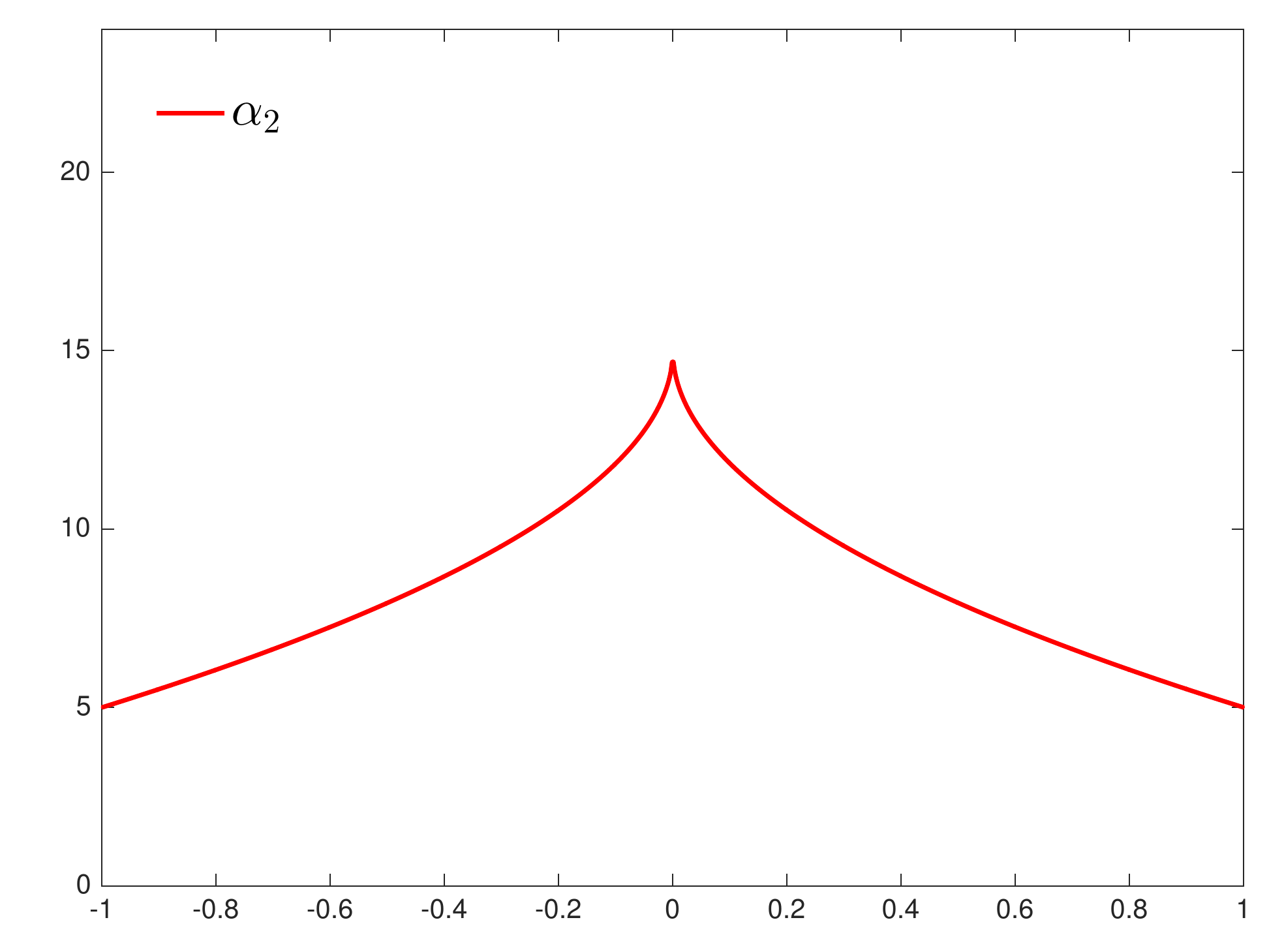}
	\caption{Regularisation of the result in Figure \ref{fig:sg_spikes_f_u1} with an upward  spike weight function $\alpha_{2}=20-\alpha_{1}$. Creation of a plateau}
	\label{fig:sg_spikes_u1_u2}
\end{subfigure}
\begin{subfigure}[t]{0.3\textwidth}
	\centering
	\captionsetup{width=.9\linewidth}
	\includegraphics[width=0.9\textwidth]{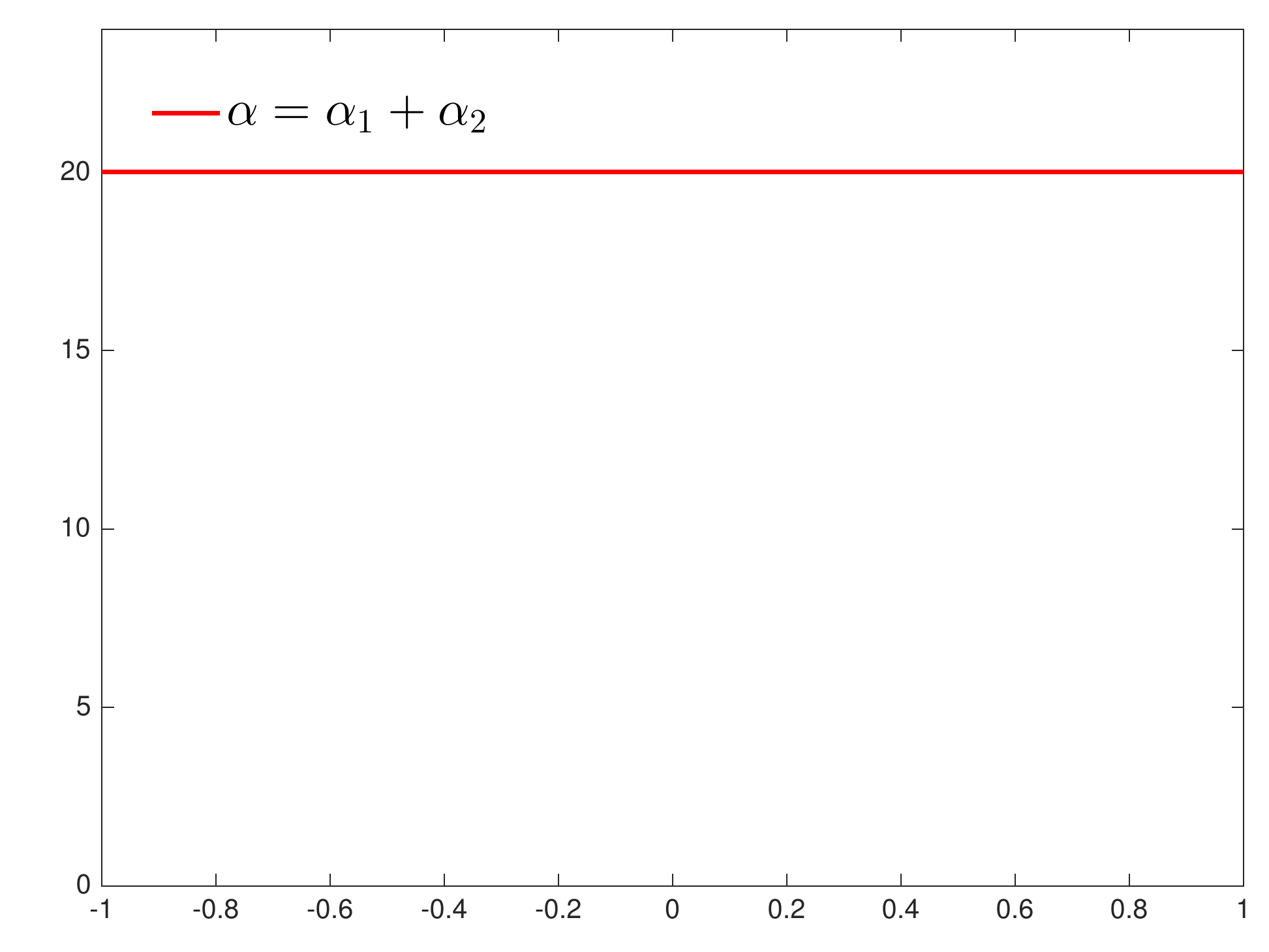}
	\caption{Regularisation of the initial data with the sum of the previous weight functions $\alpha=\alpha_{1}+\alpha_{2}$. The result is different than the one obtained by applying the regularisation in two steps}
	\label{fig:sg_spikes_f_u}
\end{subfigure}
\caption{Failure of the semigroup property for the weighted total variation model in the case where both weights $\alpha_{1}$  and $\alpha_{2}$ are not constant}
\label{fig:sg_spikes}
\end{center}
\end{figure}
We illustrate this fact with a numerical example in Figure \ref{fig:sg_spikes}. We start with an affine function $f$ and first apply weighted $\tv$ regularisation with a square root type weight function $\alpha_{1}$. Note the discontinuity that it is created in the solution $u_{1}$ exactly at the point where $\alpha_{1}$ has the downward spike; see Figure \ref{fig:sg_spikes_f_u1}. By further applying weighted $\tv$ regularisation to the result with weight function $\alpha_{2}=20-\alpha_{1}$, a plateau is created at this point; compare Figure \ref{fig:sg_spikes_u1_u2}.  The final result is different to the solution of the $\tv$ regularisation of the initial data $f$ with weight $\alpha_1+\alpha_{2}=20$, which is shown in Figure \ref{fig:sg_spikes_f_u}.

However, what is perhaps even more surprising is that  property \eqref{smg} of Proposition \ref{lbl:semigroup} is non-commutative, i.e., it might fail even in the case where we first regularise with a constant weight function and  then with a non constant one; hence the term \emph{partial semigroup property} in  Proposition \ref{lbl:semigroup}, i.e.,
\begin{equation}\label{non_commute}
S_{\alpha_{1}(x)+\alpha_{2}}(f)=S_{\alpha_{2}}\left (S_{\alpha_{1}(x)}(f) \right )\ne S_{\alpha_{1}(x)}\left (S_{\alpha_{2}}(f) \right ).
\end{equation}
In order to get an intuition for \eqref{non_commute}, observe that the optimality conditions read
\begin{alignat*}{3}
v_{1}'&=f-u_{1}, 					    &\qquad\qquad v_{2}'&=u_{1}-u_{2},\\
-v_{1}&\in\alpha_{2}\Sgn(Du_{1}), &-v_{2}&\in\alpha_{1}(x)\Sgn(Du_{2}).
\end{alignat*}
Recalling $(v_{1}+v_{2})'=f-u_{2}$, for $u_{2}$ to be the result obtained after regularising $f$ with $\alpha_{1}+\alpha_{2}(x)$ requires that
\begin{equation}\label{could_fail}
-v_{1}-v_{2}=(\alpha_{1}(x)+\alpha_{2})\Sgn(Du_{2}),\quad |Du_{2}|-\text{a.e.}
\end{equation}
However, condition \eqref{could_fail} could fail if for instance $u_{2}$ has a new discontinuity created at a point $x$ around which $u_{1}$ is constant and below (or above) $f$. That would enforce $|v_{2}(x)|=\alpha_{1}(x)$ and $|v_{1}(x)|<\alpha_{2}$ and thus \eqref{could_fail} would not hold.

\begin{figure}[t!]
\begin{center}
\begin{subfigure}[t]{0.24\textwidth}
	\centering
	\includegraphics[width=0.9\textwidth]{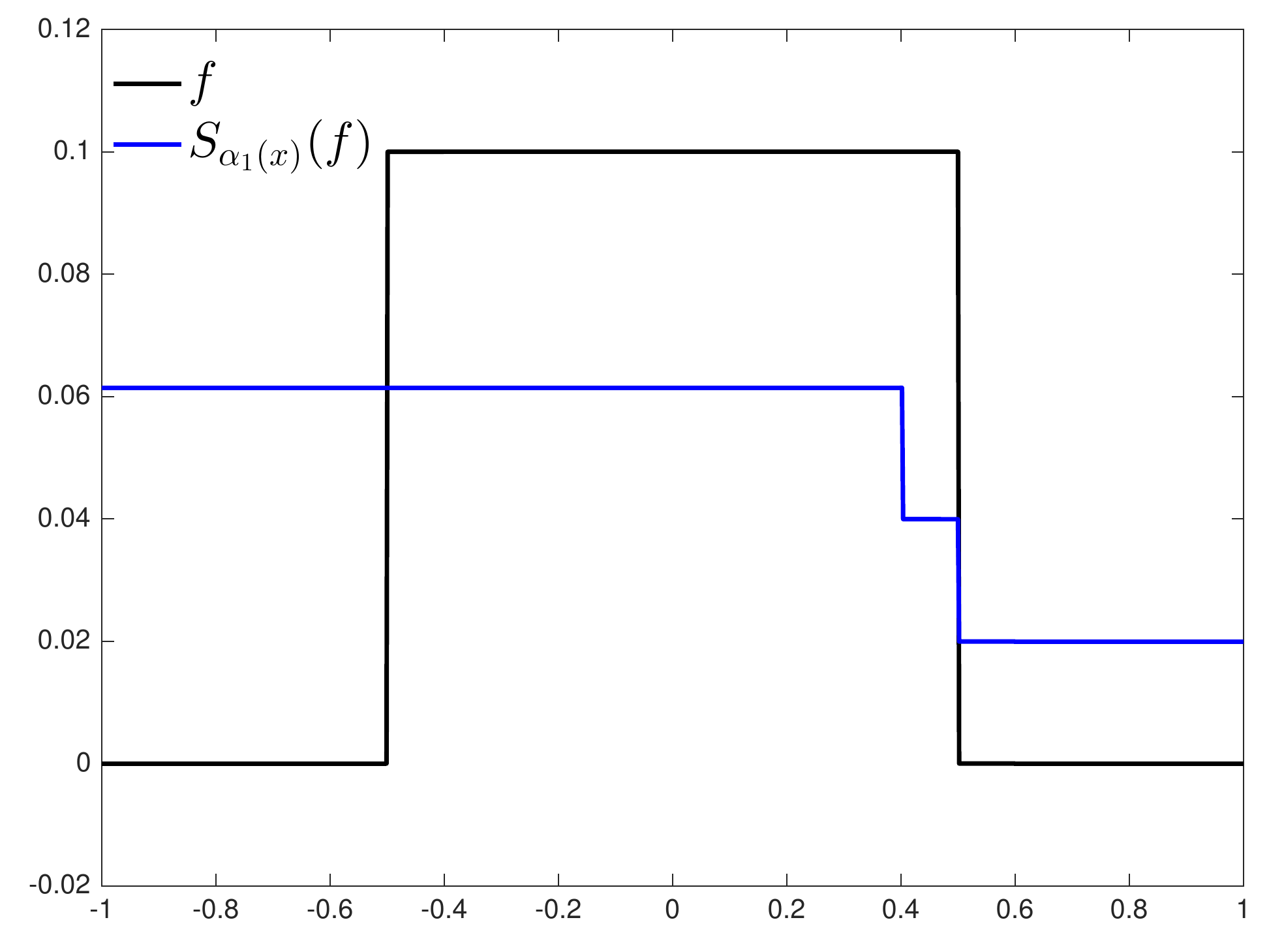}
	\label{fig:sg2_f_u1b}
\end{subfigure}
\begin{subfigure}[t]{0.24\textwidth}
	\centering
	\includegraphics[width=0.9\textwidth]{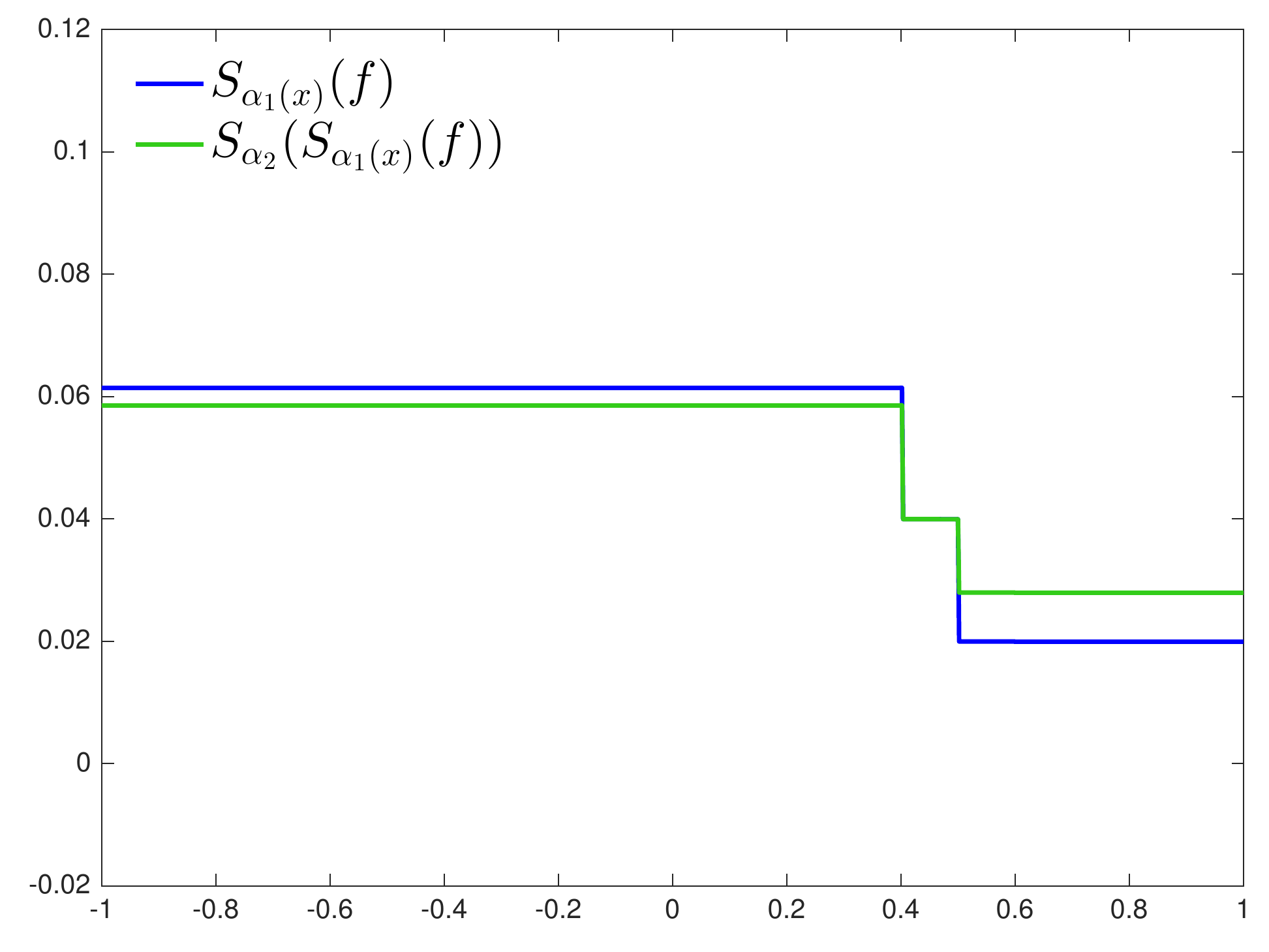}
	\label{fig:sg2_u1b_u2b}
\end{subfigure}
\begin{subfigure}[t]{0.24\textwidth}
	\centering
	\includegraphics[width=0.9\textwidth]{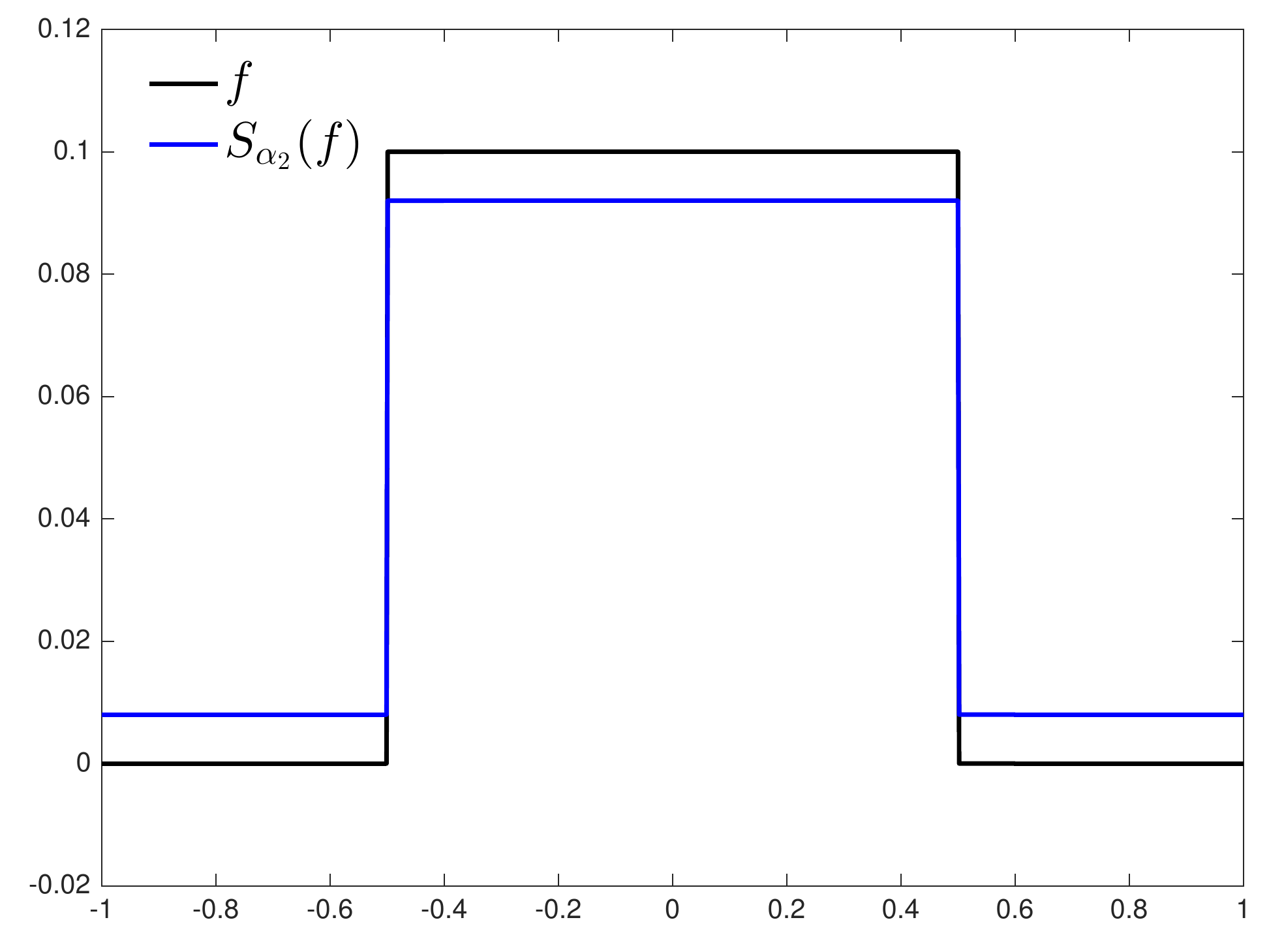}
	\label{fig:sg2_f_u1}
\end{subfigure}
\begin{subfigure}[t]{0.24\textwidth}
	\centering
	\includegraphics[width=0.9\textwidth]{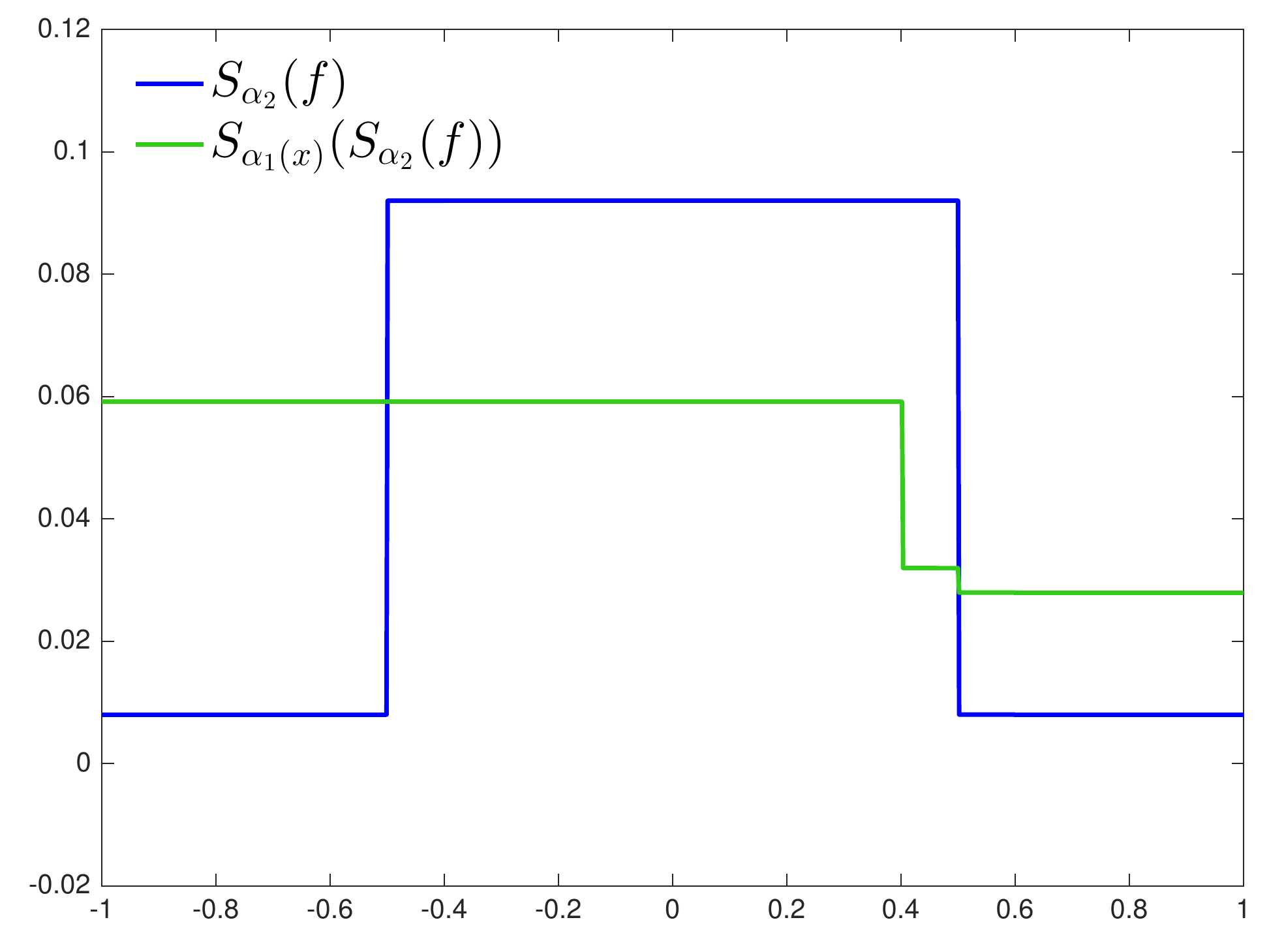}
	\label{fig:sg2_u1_u2}
\end{subfigure}

\hspace{-0.37cm}
\begin{subfigure}[t]{0.24\textwidth}
	\centering
	\captionsetup{width=.9\linewidth}
	\includegraphics[width=0.9\textwidth]{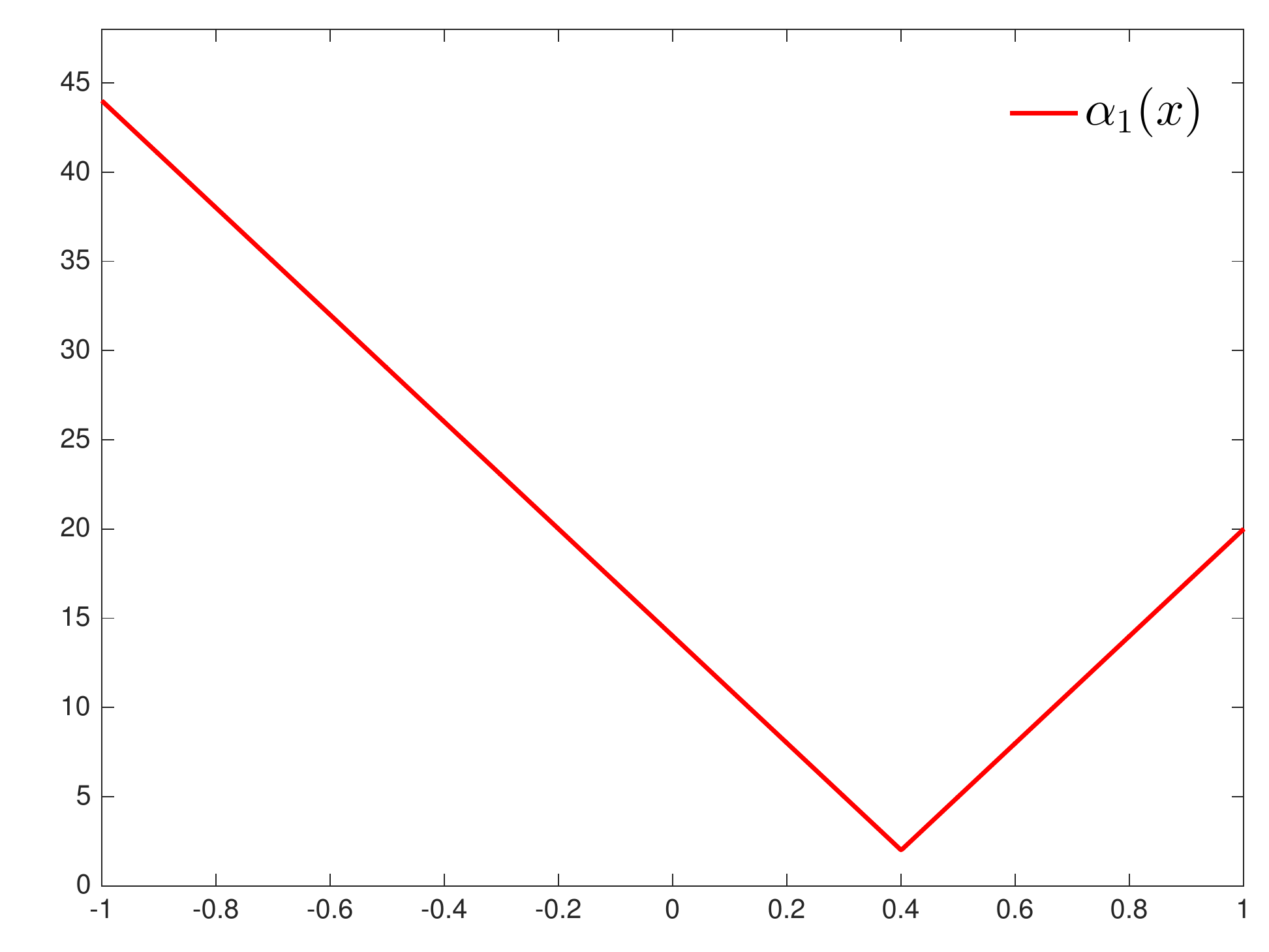}
	\caption{Regularisation of the initial data with a non-constant weight $\alpha_{1}(x)$}
	\label{fig:sg2_alpha2b}
\end{subfigure}
\begin{subfigure}[t]{0.24\textwidth}
	\centering
	\captionsetup{width=.9\linewidth}
	\includegraphics[width=0.9\textwidth]{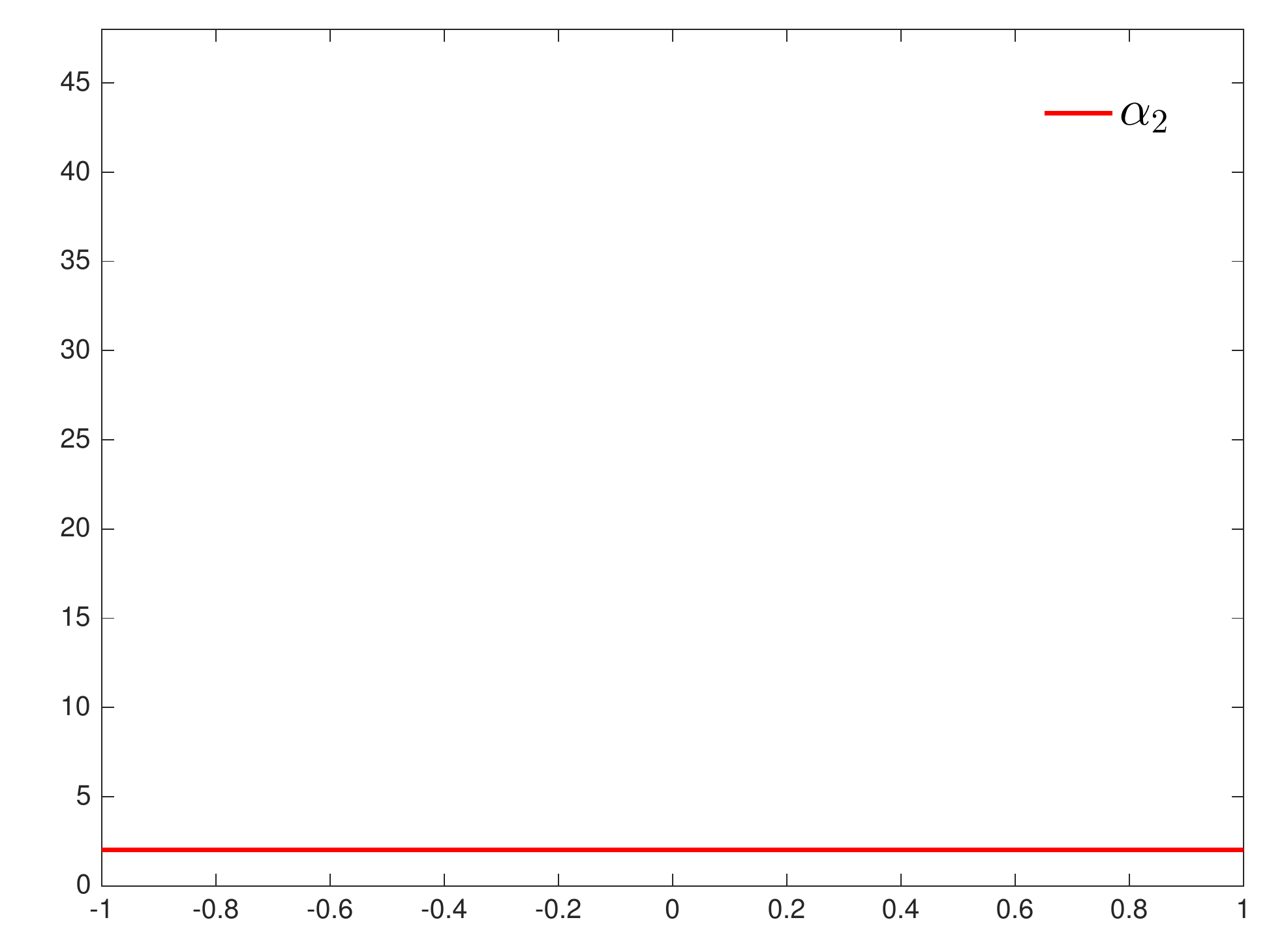}
	\caption{Regularisation of the result in Figure \ref{fig:sg2_alpha2b} with a constant weight $\alpha_{2}$}
	\label{fig:sg2_alpha1b}
\end{subfigure}
\begin{subfigure}[t]{0.24\textwidth}
	\centering
	\captionsetup{width=.9\linewidth}
	\includegraphics[width=0.9\textwidth]{sg2_alpha2_scalar.eps}
	\caption{Regularisation of the initial data with a constant weight $\alpha_{2}$}
	\label{fig:sg2_alpha1a}
\end{subfigure}
\begin{subfigure}[t]{0.24\textwidth}
	\centering
	\captionsetup{width=.9\linewidth}
	\includegraphics[width=0.9\textwidth]{sg2_alpha1x.eps}
	\caption{Regularisation of the result in Figure \ref{fig:sg2_alpha1a} with a non-constant weight $\alpha_{1}(x)$}
	\label{fig:sg2_alpha2a}
\end{subfigure}

\begin{subfigure}[t]{0.35\textwidth}
	\centering
	\captionsetup{width=2\linewidth}
	\includegraphics[width=0.9\textwidth]{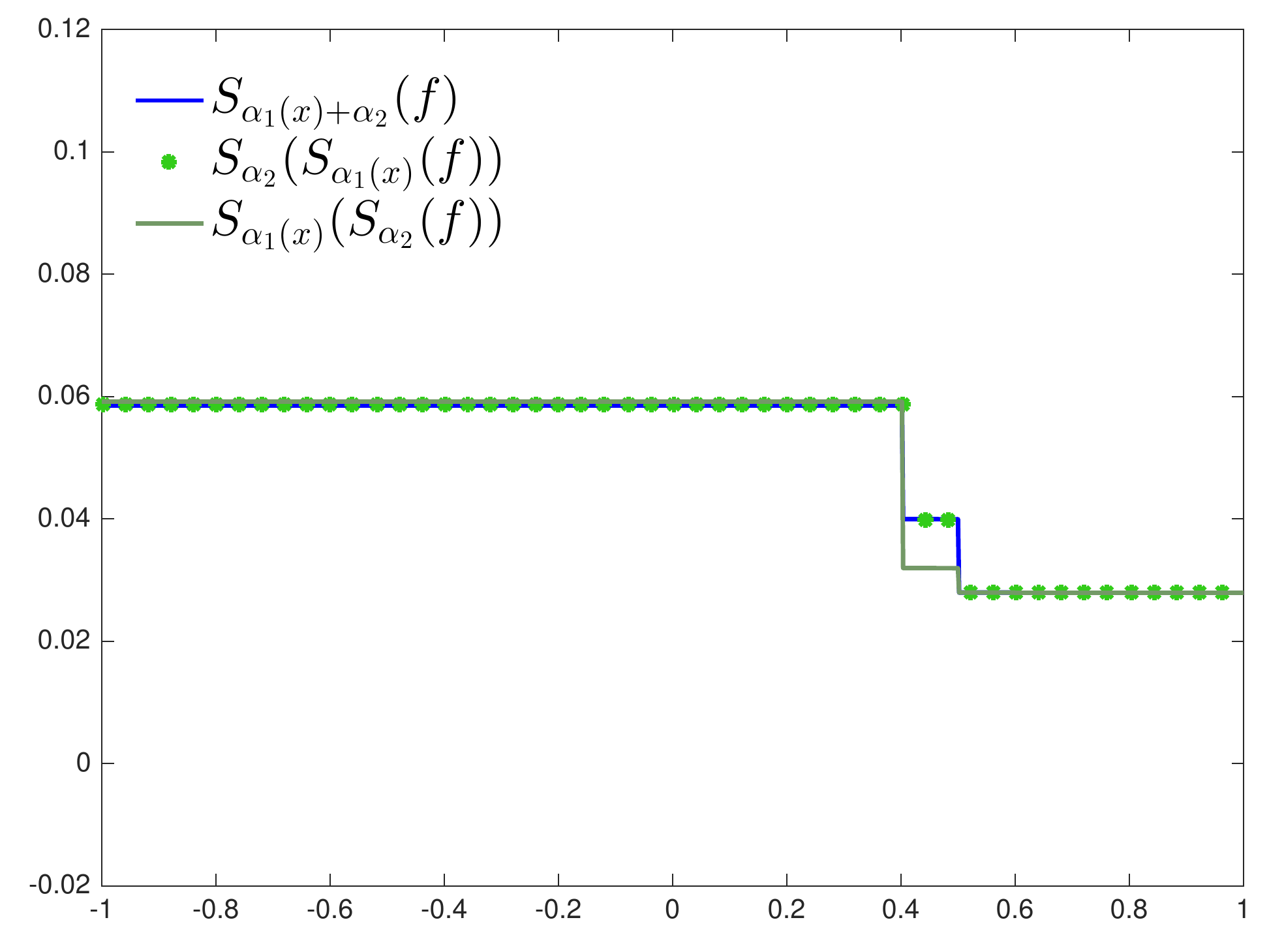}
	\caption{The result $S_{\alpha_{1}(x)+\alpha_{2}}(f)$ obtained by solving \eqref{weighted_rof} using as weight the sum of a non-constant $\alpha_{1}(x)$ and a constant  function $\alpha_{2}$ is not also obtained by successively solve \eqref{weighted_rof} using the constant and then with the non-constant weight, $S_{\alpha_{1}(x)}\left (S_{\alpha_{2}}(f)\right )$ but can be obtained by doing so, first with the non-constant and then with the constant weight $S_{\alpha_{2}}\left (S_{\alpha_{1}(x)}(f)\right )$, as Proposition \ref{lbl:semigroup} dictates}
	\label{fig:non_commute}
\end{subfigure}
\caption{Non-commutativity of the semigroup property even when one of the weights is constant}
\label{fig:sg2}
\end{center}
\end{figure}

We provide a numerical example in Figure \ref{fig:sg2}.
In Figures \ref{fig:sg2_alpha2b}--\ref{fig:sg2_alpha1b}, we display the result which is obtained when we first regularise with a non-constant weight $\alpha_{1}(x)$ and  then with a constant $\alpha_{2}$. Figures \ref{fig:sg2_alpha1a}--\ref{fig:sg2_alpha2a}, on the other hand, show the result when first regularising with $\alpha_{2}$ and then with $\alpha_{1}(x)$. Figure \ref{fig:non_commute} confirms that \eqref{non_commute} holds for this example.


\subsection{Analytic solutions for simple data and weight functions}\label{sec:exact}

Next we compute several analytic solutions for the problem \eqref{weighted_rof}. We get a further intuition about the structure of solutions and how it changes with respect to different weight functions. 

Suppose $\om=(-L,L)$, where $L>0$, and
\begin{equation}\label{data_affine}
f(x)=\lambda x,\quad x\in (-L,L),\quad \lambda>0.
\end{equation}
We also consider an absolute value type function as weight, i.e., 
\begin{equation}\label{weight_abs}
\alpha(x)=\mu|x| +c,\quad x\in(-L,L),\quad \mu,c>0.
\end{equation}

Note that the symmetry of $f$ and $\alpha$ imply that the solution $u$ is also symmetric. Thus it suffices to describe the solution in the interval $[0,L)$. According to Proposition \ref{lbl:boundary_constant}, $u$ will be constant at the right boundary of $[0,L)$, say at the interval $(x_{0},L)$. Assuming that the solution is strictly increasing on $(0,x_{0}]$,  the optimality conditions \eqref{opt1}--\eqref{opt2} yield $-\alpha'(x)=f-u$. Note that  according to Proposition \ref{lbl:lipschitz_alpha}, in this case,  $u$ will have a jump discontinuity at $x=0$.
  Thus, we study the ansatz
\[
u(x)=
\begin{cases}
\lambda x+\mu, & \text{if}\quad x\in (0,x_{0}],\\
\lambda x_{0}+\mu, & \text{if}\quad x\in (x_{0},L).
\end{cases}
\]
From the optimality conditions \eqref{opt1}--\eqref{opt2} it follows that $v$ has the form
\[
v(x)=
\begin{cases}
-\alpha(x),& \text{if}\quad x\in (0,x_{0}],\\
\frac{1}{2}\lambda x^{2}-(\lambda x_{0}+\mu)x+d,  & \text{if}\quad x\in (x_{0},L),
\end{cases}
\]
for some constant $d$. From the continuity of $v$ at $x_{0}$ and $v\in H_{0}^{1}(\om)$, we infer
\begin{align}
\lim_{x\to x_{0}+}v(x)&=-\alpha(x_{0}),\label{cond_v_1}\\
\lim_{x\to L}v(x)&=0\label{cond_v_2}.
\end{align}
Using conditions \eqref{cond_v_1}--\eqref{cond_v_2}, one computes
\begin{equation}\label{x0}
x_{0}=L-\frac{\sqrt{2\lambda \mu L+2\lambda c}}{\lambda}.
\end{equation}
Notice that $|v(x)|\le\ \alpha(x)$ is  satisfied for all $x\in(0,L)$. Finally, we need $0<x_{0}$. This holds if and only if
\begin{equation}\label{cond_sol_1}
\mu L+ c<\frac{\lambda L^{2}}{2}.
\end{equation}

Next we examine the case $x_{0}=0$, that is the solution is constant in $(0,L)$ but nonetheless has  a jump discontinuity at $x=0$. Hence, our ansatz here is
\[u(x)=M, \quad x\in (0,L), \quad M>0.\] 
Now the variable $v$ is of the form
\[v(x)=\frac{1}{2}\lambda x^{2} -Mx+d,\quad x\in (0,L).\]
From the fact that $u$ has a jump discontinuity at $x=0$ (note that the jump must be positive) and  $v\in H_{0}^{1}(\om)$, we infer
\[
v(0)=-\alpha(0),\quad
\lim_{x\to L}v(x)=0,
\]
which in turn give
$M=\frac{\lambda L}{2}-\frac{c}{L}$.
Since we require $M>0$ we must have
\begin{equation}\label{cond_sol_2a}
c<\frac{\lambda L^{2}}{2}.
\end{equation}
Finally, in order to guarantee that $|v(x)|\le\ \alpha(x)$ for all $x$, it suffices to enforce $v'(0)\ge-\mu$, i.e.,
\begin{equation}\label{cond_sol_2b}
\mu L+c\ge \frac{\lambda L^{2}}{2}.
\end{equation}
Note that the inequalities \eqref{cond_sol_2a}--\eqref{cond_sol_2b} form necessary and sufficient conditions for the occurrence of this kind of solution.

The last alternative for the solution is to be constant, equal to the mean value of the data, i.e., $u=0$. Working similarly as in the previous cases we deduce that this solution occurs if and only if 
\begin{equation}\label{cond_sol_3}
c\ge \frac{\lambda L^{2}}{2}.
\end{equation}

Note that the conditions \eqref{cond_sol_1}, \eqref{cond_sol_2a}--\eqref{cond_sol_2b} and \eqref{cond_sol_3} define a partition of the quadrant $\{\mu\ge 0,\;c\ge 0\}$.
 We summarise our findings in the following proposition.

\newtheorem{exact_affine_abs}[alpha_sgn]{Proposition}
\begin{exact_affine_abs}\label{lbl:exact_affine_abs}
Let $\Omega=(-L,L)$, $f(x)=\lambda x$ and $\alpha(x)=\mu|x|+c$ with $L,\lambda, \mu,c>0$. Then the solution $u$ of the problem
\[\min_{u\in\bv(\om)} \frac{1}{2} \int_{\om} (f-u)^{2}dx+\int_{\om}\alpha(x)d|Du|,\]
is given by the following formulae:
\begin{alignat*}{3}
&\text{If}\quad \mu L+c<\frac{\lambda L^{2}}{2},&\qquad
&u(x)=
\begin{cases}
-\lambda x_{\mu,c}-\mu,& \text{if}\quad x\in (-L,-x_{\mu,c}],\\
-\lambda x-\mu,		  &  \text{if}\quad x\in (-x_{\mu,c},0),\\
\lambda x +\mu, 		  &  \text{if}\quad x\in [0,x_{\mu,c}),\\
\lambda x_{\mu,c}+\mu,& \text{if}\quad x\in [x_{\mu,c},L),
\end{cases}\\
&\text{where}\quad x_{\mu,c}=L-\frac{\sqrt{2\lambda \mu L+2\lambda c}}{\lambda}.\\
&\text{If}\quad \mu L+c\ge\frac{\lambda L^{2}}{2} \; \& \; c<\frac{\lambda L^{2}}{2}, &\quad
&u(x)=
\begin{cases}
-\frac{\lambda L}{2} +\frac{c}{L}, & \text{if}\quad x\in (-L,0),\\
\frac{\lambda L}{2} -\frac{c}{L},& \text{if}\quad x\in [0,L).
\end{cases}\\
&\text{If}\quad c\ge\frac{\lambda L^{2}}{2},&\quad
&u(x)=0,\quad x\in (-L,L).
\end{alignat*}
\end{exact_affine_abs}

\begin{figure}[t!]
\begin{center}
\includegraphics[width=0.7\textwidth]{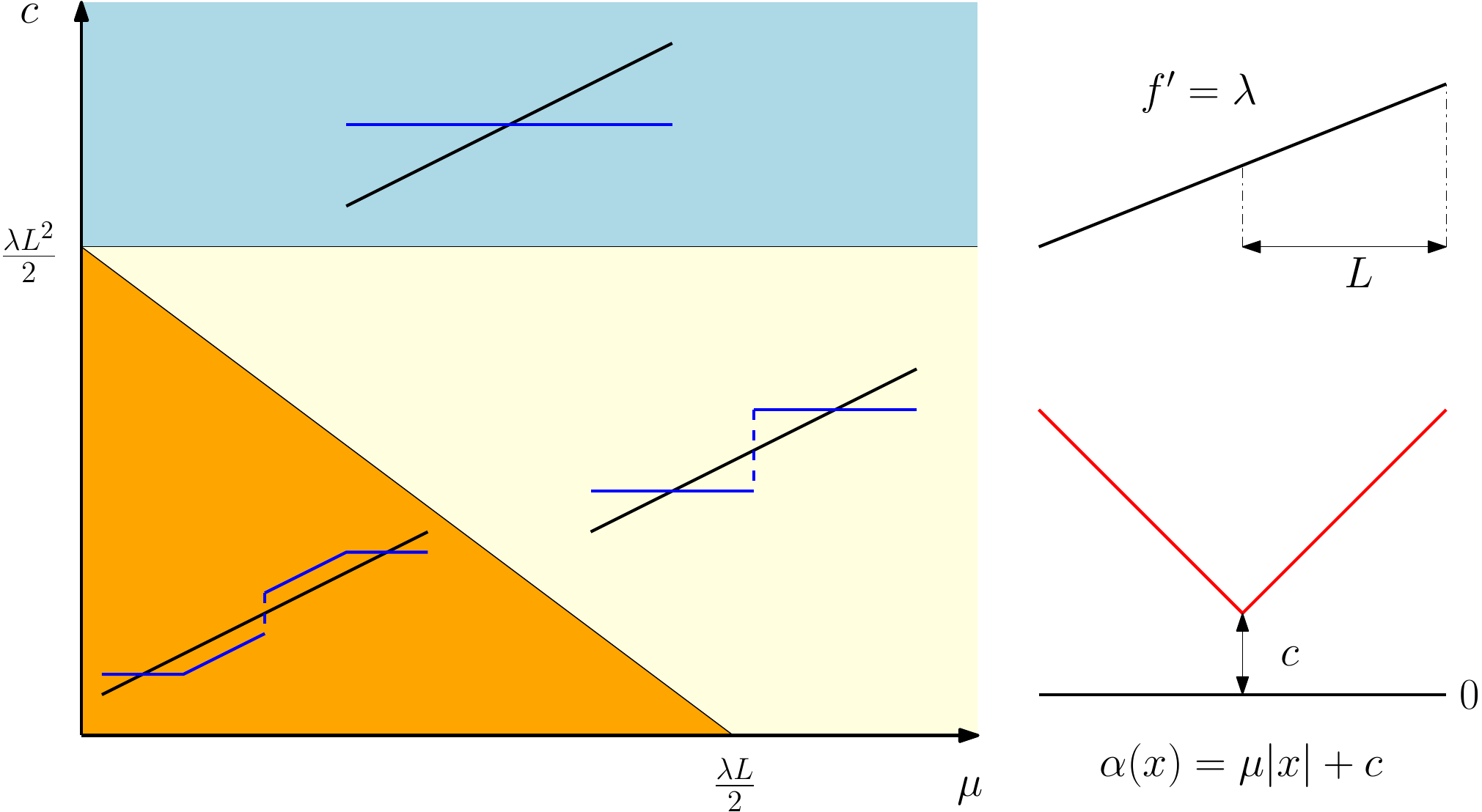}
\caption{Different types of solutions of the weighted $\mathrm{TV}$ minimisation problem \eqref{weighted_rof} with data $f(x)=\lambda x$ in $(-L,L)$ and weight function $\alpha(x)=\mu|x|+c$}
\label{fig:affine_abs}
\end{center}
\end{figure}

In Figure \ref{fig:affine_abs} we depict the different areas of the parameter space of the absolute type weight function, $\{\mu\ge 0,\;c\ge 0\}$ that correspond to the different types of solutions.

\begin{figure}[t!]
\begin{center}
\begin{subfigure}[t]{0.48\textwidth}
	\centering
	\includegraphics[width=0.98\textwidth]{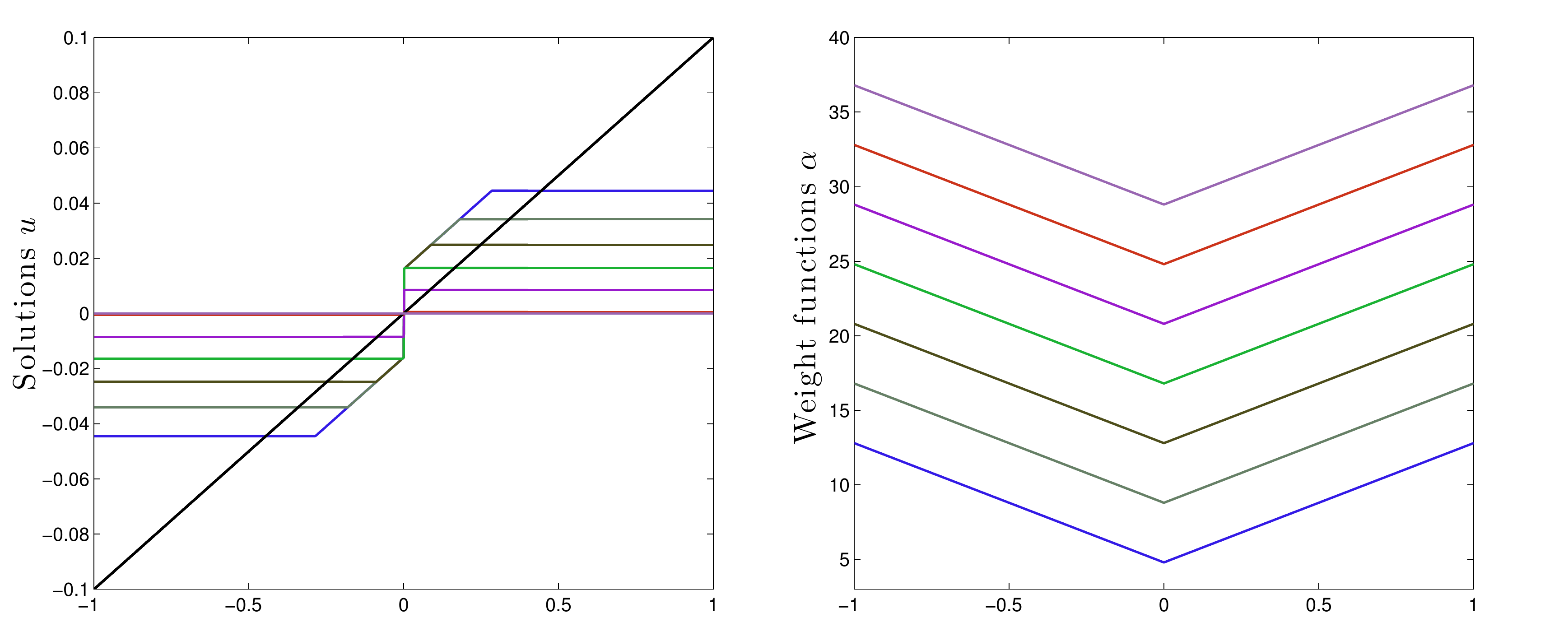}
	\caption{Varying $c$ by keeping $\mu$ fixed}
	\label{fig:increasing_c}
\end{subfigure}
\begin{subfigure}[t]{0.48\textwidth}
	\centering
	\includegraphics[width=0.98\textwidth]{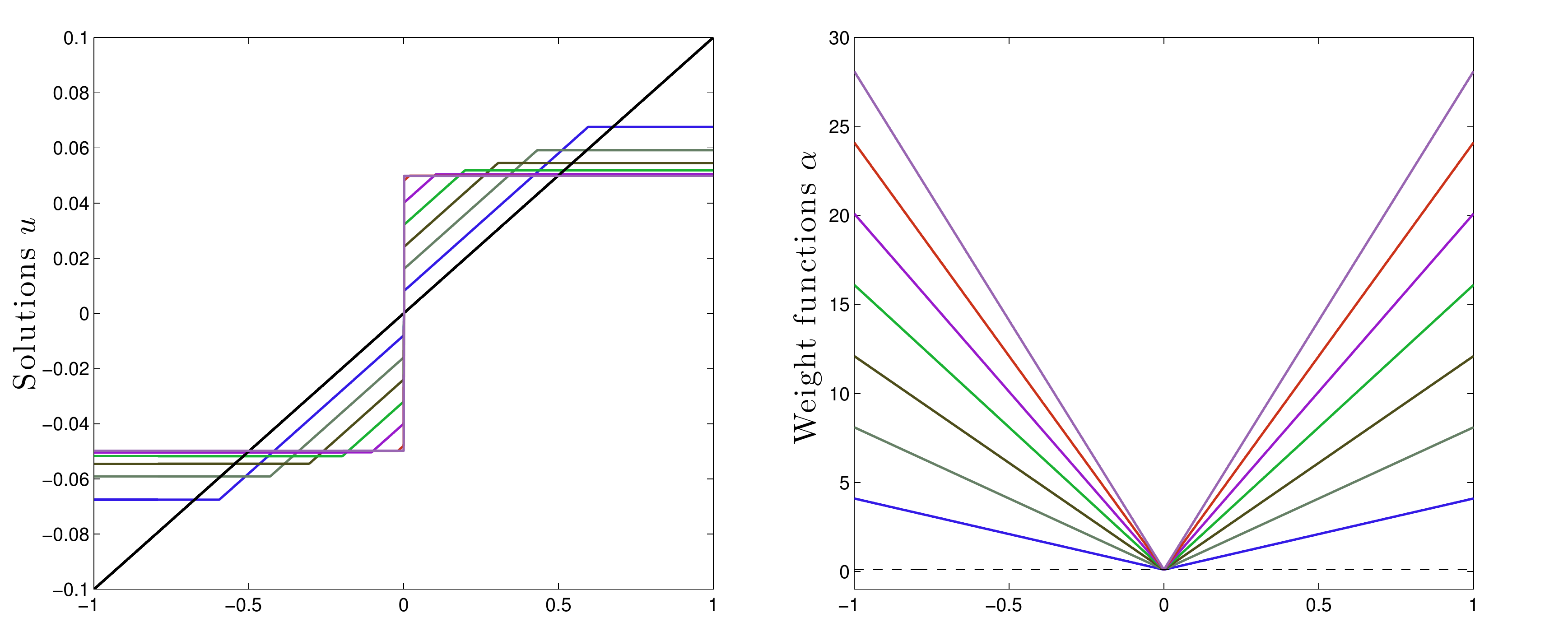}
	\caption{Varying $\mu$ by keeping $c$ fixed}
	\label{fig:increasing_mu}
\end{subfigure}
\caption{Numerical solutions of the weighted $\tv$ minimisation \eqref{weighted_rof} with data $f(x)=\lambda x$ in $(-L,L)$ and weight functions of the type $\alpha(x)=\mu|x|+c$}
\label{fig:affine_increasing}
\end{center}
\end{figure}

In Figure \ref{fig:affine_increasing} we summarise  numerical examples to better understand the solution dependence on the weight parameters $c$ and $\mu$. In Figure \ref{fig:increasing_c},   keeping $\mu$ fixed, we vary the parameter $c$ by adding constants. Among others, one can observe here the numerical verification of the partial semigroup property of Proposition \ref{lbl:semigroup}. The first weight (light blue line) produces a result, say $u_{1}$, that has a new discontinuity at the origin. As we increase the weight function by a constant the solutions we obtain are the corresponding solutions of the  scalar $\tv$ problem with data $u_{1}$ and the parameter being the very constant.   

In Figure \ref{fig:increasing_mu}, by keeping $c$ fixed, we vary the parameter $\mu$ by increasing the slope of the weight function. Notice that as the slope increases, i.e., $D\alpha'(\{0\})$ increases, so does the jump discontinuity of the solution $u$. This is in accordance with Proposition \ref{lbl:lipschitz_alpha}, i.e., we have $Du(\{x\})=D\alpha'(\{0\})$, since at least when $\mu$ is not too large, a whole neighbourhood of $x=0$ belongs to $\mathrm{supp}(|Du|)$. This behaviour stops when $\mu$, and thus $\alpha'$, becomes too large (light purple line). Then the jump size of the solutions stops growing, no matter the size of $\mu$ and it depends only on the value of $c$, cf. Propositions \ref{lbl:large_gradient_plateau} and \ref{lbl:exact_affine_abs}. 

\section{Bound of the total variation of the solution by the total variation of the data}\label{sec:boundTV}

It is a standard result in the scalar total variation minimisation that the total variation of the solution is bounded by the total variation of data, i.e., if $\alpha$ is a positive constant and $u$ solves 
\[\min_{u\in\bv(\om)}\frac{1}{2}\int_{\om}(f-u)^{2}dx+\alpha |Du|(\om),\]
then
\begin{equation}\label{tvu_tvf}
|Du|(\om)\le |Df|(\om).
\end{equation}
This is simply proven by comparing the energy of the minimiser $u$ and the data $f$, that is
\[\alpha |Du|(\om)\le \frac{1}{2}\int_{\om}(f-u)^{2}dx+\alpha |Du|(\om)\le \alpha |Df|(\om).\]
Note that this holds for any dimension.
A similar argument for the weighted case would only give
\begin{equation}\label{tvu_tvf_weighted}
\int_{\om}\alpha(x)d|Du|\le \int_{\om}\alpha(x)d|Df| \quad \Rightarrow\quad |Du|(\om)\le \frac{\max_{x\in\om} \alpha(x)}{\min_{x\in\om}\alpha(x)} |Df|(\om).
\end{equation}
However, the estimate in \eqref{tvu_tvf_weighted} is not satisfactory since the constant in the right-hand side depends on the weight $\alpha$ and in fact blows up as $\alpha$ tends to zero. In this section, we show that the estimate \eqref{tvu_tvf} holds in the weighted case as well. At first glance, \eqref{tvu_tvf} is counterintuitive as weighted $\tv$ may create new discontinuities, thus, increasing the variation locally. Our proof here uses fine scale analysis of the structure of solutions of the weighted $\tv$ problem and thus we are only able to show this result in dimension one.

We start by showing that weighted $\tv$ does not introduce oscillations in the solution $u$.

\newtheorem{few_oscillations}[alpha_sgn]{Proposition}
\begin{few_oscillations}\label{lbl:few_oscillations}
Let $u\in\bv(\om)$ solve the weighted $\tv$ problem \eqref{weighted_rof} with data $f\in \bv(\om)$ and a differentiable weight function $\alpha\in C(\overline{\om})$ with $\alpha>0$.  Then $\{\underline{u}>\underline{f}\}\cap\{ \overline{u}>\overline{f}\}$ is an open set and hence it is the countable union of disjoint intervals $I_{n}$. In each one of these intervals $I_{n}$, the solution $u$ has at most one initial decreasing part, followed by at most one increasing part. The analogous result holds for the set $\{\underline{u}<\underline{f}\}\cap\{ \overline{u}<\overline{f}\}$. There the solution $u$ has at most one initial increasing part, followed by an at most one decreasing part.
\end{few_oscillations}

\begin{proof}
Observe first that
\begin{equation}\label{few_oscillations_set_equi}
\{\underline{u}>\underline{f}\}\cap\{ \overline{u}>\overline{f}\}=\{\underline{u}>\overline{f}\}\cup \{\underline{f}<\underline{u}\le \overline{f}<\overline{u}\}.
\end{equation}

The fact that the set $\{\underline{u}>\overline{f}\}$ is open is shown in \cite{BrediesL1}. Let  now $x\in\{\underline{f}<\underline{u}\le \overline{f}<\overline{u}\}$ and suppose without loss of generality that $Df(\{x\})>0$. Then, $Du(\{x\})>0$ as well, since $\alpha$ is differentiable, cf. Proposition \ref{lbl:correct_jump}. This means that
\[f^{l}(x)<u^{l}(x)\le f^{r}(x)<u^{r}(x).\]
Bearing in mind that $f^{l}(t)=c_{1}+Df((a,t))$, $u^{l}(t)=c_{2}+Du((a,t))$, from the left continuity of $f^{l}$ and $u^{l}$ it follows that there exists an $\epsilon>0$ and two real numbers $m<M$ such that
\begin{equation}\label{left_above}
c_{1}+Df((a,t))<m<M<c_{2}+Du((a,t)),\quad \text{for all } t\in (x-\epsilon,x).
\end{equation}
But since $Df((\alpha,t_{0}])=\lim_{t\to t_{0}^{+}}Df((\alpha,t))$, and analogously for $u$, we have that
\begin{equation}\label{right_above}
c_{1}+Df((a,t])\le m<M \le c_{2}+Du((a,t]),\quad \text{for all } t\in (x-\epsilon,x),
\end{equation}
also holds and hence every $t$ in $(x-\epsilon,x)$ also belongs to  $\{\underline{u}>\overline{f}\}$. Similarly we show that there exists an $\epsilon>0$ such that every $t$ in $(x,x+\epsilon)$ also belongs to  $\{\underline{u}>\overline{f}\}$ and hence  from \eqref{few_oscillations_set_equi} the set $\{\underline{u}>\underline{f}\}\cap\{ \overline{u}>\overline{f}\}$ is open. Thus it can be written as a countable union of disjoint intervals $I_{n}$.

 We focus on a single interval $I_{n}$. From the optimality conditions \eqref{opt1}--\eqref{opt2} we have that $v'<0$ almost everywhere in $I_{n}$ and thus $v$ is strictly decreasing there. Suppose now that $u$ has an increasing part in $I_{n}$ followed by a decreasing part. Then there would exist points $x_{1}$ and $x_{2}$ in $I_{n}$ with $x_{1}<x_{2}$ and
\[\sgn(Du)(x_{1})=1\quad \text{and}\quad \sgn(Du)(x_{2})=-1.\]
However, again from the optimality conditions one obtains
\[v(x_{1})=-\alpha(x_{1})<0\quad \text{and} \quad v(x_{2})=\alpha(x_{2})>0,\]
which is a contradiction since $v$ is decreasing in $(x_{1},x_{2})\subseteq I_{n}$.

\end{proof}

We are now ready to prove the main theorem of the section. Note that we will assume for the time being that the weight function $\alpha$ is differentiable. In Section \ref{sec:alpha_zero} we will extend this result to continuous weights $\alpha$ via a $\Gamma$-convergence argument.

\newtheorem{TVu_less_TVf}[alpha_sgn]{Theorem}
\begin{TVu_less_TVf}[$|Du|(\om)\le |Df|(\om)$, for differentiable $\alpha$]\label{lbl:TVu_less_TVf}
Let $\om=(a,b)\subseteq\RR$, $\alpha\in C(\overline{\om})$ differentiable with $\alpha>0$ and $f\in\bv(\om)$. If
\[u=\underset{u\in\bv(\om)}{\operatorname{argmin}}\; \frac{1}{2}\int_{\om}(f-u)^{2}dx+\int_{\om}\alpha(x)d|Du|,\]
then 
\[|Du|(\om)\le |Df|(\om).\]
\end{TVu_less_TVf}

\begin{proof}

Since the weight function $\alpha$ is differentiable we have that, according to Proposition \ref{lbl:jump_set_incl}, the jump discontinuities of the solution $u$ are contained in the ones of the data $f$. Moreover, $\{\underline{u}\le \underline{f}<\overline{f}<\overline{u}\}\cup\{\underline{u}< \underline{f}<\overline{f}\le\overline{u}\} $ is empty; cf. Proposition \ref{lbl:correct_jump}. As $J_{f}$ is at most countable we can write
\begin{align}
|Du|(\om)&=|Du|(\{\underline{u}>\underline{f}\}\cap \{\overline{u}>\overline{f}\})+|Du|(\{\underline{u}<\underline{f}\}\cap \{ \overline{u}<\overline{f}\})\nonumber\\
&\;\;\;+|Du|(\{\underline{f}\le \underline{u}<\overline{u}\le \overline{f}\})+|Du|(\{\underline{u}=\overline{u}=\underline{f}=\overline{f}\}).\label{tv_break_down}
\end{align}
Notice that the sets that appear in \eqref{tv_break_down} are disjoint. We focus first on the set $\{\underline{u}>\underline{f}\}\cap \{\overline{u}>\overline{f}\}$. According to Proposition \ref{lbl:few_oscillations}, this can be written as a countable  union of disjoint open intervals. Let $(x_{1},x_{2})$ be one of these intervals and assume for the moment that $a<x_{1}<x_{2}<b$. Then we have that $(x_{1},x_{2})$ is maximal in the sense that in the endpoints $x_{1}$ and $x_{2}$, at least one of the conditions $\underline{u}(x)>\underline{f}(x)$, $\overline{u}(x)>\overline{f}(x)$ does not hold. Recall from Proposition \ref{lbl:few_oscillations} that in this interval, $u$  starts with a decreasing part followed by an increasing one (both not necessarily strict). In what follows, we will infer bounds on $|Du|((x_{1},x_{2}))$ by considering different alternative cases.

(i) Consider first that $f$, and hence $u$, is continuous at both endpoints $x_{1}$ and $x_{2}$. From the maximality of the associated interval we have
\[\underline{u}(x_{1})=\overline{u}(x_{1})=\underline{f}(x_{1})=\overline{f}(x_{1})\quad \text{and}\quad\underline{u}(x_{2})=\overline{u}(x_{2})=\underline{f}(x_{2})=\overline{f}(x_{2}).\]
Then from the monotonicity structure of $u$ and the fact that in this interval $\essinf_{t\in(x_{1},x_{2})} f(t)<\essinf_{t\in(x_{1},x_{2})} u(t)$ we have that
\[|Du|((x_{1},x_{2}))\le |Df|((x_{1},x_{2})).\]

(ii) The second case is that $f$ is continuous at $x_{1}$ and has a jump discontinuity at $x_{2}$. Note that this jump must be a positive one as otherwise the condition $\underline{u}>\overline{f}$ would be violated inside the interval, unless we have $\overline{u}(x_{2})=\overline{f}(x_{2})$ in which case the estimate $|Du|((x_{1},x_{2}))\le |Df|((x_{1},x_{2}))$  holds. Notice also that due to the maximality of $x_{2}$ we cannot have 
$\underline{f}(x)<\underline{u}(x)\le \overline{f}(x)<\overline{u}(x)$.
So we must have
\begin{equation}\label{if_u_jump}
\underline{u}(x_{1})=\overline{u}(x_{1})=\underline{f}(x_{1})=\overline{f}(x_{1})\quad \text{and}\quad\underline{f}(x_{2})\le\underline{u}(x_{2})< \overline{u}(x_{2})\le\overline{f}(x_{2}),
\end{equation}
if $u$ has a jump in $x_{2}$, or 
\begin{equation}\label{if_u_notjump}
\underline{u}(x_{1})=\overline{u}(x_{1})=\underline{f}(x_{1})=\overline{f}(x_{1})\quad \text{and}\quad\underline{f}(x_{2})\le \underline{u}(x_{2})= \overline{u}(x_{2})\le \overline{f}(x_{2}),
\end{equation}
if $u$ does not have a jump in $x_{2}$. Suppose we have \eqref{if_u_jump}, then arguing similarly as before we have that
\[|Du|((x_{1},x_{2}))\le |Df|((x_{1},x_{2}))+\underline{u}(x_{2})-\underline{f}(x_{2}),\]
 i.e., in that case we need to add to the right-hand side also a part of the jump of $f$ at $x_{2}$, which is however below the jump of $u$. If we have \eqref{if_u_notjump} we estimate again
\[|Du|((x_{1},x_{2}))\le |Df|((x_{1},x_{2}))+u(x_{2})-\underline{f}(x_{2}).\]

(iii) The third case is that $f$ has a jump discontinuity at $x_{1}$ and it is continuous at $x_{2}$. This is treated analogously to the second case. 

(iv) The fourth case is that $f$ has jump discontinuities at both points $x_{1}$ and $x_{2}$. The only case the  jump of $f$ at $x_{1}$ is  positive is when $\overline{u}(x_{1})=\overline{f}(x_{1})$ and in that case we have
\[|Du|((x_{1},x_{2}))\le |Df|((x_{1},x_{2}))\quad \text{or} \quad |Du|((x_{1},x_{2}))\le |Df|((x_{1},x_{2}))+\underline{u}(x_{2})-\underline{f}(x_{2}),\]
depending on whether the jump of $f$ at $x_{2}$ is negative or positive. So we are left with the case where the jumps of $f$ at $x_{1}$ and $x_{2}$ are negative and positive, respectively. We then have
\begin{equation}\label{if_u_notjump2}
\underline{f}(x_{1})\le \underline{u}(x_{1})\le \overline{u}(x_{1})\le \overline{f}(x_{1})\quad \text{and} \quad  \underline{f}(x_{2})\le \underline{u}(x_{2})\le \overline{u}(x_{2})\le \overline{f}(x_{2}),
\end{equation}
where we estimate
\[|Du|((x_{1},x_{2}))\le |Df|((x_{1},x_{2}))+ (\underline{u}(x_{1})-\underline{f}(x_{1}))+(\underline{u}(x_{2})-\underline{f}(x_{2})).\]

Note that analogous estimates hold if $x_{1}=a$ or $x_{2}=b$.
By summing over the countable set of disjoint intervals related to $\{\underline{u}>\underline{f}\}\cap\{\overline{u}>\overline{f}\} $, the following inequality holds
\begin{equation}\label{ulessf_ineq}
|Du|(\{\underline{u}>\underline{f}\}\cap\{\overline{u}>\overline{f}\} )\le |Df|(\{\underline{u}>\underline{f}\}\cap\{\overline{u}>\overline{f}\} )+\sum_{x\in\{\underline{f}\le \underline{u}<\overline{u}\le \overline{f}\}} (\underline{u}(x)-\underline{f}(x)).
\end{equation}
This is also due to the fact that any term of the form $\underline{u}(x)-\underline{f}(x))$ appears in only one of the corresponding estimates for the intervals $I_{n}$. Similarly, for $\{\underline{u}<\underline{f}\}\cap \{ \overline{u}<\overline{f}\}$ we obtain
\begin{equation}\label{ulessf_ineq}
|Du|(\{\underline{u}<\underline{f}\}\cap \{ \overline{u}<\overline{f}\})\le |Df|(\{\underline{u}<\underline{f}\}\cap \{ \overline{u}<\overline{f}\})+\sum_{x\in \{\underline{f}\le \underline{u}<\overline{u}\le \overline{f}\}} (\overline{f}(x)-\overline{u}(x)).
\end{equation}
The final estimate then follows:
\begin{align*}
|Du|(\om)&= |Du|(\{\underline{u}>\underline{f}\}\cap\{\overline{u}>\overline{f}\})+|Du|(\{\underline{u}<\underline{f}\}\cap \{ \overline{u}<\overline{f}\})\\
&\;\;\;\;+ |Du|(\{\underline{f}\le \underline{u}<\overline{u}\le \overline{f}\})+|Du|(\{\underline{u}=\overline{u}=\underline{f}=\overline{f}\})\\
               &\le |Df|(\{\underline{u}>\underline{f}\}\cap\{\overline{u}>\overline{f}\})+|Df|(\{\underline{u}<\underline{f}\}\cap \{ \overline{u}<\overline{f}\})\\
               &\;\;\;\; +\sum_{x\in \{\underline{f}\le \underline{u}<\overline{u}\le \overline{f}\}} (\underline{u}(x)-\underline{f}(x))+\sum_{x\in \{\underline{f}\le \underline{u}<\overline{u}\le \overline{f}\}} (\overline{f}(x)-\overline{u}(x))+\sum_{x\in \{\underline{f}\le \underline{u}<\overline{u}\le \overline{f}\}} (\overline{u}(x)-\underline{u}(x))\\
               &\;\;\;\;+|Df|(\{\underline{u}=\overline{u}=\underline{f}=\overline{f}\})\\
               &\le |Df|(\{\underline{u}>\underline{f}\}\cap\{\overline{u}>\overline{f}\})+|Df|(\{\underline{u}<\underline{f}\}\cap \{ \overline{u}<\overline{f}\})\\
               &\;\;\;\; +\sum_{x\in \{\underline{f}\le \underline{u}<\overline{u}\le \overline{f}\}} (\overline{f}(x)-\underline{f}(x))\\
               &\;\;\;\;+ |Df|(\{\underline{u}=\overline{u}=\underline{f}=\overline{f}\})\\
               &\le |Df|(\om).
\end{align*}
Here we also use Lemma \ref{lbl:variation_equal_f_u} to infer
\[|Du|(\{\underline{u}=\overline{u}=\underline{f}=\overline{f}\})=|Df|(\{\underline{u}=\overline{u}=\underline{f}=\overline{f}\}).\]

\end{proof}

It remains to prove the following lemma that was used in the final estimate above.

\newtheorem{variations_equal_f_u}[alpha_sgn]{Lemma}
\begin{variations_equal_f_u}\label{lbl:variation_equal_f_u}
Let $\om=(a,b)$, $u,v\in \bv(\om)$ and let 
\[A=\left\{x\in \om:\; u,v \text{ are continuous at }x \text{ and }u(x)=v(x)\right\}.\]
Then
\[|Du|(A)=|Dv|(A).\]
\end{variations_equal_f_u}

\begin{proof}
Note first that $A$ can also be written as 
\[A=\left\{x\in \om:\; \underline{u}(x)=\overline{u}(x)=\underline{v}(x)=\overline{v}(x)\right\}.\]
We claim that $A$ is a $G_{\delta}$ set, i.e., a countable intersection of open sets. Indeed recall first, that the set of continuity points of a function is a $G_{\delta}$; see for instance \cite{Olmsted}. Therefore, the sets of continuity points of $\underline{u}$ and $\underline{v}$, i.e., $J_{u}^{c}$ and $J_{v}^{c}$, respectively, are $G_{\delta}$ sets. Note that $A$ can also be written as
\begin{equation}\label{A1}
A=J_{u}^{c}\cap J_{v}^{c}\cap \{x\in \om:\; \underline{u}(x)=\underline{v}(x)\}.
\end{equation}
Observe that $\underline{u}$, $\underline{v}$ as well as their difference $w=\underline{u}-\underline{v}$ are lower semicontinuous functions. In particular, for any $c\in\RR$ the sets $w^{-1}((c,\infty))$ and $w^{-1}((-\infty,c])$ are open and closed, respectively. Thus we have
\begin{align*}
\{x\in \om:\; \underline{u}(x)=\underline{v}(x)\}&=w^{-1}(\{0\})\\
						&= \underbrace{w^{-1} \left ((-\infty,0] \right )}_{\text{closed}}\cap \underbrace{\bigcap_{n\in\NN} w^{-1} \left(\left (-\frac{1}{n},\infty \right ) \right)}_{G_{\delta}},
\end{align*}
and hence the set $\{x\in \om:\; \underline{u}(x)=\underline{v}(x)\}$ is $G_{\delta}$ as an intersection of $G_{\delta}$ sets (recall here that in any metric space every closed set $F$ is $G_{\delta}$ since $F=\bigcap_{n=1}^{\infty}\{x:\;\mathrm{dist}(x,F)<\frac{1}{n}\}$). Hence from \eqref{A1} we have that $A$ is $G_{\delta}$ as well. Likewise, if $I\subseteq\om$ is an open interval then $I\cap A$ is also $G_{\delta}$.
It is now convenient to consider $I\cap A$ without its potential isolated points \footnote{ Here by isolated points, we mean that there exists $\epsilon>0$ such that $(\alpha_{n},\alpha_{n}+\epsilon)$ or $(\alpha_{n}-\epsilon,\alpha_{n})$  does not intersect $I\cap A$.}
 $\{a_{n}\}_{n\in\NN}$. 
  Thus we define
\[\tilde{A}=(I\cap A)\cap (\{a_{n}\}_{n\in\NN})^{c},\]
which remains a $G_{\delta}$ set as an intersection of two $G_\delta$ sets. Hence, we can write 
\[\tilde{A}=\bigcap_{n\in\NN}A_{n},\quad  A_{n}=\bigcup_{k\in\NN} I_{n}^{k},\]
where the sequence $(A_{n})_{n\in\NN}$ can be chosen to be decreasing and with each $\{I_{n}^{k}\}_{k\in\NN}$ being a disjoint family of open intervals.

We are now ready to proceed with the main part of the proof. For every $x_{1}, x_{2}\in \tilde{A}$ with $x_{1}<x_{2}$ we have the following equalities
\begin{align*}
Du((x_{1},x_{2}))=u(x_{2})-u(x_{1})=
v(x_{2})-v(x_{1})=Dv((x_{1},x_{2})).
\end{align*}
We fix an interval $I_{n}^{k}$ and claim that $Du(I_{n}^{k})=Dv(I_{n}^{k})$. Since we have removed all the isolated points from $A$ (left and right ones), we can assume, by potentially making $I_{n}^{k}$ smaller, that its endpoints can be approximated by points of $\tilde{A}$, i.e.,
\[I_{n}^{k}=\bigcup_{i\in\NN} (x_{1}^{i},x_{2}^{i}),\quad x_{1}^{i}, x_{2}^{i}\in \tilde{A},\quad\text{with}\quad (x_{1}^{i},x_{2}^{i})\subseteq (x_{1}^{i+1},x_{2}^{i+1}),\;\;i\in\NN.\]
This implies that 
\[Du(I_{n}^{k})=Du\left ( \bigcup_{i\in\NN} (x_{1}^{i},x_{2}^{i})\right )=\lim_{i\to\infty} Du((x_{1}^{i},x_{2}^{i}))=\lim_{i\to\infty} Dv((x_{1}^{i},x_{2}^{i}))=Dv\left ( \bigcup_{i\in\NN} (x_{1}^{i},x_{2}^{i})\right )=Dv(I_{n}^{k}).\]
It then follows that for every $n\in \NN$
\[Du(A_{n})=\sum_{k\in \NN}Du(I_{n}^{k})=\sum_{k\in \NN}Dv(I_{n}^{k})=Dv(A_{n}),\]
and from continuity of the measures $Du$ and $Dv$, since $(A_{n})_{n\in\NN}$ is decreasing, we have
\[Du(\tilde{A})=\lim_{n\to\infty} Du(A_{n})=\lim_{n\to\infty} Dv(A_{n})=Dv(\tilde{A}).\]
Finally from the fact that the isolated points of $I\cup A$, $\{a_{n}\}_{n\in\NN}$ are at most countable and $u,v$ are both continuous there, we have that 
\[Du\left (\bigcup_{n\in\NN}\{a_{n}\} \right )=Dv\left (\bigcup_{n\in\NN}\{a_{n}\} \right )=0,\]
 and hence
\[Du(I\cap A)=Du(\tilde{A})=Dv(\tilde{A})=Dv(I\cap A),\]
with this equality being true for any open interval $I\subseteq \om$. This implies that $|Du|(A)=|Dv|(A)$.
\end{proof}

\section{Weighted total variation  with vanishing weight function $\alpha$}\label{sec:alpha_zero}

In this section we are considering the problem \eqref{weighted_rof} in the case where the weight function $\alpha\in C(\overline{\om})$ can also have zero values, i.e., $\alpha\ge 0$. This is motivated by the fact that a vanishing weight function imposes locally no regularisation. In terms of image processing this is useful when one wants the reconstructed image to be equal to the data in some areas. Moreover, as we analyse in Section \ref{exact_recon}, a vanishing weight function can result to a better preservation of contrast and exact reconstruction of piecewise constant data.

\subsection{The well-posedness}\label{sec:wellposed_alpha_zero}
Concerning existence of solutions to \eqref{weighted_rof} with $\alpha\ge 0$, 
 note that a simple application of the direct method of calculus of variations  fails here,  due to the absence of coercivity of the objective functional. 
We overcome this complication by employing a $\Gamma$-convergence argument. First we focus on the lower semicontinuity of the weighted $\tv$ functional. 

\newtheorem{lsc_a0}[alpha_sgn]{Proposition}
\begin{lsc_a0}\label{lbl:lcs_a0}
Let $\alpha\in C(\overline{\om})$ with $\alpha\ge 0$. Then the map $u\mapsto \int_{\om}\alpha(x)d|Du|$ defined on $\bv(\om)$ is lower semicontinuous with respect to the weak$^{\ast}$ convergence in $\bv(\om)$. 
\end{lsc_a0}

\begin{proof}
Let $(u_{n})_{n\in\NN}$ be a sequence converging to $u$ weakly$^{\ast}$ in $\bv(\om)$. In particular,  $Du_{n}\to Du$ weakly$^\ast$ in the space of Radon measures, i.e.,
\[\lim_{n\to\infty} \int_{\om} v \, dDu_{n}\to \int_{\om} v\, dDu,\quad \text{for every }v\in C_{0}(\om). \]
This implies that
\[\lim_{n\to\infty} \int_{\om} v \alpha\, dDu_{n}\to \int_{\om} v\alpha\, dDu,\quad \text{for every }v\in C_{0}(\om), \]
since $v\alpha\in C_{0}(\om)$ for every $v\in C_{0}(\om)$. Thus, we also have $\alpha Du_{n}\to \alpha Du$ weakly$^\ast$ in the space of Radon measures. From the lower semicontinuity of the total variation map with respect to the weak$^{\ast}$ convergence in the space of Radon measures we get
\[\int_{\om}\alpha (x) d|Du|=\left | \alpha Du \right |(\om)\le \liminf_{n\to\infty} \left |\alpha Du_{n} \right |(\om)=\liminf_{n\to\infty} \int_{\om}\alpha(x) d|Du_{n}|.\]
 
\end{proof}

 We now proceed with  the following $\Gamma$-convergence result.

\newtheorem{gamma_conv_F}[alpha_sgn]{Proposition}
\begin{gamma_conv_F}\label{lbl:gamma_conv_F}
Let $\alpha\in C(\overline{\om})$ with $\alpha\ge 0$ and $(\alpha_{n})_{n\in\NN}\subseteq C(\overline{\om})$ being a decreasing sequence with $\alpha_{n}\ge 0$ for every $n\in\NN$ and $\alpha_{n}\to \alpha$, uniformly.
Defining $F,F_{n}:\bv(\om)\to \RR$, by
\[
F(u):= \int_{\om} \alpha(x) d|Du|,\quad
F_{n}(u):= \int_{\om}\alpha_{n}(x)d|Du|,
\]
then, the sequence $(F_{n})_{n\in\NN}$ $\Gamma$-converges to $F$ with the underlying topology being the topology of weak$^{\ast}$ convergence in $\bv(\om)$.
\end{gamma_conv_F}

\begin{proof}
It is straightforward to show that $(F_{n})_{n\in\NN}$ is a decreasing sequence which converges to $F$ pointwise since for every $u\in\bv(\om)$ we have
\[|F_{n}(u)-F(u)|\le \int_{\om} |\alpha_{n}(x)-\alpha(x)|d|Du|\le \|\alpha_{n}-\alpha\|_{\infty} |Du|(\om)\to 0\quad \text{as }n\to\infty.\]
Then the conclusion follows from \cite[Prop. 5.7]{dalmasogamma} and the fact that $F$  is  lower semicontinuous with respect to the weak$^{\ast}$ convergence in $\bv(\om)$.
\end{proof}

Note that the above results hold in arbitrary dimension. However, the proof of the  well-posedness of \eqref{weighted_rof} with vanishing weight $\alpha$, strongly relies on Theorem \ref{lbl:TVu_less_TVf} and, thus, it only holds in dimension one.

\newtheorem{well_alpha_zero}[alpha_sgn]{Theorem}
\begin{well_alpha_zero}\label{lbl:well_alpha_zero}
Let $\om=(a,b)$, $f\in\bv(\om)$ and $\alpha\in C(\overline{\om})$ with $\alpha\ge 0$. Then the minimisation problem
\[\min_{u\in\bv(\om)} J(u):=\frac{1}{2} \int_{\om} (f-u)^{2}dx +\int_{\om} \alpha(x)d|Du|,
\]
has a unique solution $u^{\ast}\in\bv(\om)$. Moreover if 
\[J_{n}(u):=\frac{1}{2} \int_{\om} (f-u)^{2}dx +\int_{\om} \alpha_{n}(x)d|Du|,\]
where $(\alpha_{n})_{n\in\NN}\subseteq C^{\infty}(\overline{\om})$ is a decreasing sequence, uniformly convergent to $\alpha$, with $\alpha_{n}>0$ for every $n\in\NN$, then
\[J(u^{\ast})=\lim_{n\to\infty} \min_{u\in\bv(\om)} J_{n}(u).\]
\end{well_alpha_zero}

\begin{proof}
Note first that since $\alpha$ is continuous, using a standard mollification argument, we can indeed construct a decreasing sequence $(\alpha_{n})_{n\in\NN}$ in $C^{\infty}(\overline{\om})$, uniformly convergent to $\alpha$, and $\alpha_{n}>0$ for every $n\in\NN$.
Using Proposition \ref{lbl:gamma_conv_F} and the fact that the map $u\mapsto \frac{1}{2}\int_{\om}(f-u)^{2}dx$ is lower semicontinuous with respect to the weak$^\ast$ convergence in $\bv(\om)$ we have \cite[Prop. 6.25]{dalmasogamma} that the sequence $(J_{n})_{n\in\NN}$ $\Gamma$-converges to $J$. Let now 
\[u_{n}=\underset{u\in\bv(\om)}{\operatorname{argmin}}\; J_{n}(u).\]
We then have 
\[\frac{1}{2} \int_{\om} (f-u_{n})^{2}dx\le \frac{1}{2} \int_{\om} (f-u_{n})^{2}dx +\int_{\om} \alpha_{n}(x)d|Du_{n}| \le \frac{1}{2}\int_{\om}f^{2}dx.\]
Thus the sequence $(u_{n})_{n\in\NN}$ is bounded in $L^{2}(\om)$ and hence also in $L^{1}(\om)$. Using now Theorem \ref{lbl:TVu_less_TVf} we have that $(u_{n})_{n\in\NN}$ is also bounded in $\bv(\om)$, i.e., there exists a positive constant $C$ such that
\[u_{n}\in K:=\{u\in \bv(\om):\; \|u\|_{\bv(\om)}\le C\}, \quad \text{for all } n\in\NN.\]
This means that  for every $n\in\NN$
\[\min_{u\in\bv(\om)} J_{n}(u)=\min_{u\in K} J_{n}(u),\]
witk $K$ being a weakly$^\ast$ sequentially compact set. Then the conclusions of the theorem follow straightforwardly from \cite[Thm. 7.4]{dalmasogamma} and the fact that the functional $J$ is strictly convex.

\end{proof}

We would like to finish this section with an extension of Theorem \ref{lbl:TVu_less_TVf} for continuous weight functions $\alpha$ that are also not necessarily bounded away from zero.

\newtheorem{TVu_less_TVf_2}[alpha_sgn]{Theorem}
\begin{TVu_less_TVf_2}[$|Du|(\om)\le |Df|(\om)$, for continuous $\alpha$]\label{lbl:TVu_less_TVf_2}
Let $\om=(a,b)\subseteq\RR$, $\alpha\in C(\overline{\om})$ with $\alpha\ge 0$ and $f\in\bv(\om)$. If
\[u=\underset{u\in\bv(\om)}{\operatorname{argmin}}\; \frac{1}{2}\int_{\om}(f-u)^{2}dx+\int_{\om}\alpha(x)d|Du|,\]
then 
\[|Du|(\om)\le |Df|(\om).\]
\end{TVu_less_TVf_2}

\begin{proof}
Define the weights $(\alpha_{n})_{n\in\NN}\in C^{\infty}(\overline{\om})$ and the functionals $J$, $(J_{n})_{n\in\NN}$ as in Theorem \ref{lbl:well_alpha_zero}. Let also $u$ and $(u_{n})_{n\in\NN}$ be the corresponding minimisers. Note that we have already shown in Theorem \ref{lbl:TVu_less_TVf} that
\begin{equation}\label{from_main_Thm}
|Du_{n}|(\om)\le |Df|(\om), \quad \text{for all }n\in\NN.
\end{equation}
Since the sequence $(J_{n})_{n\in\NN}$ $\Gamma$-converges to $J$, we have that any weak$^{\ast}$ cluster point of $(u_{n})_{n\in\NN}$ is equal to $u$  \cite[Cor. 7.20]{dalmasogamma}. This implies that the whole sequence $(u_{n})_{n\in\NN}$ converges to $u$ weakly$^{\ast}$ in $\bv(\om)$.
 Hence from the lower semicontinuity of total variation with respect to $L^{1}$ convergence and \eqref{from_main_Thm} we have
\[|Du|(\om)\le \liminf_{n\to\infty} |Du_{n}|(\om)\le |Df|(\om).\]

\end{proof}
Notice that in view of Theorem \ref{lbl:TVu_less_TVf_2}, Theorem \ref{lbl:well_alpha_zero} holds also in the case where $(\alpha_{n})_{n\in\NN}\subseteq C(\overline{\om})$.

\subsection{Application: Exact reconstruction of piecewise constant noisy data}\label{exact_recon}

It is well-known that scalar total variation regularisation is very efficient in recovering noisy piecewise constant functions. Nevertheless the reconstruction typically suffers from a loss of contrast. Techniques like Bregman iteration \cite{OBG} have been used to reduce this effect. In this section  we show that by choosing a suitable weight function $\alpha$, it is possible to avoid  this loss of contrast and, under some mild assumptions on the noise, to exactly recover the piecewise constant function. 

 In what follows, let $f_{0}$ be a piecewise constant function, i.e.,
\begin{equation}\label{f0}
f_{0}(x)=\sum_{i=1}^{N}f_{i}\mathcal{X}_{I_{i}}(x),\quad x\in\om,
\end{equation}
where $(I_{i})_{i=1}^{N}$ are disjoint  intervals with $\bigcup_{i=1}^{N}I_{i}=\om$ and $f_{i}\in\RR$ for every $i=1,\ldots, N$ with $f_{i}\ne f_{j}$ for $i\ne j$. Moreover let $\eta$ denote  an oscillatory  function that belongs to $\bv(\om)$ and satisfies
\begin{equation}\label{noise}
 \int_{I_{i}} \eta \,dx=0 \quad \text{ for all } i=1,\ldots,N.
\end{equation}

\newtheorem{exact_pc}[alpha_sgn]{Proposition}
\begin{exact_pc}\label{lbl:exact_pc}
 Let $f\in\bv(\om)$ with
\[f=f_{0}+\eta,\]
where $f_{0}$ and $\eta$ are as described above. Define $\alpha$ to be a continuous piecewise affine function such that
\begin{enumerate}
\item $\alpha(x_{i})=0$, where $x_{i}=\sup I_{i}$, $i=1,\ldots, N-1$.
\item $\alpha\ge 0$ and it is of the form $-\mu_{i}|x-c_{i}|+d_{i}$ in every interval $I_{i}$, with $\mu_{i}>2\|f\|_{\infty}$ for every $i=1,\ldots,N$. 
 \end{enumerate}
Then we have that
\begin{equation}\label{min_pc}
f_{0}=\underset{u\in\bv(\om)}{\operatorname{argmin}}\; \frac{1}{2}\int_{\om} (f-u)^{2}dx+\int_{\om} \alpha(x)d|Du|.
\end{equation}

\end{exact_pc}

\begin{proof}
 Consider first the corresponding minimisation problems with weights $\alpha_{n}:=\alpha+1/n$ and solutions $u_{n}$. Let $u$ be the solution of the weighted total variation minimisation with weight $\alpha$. We have to show that $u=f_{0}$. Since $|\alpha_{n}'(x)|>2\|f\|_{\infty}$ in every interval $I_{i}$, from Proposition \ref{lbl:large_gradient_plateau} in combination with Proposition \ref{lbl:lipschitz_alpha} we have that $u_{n}$ will be constant in the interior of every $I_{i}$. Moreover, the $\Gamma$-convergence argument from the proof of Theorem \ref{lbl:TVu_less_TVf_2} implies that $u_{n}\to u$ weakly$^{\ast}$ in $\bv(\om)$ and thus $u$ will be also constant in the interior of every $I_{i}$. From the fact that $\alpha(x_{i})=0$ we have that
\[\int_{\om}\alpha(x)d|Du|=0,\]
i.e.,
\begin{equation}\label{min_kernel}
u=\underset{\phi\in \mathrm{Ker}\int\alpha(x) d|D\cdot|}{\operatorname{argmin}} \frac{1}{2} \int_{\om } (f-\phi)^{2}dx,
\end{equation}
where 
\[\mathrm{Ker}\int\alpha(x) d|D\cdot|=\left\{u\in\bv(\om):\; u=\sum_{i=1}^{N}u_{i}\mathcal{X}_{I_{i}}(x),\;u_{i}\in \RR,\;i=1,\ldots,N \right\}.\]
It is then easy to see that the minimisation in \eqref{min_kernel} can be separated into $N$ minimisation problems each of which corresponds to an interval $I_{i}$. Consequently, the solution $u$ is of the form
\[u=\sum_{i=1}^{N}u_{i}\mathcal{X}_{I_{i}}(x),\]
where
\[u_{i}=\underset{c\in\RR}{\operatorname{argmin}}\; \frac{1}{2} \int_{I_{i} } (f-c)^{2}dx=\frac{1}{|I_{i}|}\int_{I_{i}}f\,dx=\frac{1}{|I_{i}|}\int_{I_{i}} f_{0}+\eta\,dx\overset{\eqref{noise}}{=}\frac{1}{|I_{i}|}\int_{I_{i}} f_{0}\,dx=f_{i},\]
and thus $u=f_{0}$.
\end{proof}

\begin{figure}[t!]
\begin{center}
\includegraphics[width=0.3\textwidth]{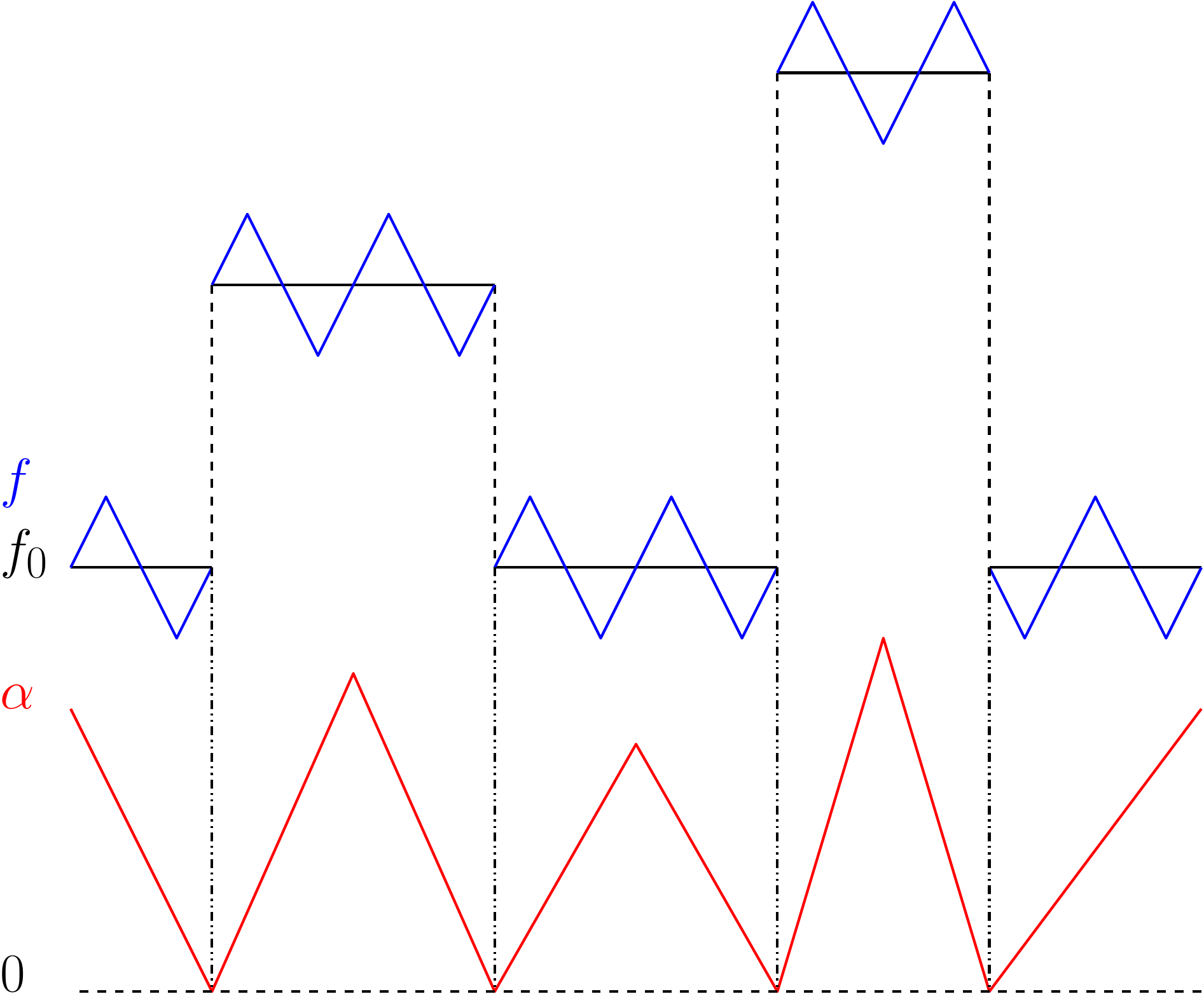}
\caption{Schematic explanation of Proposition \ref{lbl:exact_pc}. A noisy piecewise constant function can be recovered exactly using weighted $\tv$ regularisation under a suitable weight function $\alpha$ which vanishes exactly at the jump points of $f_{0}$ and has large gradient everywhere else}
\label{fig:pc}
\end{center}
\end{figure}

\begin{figure}[t!]
\begin{center}
\begin{subfigure}[t]{0.48\textwidth}
	\centering
	\captionsetup{width=.9\linewidth}
	\includegraphics[width=0.98\textwidth]{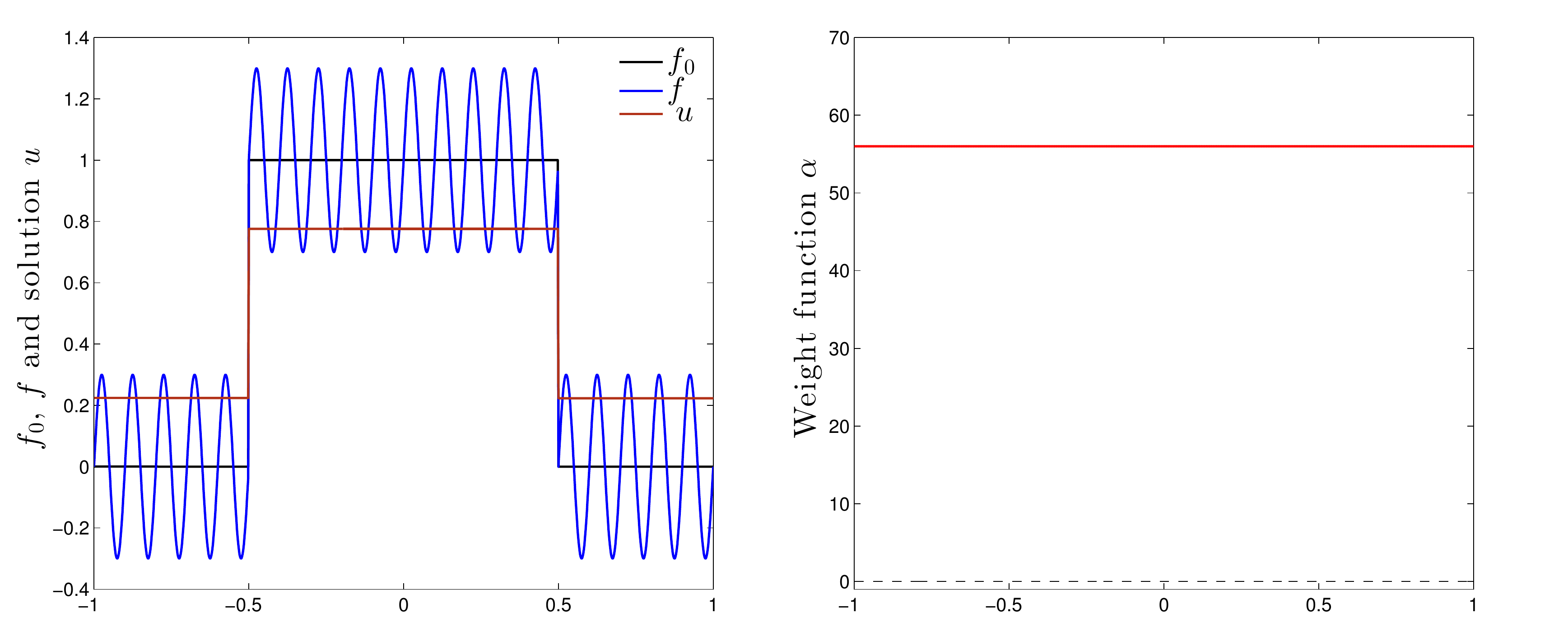}
	\caption{Scalar $\tv$ denoising}
	\label{fig:pc_numerics_scalar}
\end{subfigure}
\begin{subfigure}[t]{0.48\textwidth}
	\centering
	\captionsetup{width=.9\linewidth}
	\includegraphics[width=0.98\textwidth]{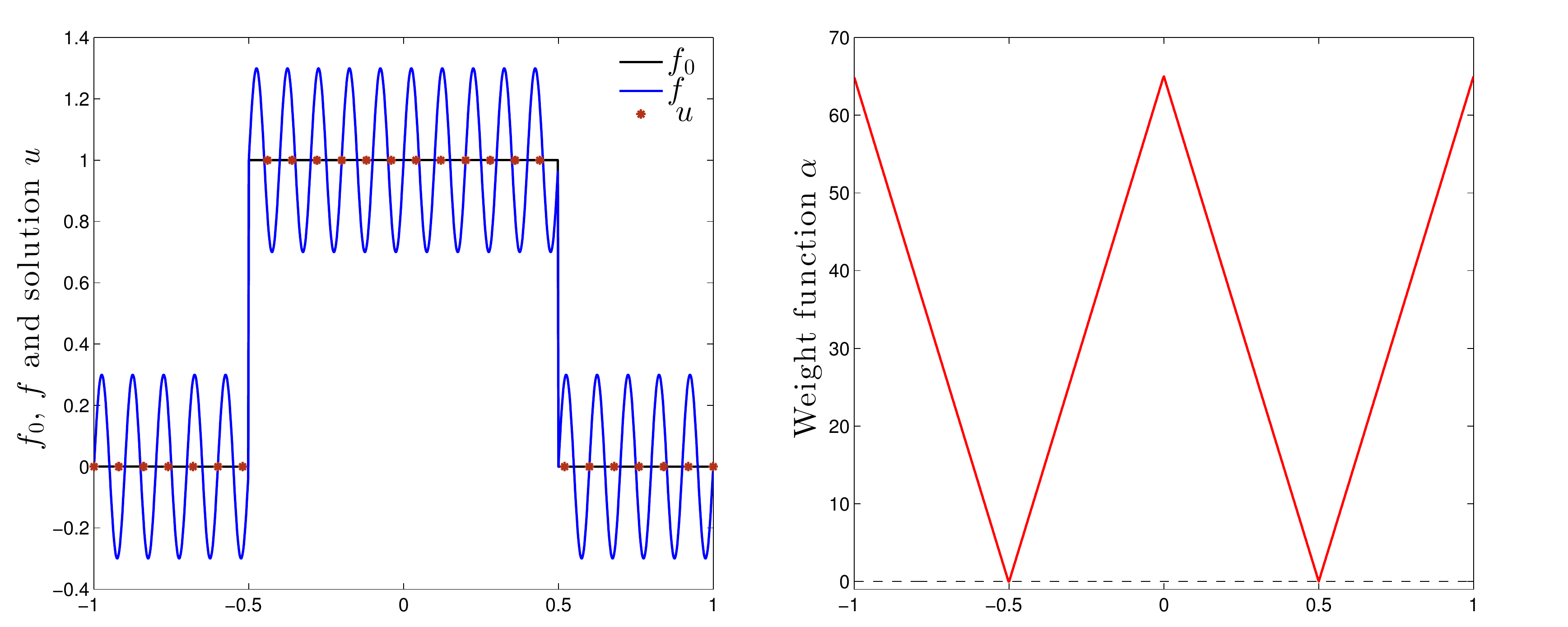}
	\caption{Weighted $\tv$ denoising}
	\label{fig:pc_numerics_spatial}
\end{subfigure}
\caption{Numerical examples: Unlike the scalar case, weighted $\tv$ can recover exactly noisy piecewise constant functions}
\label{fig:pc_numerics}
\end{center}
\end{figure}

We should remark here that this exact recovery of piecewise constant functions can be achieved for a whole family of weight functions $\alpha$ that are not necessarily of the form $-\mu_{i}|x-c_{i}|+d_{i}$ in every interval $I_{i}$. For instance any weight function $\alpha$ which is initially increasing and then decreasing in every interval $I_{i}$ with the property $|\alpha'(x)|>\|f\|_{\infty}$ (except at its maximum point in $I_{i}$) will also lead to an exact recovery of $f_{0}$.

Figure  \ref{fig:pc} depicts the result of Proposition  \ref{lbl:exact_pc}. Note that $\alpha$ must be zero at the jump points of $f$ so that these jumps are not penalised at all.  On the other hand, $\alpha$ should be steep enough so the remaining parts of the solution $u$ are constant according to  Proposition \ref{lbl:large_gradient_plateau}. We also provide some numerical examples in Figure \ref{fig:pc_numerics}. There, the data $f$ is a perturbation of a piecewise constant function $f_{0}$, under a $\sin$ type of noisy function that satisfies the mean value property \eqref{noise}. For the scalar $\tv$ denoising, we have chosen the smallest $\alpha$ such that we get a  piecewise constant solution. As expected, this results in  a significant loss of contrast, Figure \ref{fig:pc_numerics_scalar}. However, by applying weighted $\tv$ regularisation with a weight as it is described in Proposition \ref{lbl:exact_pc}, an exact recovery is achieved, Figure \ref{fig:pc_numerics_spatial}. 
 One might expect that since the values of $\alpha$ remain small close to the jump points, then the noise would be still present in that area. However, this is not true, since constant parts do not only result due to the large magnitude of $\alpha$ but also due to high values of $\alpha'$, Proposition \ref{lbl:large_gradient_plateau}.

\section{Total variation regularisation with weighted fidelity term}\label{sec:weight_fid}

As we have already mentioned in the introduction, one has two options when it comes to spatially adapted regularisation, i.e., introducing a weight function either in the regulariser or in the fidelity term. In this final section of the paper we would like to briefly examine the second case and in particular study the total variation regularisation problem with weighted fidelity term 
\begin{equation}\label{weight_fid}
\min_{u\in\bv(\om)} \frac{1}{2}\int_{\om}w(f-u)^{2}dx+|Du|(\om),
\end{equation}
where $\om=(a,b)$, $f\in L^{2}(\om)$ and $w\in L^{\infty}(\om)$ with $w\ge 0$. 

Existence of solutions for the problem \eqref{weight_fid} follows straightforwardly using standard methods. Note that the possibility of $w$ vanishing in some areas does not pose extra difficulties as in the weighted total variation case. It is clear, however, that the solution of \eqref{weight_fid} is not always unique. The solution will be unique if the operator $T_{w}:L^{2}(\om)\to L^{2}(\om)$ with
$T_{w}u(x)=\sqrt{w(x)}u(x)$,
is injective, i.e., when the set $\{w=0\}$ is of zero Lebesgue measure. 

First order optimality conditions for \eqref{weight_fid} are stated next, see, e.g., \cite{ring2000structural} for a proof.

\newtheorem{weight_fid_opt}[alpha_sgn]{Proposition}
\begin{weight_fid_opt}[\cite{ring2000structural}]\label{lbl:weight_fid_opt}
Let $\om=(a,b)$, $f\in L^{2}(\om)$ and $w\in L^{\infty}(\om)$ with $w\ge 0$. A function $u\in\bv(\om)$ is a solution of \eqref{weight_fid} if and only if there exists a function $v\in H_{0}^{1}(\om)$ such that
\begin{align}
v'&=w(f-u),\label{fid_opt1}\\
-v&\in \Sgn (Du). \label{fid_opt2}
\end{align}
\end{weight_fid_opt}

As it can be readily seen from \eqref{fid_opt1}--\eqref{fid_opt2}  and also shown in \cite{ring2000structural} we have that $Du=0$ when $f\ne u$, or more rigorously on the open sets $\{\underline{f}>\overline{u}\}$ and $\{\overline{f}<\underline{u}\}$, provided $w>0$ there. This is in strong contrast to the weighted $\tv$ case where this is not true, in general. Furthermore, unlike weighted $\tv$, no new discontinuities can be created with the one dimensional version of \eqref{weight_fid}. This is shown in the following proposition whose proof  is similar to the one dimensional scalar $\tv$ case. 

\newtheorem{weight_fid_discont}[alpha_sgn]{Proposition}
\begin{weight_fid_discont}\label{lbl:weight_fid_discont}
Let $\om=(a,b)$, $f\in \bv(\om)$ and $w\in L^{\infty}(\om)$ with $w\ge 0$. If $u$ is a solution to \eqref{weight_fid} then
\[J_{u}\cap \mathrm{supp}_{l,r}w\subseteq J_{f},\]
where 
\[\mathrm{supp}_{l,r}w=\left \{x\in \om:\; \text{for every }\epsilon>0,\;  w \text{ is not zero a.e. in each of the sets } (x-\epsilon,x) \text{ and } (x,x+\epsilon) \right \}.\]
In particular if $w>0$ a.e. then 
\[J_{u}\subseteq J_{f}.\]
\end{weight_fid_discont}
\begin{proof}
Let $x\in J_{u}\cap \mathrm{supp}_{l,r}w$ and suppose that $x\notin J_{f}$.  Without loss of generality, we can assume that $Du(\{x\})>0$. Then from condition \eqref{fid_opt2} we have $v(x)=-1$, where $v$ is the corresponding $H_{0}^{1}(\om)$ variable of Proposition \ref{lbl:weight_fid_opt}. 
Observe that since $u$ has a positive jump at $x$ and $f$ is continuous there we have that there exists $m>0$ and a sufficiently small $\epsilon>0$ such that
\[\text{either}\quad 0<m<\essinf_{t\in (x-\epsilon,x)}f(t)-u(t)\quad \text{or} \quad \esssup_{t\in (x,x+\epsilon)} f(t)-u(t)<-m<0.\]
It is also possible that both inequalities above hold. If the first inequality holds, then, in view of \eqref{fid_opt1} and the fact that $x\in \mathrm{supp}_{l,r}w$, we have that
\[v(x)-v(x-\delta)=\int_{x-\delta}^{\delta}v'(s)ds> 0,\quad \text{for all }\delta\in (0,\epsilon).\]
This implies that $v(x-\delta)<v(x)=-1$. If the second inequality holds then we have 
\[v(x+\delta)-v(x)=\int_{x}^{x+\delta} v'(s)ds<0,\quad \text{for all }\delta\in (0,\epsilon),\]
which implies that $v(x+\delta)<v(x)=-1$. Thus, in both cases we end up in a contradiction since we must have $|v|\le 1$ in $\om$ according to \eqref{fid_opt2}. 
\end{proof}

Next we study relations between \eqref{weighted_rof} and \eqref{weight_fid} by means of an explicit example, which offers a good insight into different structural properties of the associated solutions.

\newtheorem{affine_exact_fid}[alpha_sgn]{Proposition}
\begin{affine_exact_fid}\label{lbl:affine_exact_fid}
Let $\om=(-L,L)$, $f(x)=\lambda x$ with $\lambda>0$. Moreover let $w\in L^{\infty}(\om)$ with $w\ge 0$. Then, a solution to the problem 
\[\min_{u\in\bv(\om)} \frac{1}{2}\int_{\om}w(f-u)^{2}dx+|Du|(\om),\]
will be of the following form:
\begin{equation}\label{affine_1}
u(x)=
\begin{cases}
\lambda x_{1}, & \text{if}\quad x\in (-L,x_{1}),\\
\lambda x, & \text{if}\quad x\in [x_{1},x_{2}], \qquad \text{where} -L<x_{1}\le x_{2}<L.\\
\lambda x_{2}, & \text{if}\quad x\in (x_{2},L),
\end{cases}
\end{equation}
\end{affine_exact_fid}

\begin{proof}
We will show that for every weight function $w\in L^{\infty}(\om)$, a function $u$ of the type \eqref{affine_1}, will always satisfy the optimality conditions \eqref{fid_opt1}--\eqref{fid_opt2} together with an appropriate function $v\in H_{0}^{1}(\om)$. Since $v\in H_{0}^{1}(\om)$ and it satisfies \eqref{fid_opt1} we have
\[
v(x_{1})=\lambda \int_{-L}^{x_{1}} w(s) (s-x_{1})ds,\quad
v(x_{2})=-\lambda \int_{x_{2}}^{L} w(s)(s-x_{2})ds.
\]Notice that if we set
\[\phi_{1}(x):=\lambda \int_{-L}^{x} w(s) (s-x)ds, \quad 
\phi_{2}(x):=-\lambda \int_{x}^{L} w(s)(s-x)ds,\quad x\in [-L,L],
\]then we have that $\phi_{1}(-L)=\phi_{2}(L)=0$, $\phi_{1}$ and $\phi_{2}$ are respectively decreasing and increasing continuous functions, and hence there exists a point $\xi\in (-L,L)$ such that $\phi_{1}(\xi)=\phi_{2}(\xi)=k<0$.\\
\emph{Case I}: $k\le -1$. In that case we set 
\[x_{1}=\min\left\{x\in \om: \; \phi_{1}(x)=-1\right\}\quad \text{and}\quad x_{2}=\max\left\{x\in\om:\;\phi_{2}(x)=-1\right\},\]
and we have $-L<x_{1}<x_{2}<L$.
Then it is easy to check that 
\[
v(x)=
\begin{cases}
\phi_{1}(x),& \text{if}\quad x\in (-L,x_{1}),\\
-1,		     & \text{if}\quad x\in [x_{1},x_{2}],\\
\phi_{2}(x),& \text{if}\quad x\in (x_{2},L),
\end{cases}
\]
belongs to $H_{0}^{1}(\om)$ and satisfies the optimality conditions \eqref{fid_opt1}--\eqref{fid_opt2} together with the function $u$ defined in \eqref{affine_1}.\\
\emph{Case II}: $k> -1$.  In this case, we simply set
\[
v(x)=
\begin{cases}
\phi_{1}(x),& \text{if}\quad x\in (-L,\xi],\\
\phi_{2}(x),& \text{if}\quad x\in (\xi,L),
\end{cases}
\]
and the optimality conditions are satisfied with this function $v$ and the constant function $u=\lambda \xi$.
\end{proof}

As the proof above shows, if the weight function is small enough in the sense that either
 \begin{equation}\label{lbl:remark_constant_solution}
 \phi_{1}(L)\le -1 \quad \text{or} \quad \phi_{2}(-L)\le -1,
\end{equation} 
then the solution $u$ of \eqref{weight_fid} will be a constant.

\begin{figure}[t!]
\begin{center}
\includegraphics[width=0.7\textwidth]{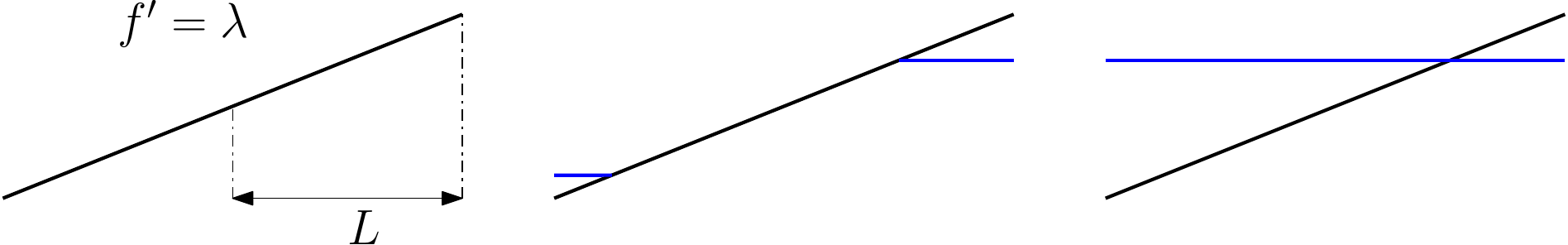}
\caption{Different types of solutions of the  total variation minimisation problem with weighted fidelity term \eqref{weight_fid} with data $f(x)=\lambda x$ in $(-L,L)$ and for any weight function $w$}
\label{fig:affine_fid}
\end{center}
\end{figure}

In Figure \ref{fig:affine_fid} we depict all these possible solutions. These should be compared to the ones of the weighted $\tv$ problem in Figure \ref{fig:affine_abs}. Note that when $x_{1}=-x_{2}$ (which can be achieved for instance using a symmetric weight function $w$), then the resulting solution also belongs to the set of solutions of the scalar total variation minimisation problem.

\begin{figure}[t!]
\begin{center}
\begin{subfigure}[t]{0.48\textwidth}
	\centering
	\captionsetup{width=.9\linewidth}
	\includegraphics[width=0.98\textwidth]{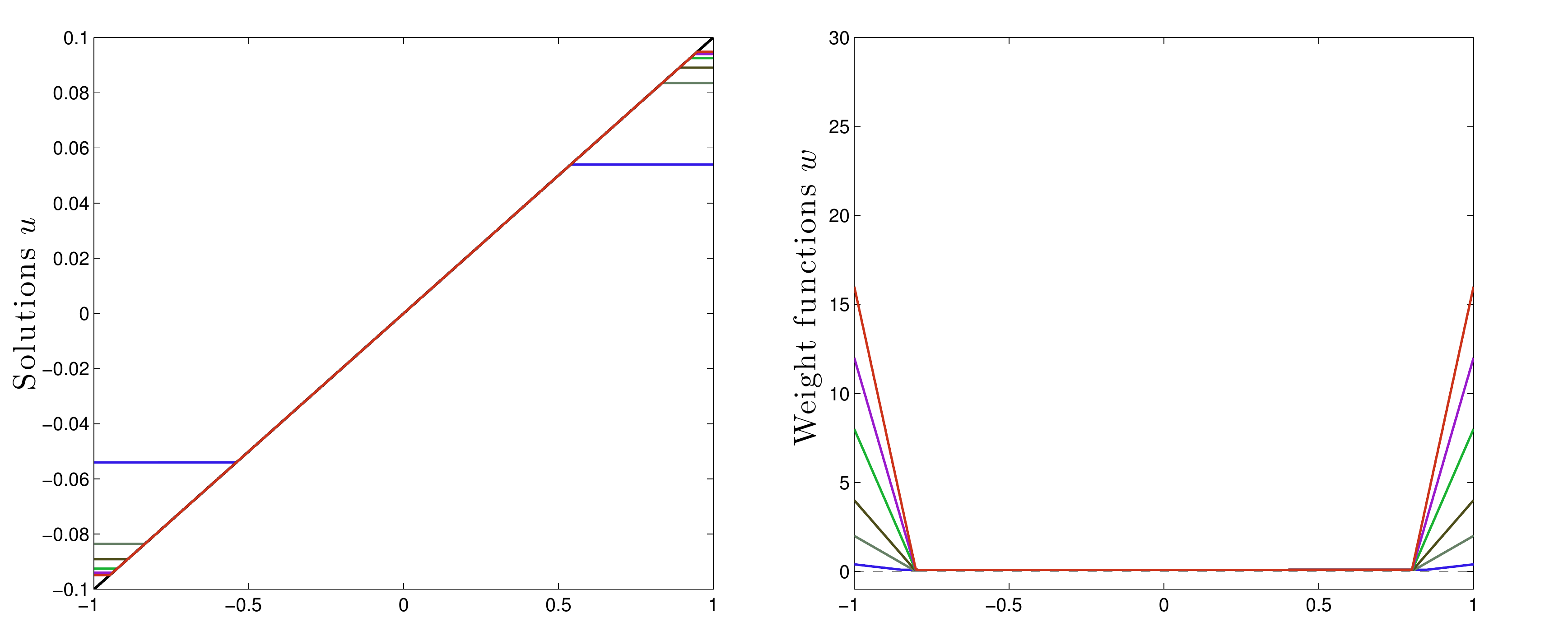}
	\caption{Weight functions $w$ with large values near the boundary}	\label{fig:w_boundary_mass}
\end{subfigure}
\begin{subfigure}[t]{0.48\textwidth}
	\centering
	\captionsetup{width=.9\linewidth}
	\includegraphics[width=0.98\textwidth]{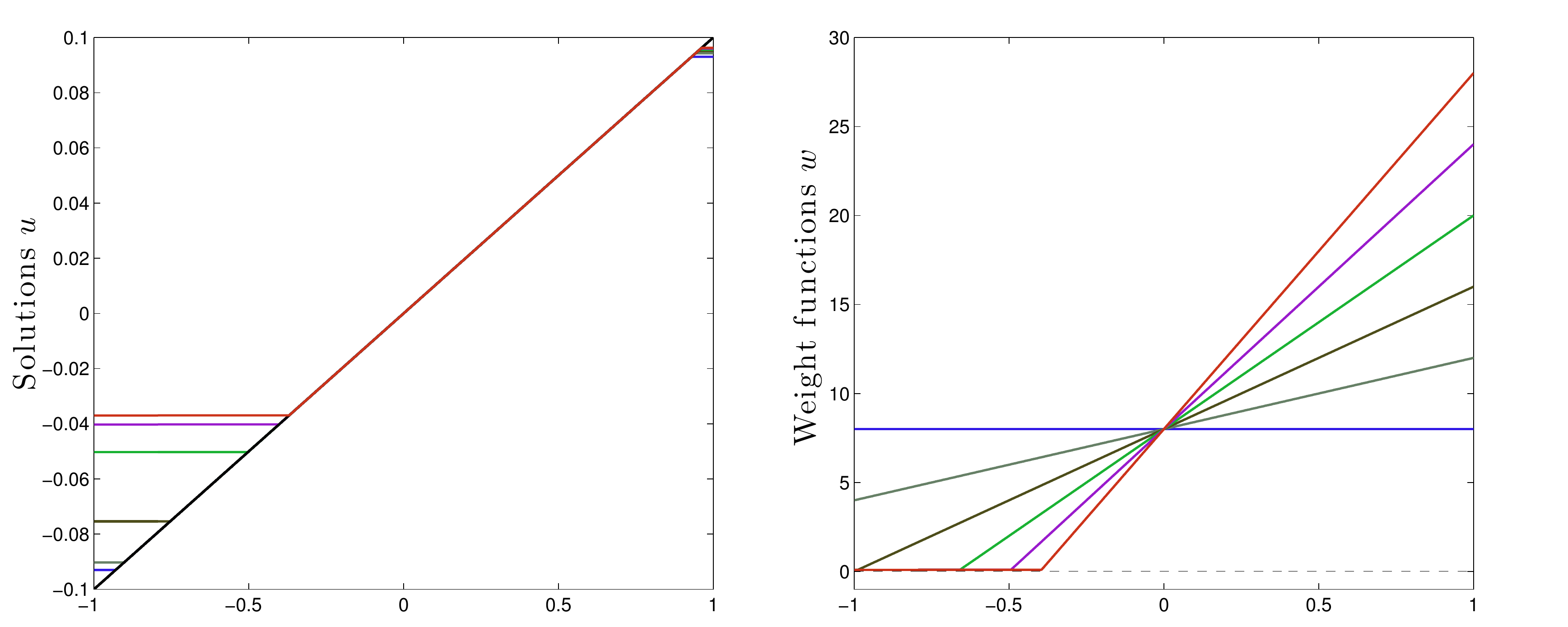}
	\caption{Linear weight functions $w$}
	\label{fig:w_lines}
\end{subfigure}

\begin{subfigure}[t]{0.48\textwidth}
	\centering
	\captionsetup{width=.9\linewidth}
	\includegraphics[width=0.98\textwidth]{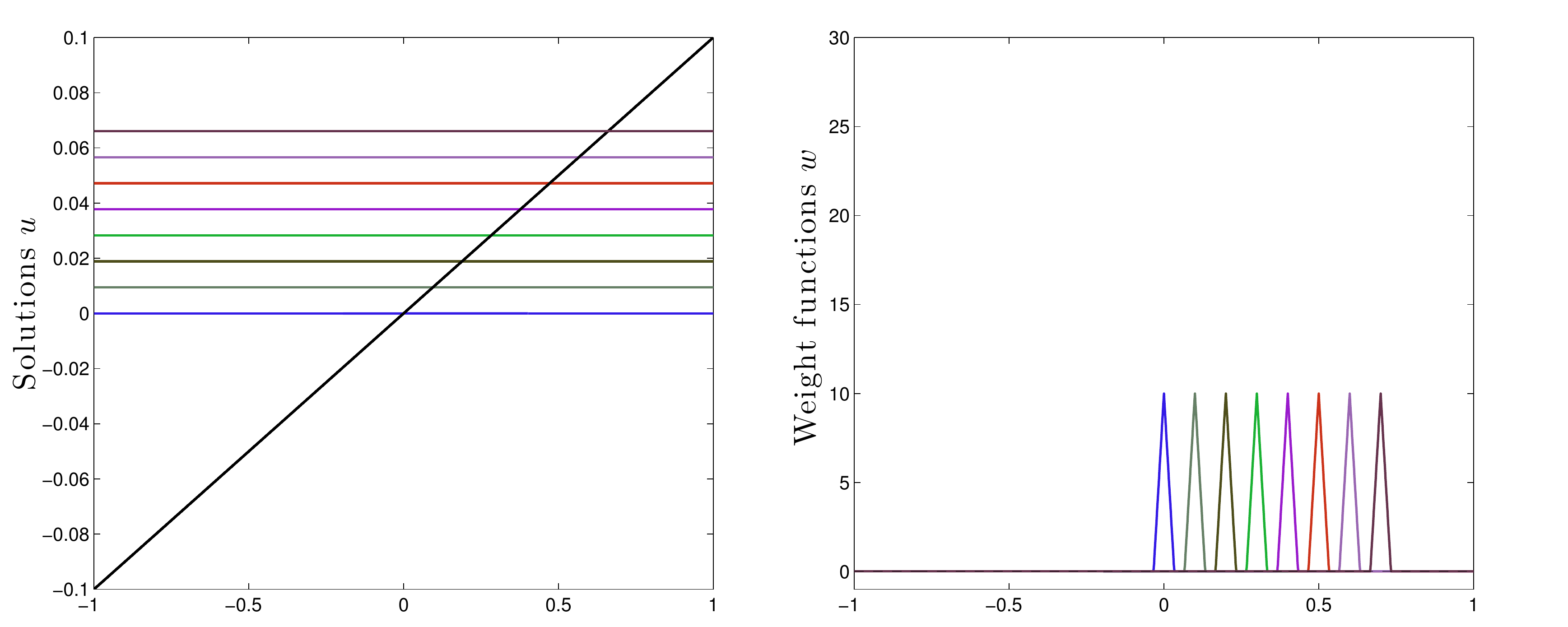}
	\caption{Weight functions $w$ of small mass}
		\label{fig:w_bumps}
\end{subfigure}
\begin{subfigure}[t]{0.48\textwidth}
	\centering
	\captionsetup{width=.9\linewidth}
	\includegraphics[width=0.98\textwidth]{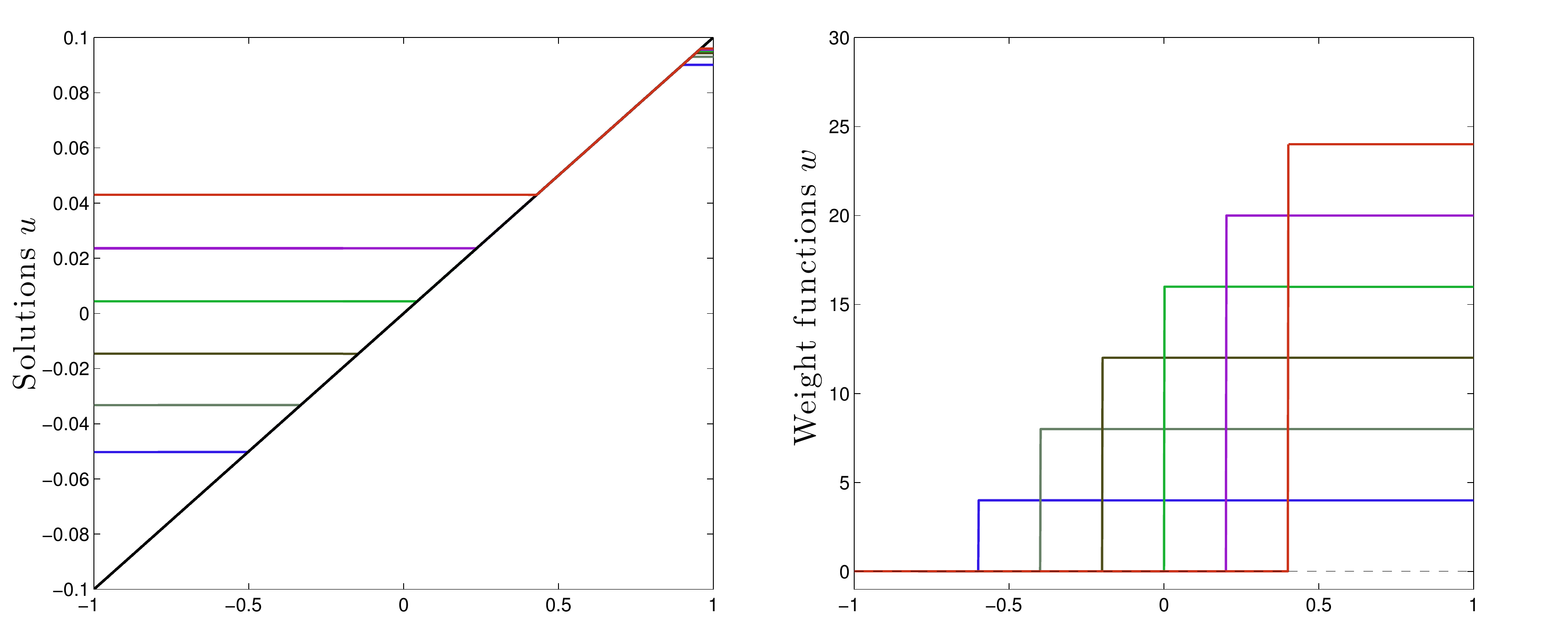}
	\caption{Step weight functions with mass concentrating at the one side of the boundary}	  		    
	\label{fig:w_step_functions}
\end{subfigure}
\caption{Numerical solutions of the problem \eqref{weight_fid} with data $f(x)=\lambda x$ in $(-L,L)$ for different types of weight functions $w$. In every case the solutions are of the form \eqref{affine_1}}
\label{fig:w}
\end{center}
\end{figure}

We have also performed some numerical simulations for the problem \eqref{weight_fid} with data $f(x)=\lambda x$; see Figure \ref{fig:w}. Note that, as predicted by the Proposition \ref{lbl:affine_exact_fid}, all the solutions are of the form \eqref{affine_1}. In order to avoid the non-uniqueness issues, we have set the minimum value of each weight function $w$ to be  a small positive constant. We briefly now comment on these results:
\begin{enumerate}[(a)]
\item In Figure \ref{fig:w_boundary_mass} we have chosen the weight functions to have a small value away  from the boundary of the domain, while increasing their values close to the boundary. Observe that the solutions $u$ converge to $f$ as the values of the weight functions near
 the boundary are increasing. Thus, for this example, in order to  recover $f$ almost perfectly, it suffices to choose $w$ to have large values only near the boundary.
\item In Figure \ref{fig:w_lines}, the weight functions have been chosen to be linear with increasing gradient. As expected, as the gradient of $w$ increases, resulting in large values close to the point $L$,  the solution $u$ approximates $f$ better, while the approximation becomes worse close to $-L$ where $w$ has small values.
\item In Figure \ref{fig:w_bumps}, all the weight functions $w$ have been chosen such that their mass is small and is shifted towards the boundary. In all the cases, the conditions \eqref{lbl:remark_constant_solution} are satisfied and hence constant solutions are obtained. 
\item Lastly, in Figure \ref{fig:w_step_functions}, we have chosen weight functions $w$ whose mass concentrates at the right side of the boundary of $\Omega$, observing a similar behaviour as in Figure \ref{fig:w_lines}.
\end{enumerate}

The following proposition states that at least for the affine data considered above, the weighted $\tv$ problem \eqref{weighted_rof} cannot produce those solutions of \eqref{weight_fid}  that are not symmetric with respect to the origin, i.e., the ones that cannot be obtained by scalar total variation regularisation. We note that this is regardless of the choice of the weight function $\alpha$.

\newtheorem{alpha_cannot}[alpha_sgn]{Proposition}
\begin{alpha_cannot}\label{lbl:alpha_cannot}
Let $\om=(-L,L)$, $\alpha\in C(\overline{\om})$, $\alpha>0$, $f(x)=\lambda x$, for every $x\in\om$, $\lambda>0$. Then the weighted total variation problem \eqref{weighted_rof} with data $f$ and weight $\alpha$ cannot have any solution $u$ of the type \eqref{affine_1} unless 
$x_{1}=-x_{2}$,
that is, unless the solution is symmetric with respect to the origin.
\end{alpha_cannot}

\begin{proof}
Suppose without loss of generality that there exists a weight function $\alpha\in C(\overline{\om})$ such that the solution of \eqref{weighted_rof} is a function of the type \eqref{affine_1}, where $x_{1}<-x_{2}$. From conditions \eqref{opt1}--\eqref{opt2} we have that this can only happen if $\alpha(x)=\alpha_{0}$ for every $x\in(x_{1},x_{2})$. Moreover, the dual function $v$ would have to be quadratic in the intervals $(0,x_{1}]$ and $[x_{2},L)$ satisfying the conditions $v(0)=0$, $v(x_{1})=-\alpha_{0}$, $v'(x_{1})=0$,  and $v(L)=0$, $v(x_{2})=-\alpha_{0}$, $v'(x_{2})=0$. One can easily see that such a $v$ cannot exist but note also that in that case we would have $-\alpha_{0}\le v(x)\le 0$ for every $x\in \om$, i.e., $u$ and $v$ would also satisfy the optimality conditions of standard  total variation minimisation with scalar parameter $\alpha_{0}$. Of course this is a contradiction since we know that any solution to the scalar parameter problem is not of that form.

\end{proof}

In contrast to the previous result, one can easily see that for the specific data function of Proposition \ref{lbl:alpha_cannot}, there are choices of non-constant weight functions $\alpha\in C(\overline{\om})$ and $w\in L^{\infty}(\om)$ such that the corresponding solutions of the problems \eqref{weighted_rof} and \eqref{weight_fid} are solutions of scalar total variation problems.  These are exactly the solutions of the type \eqref{affine_1} with $x_{1}=-x_{2}$. Figure \ref{fig:venn} visualises the relation between \eqref{weighted_rof} and \eqref{weight_fid}. 

\begin{figure}[h!]
\centering
\begin{tikzpicture}
\usetikzlibrary{shapes}
  \tikzset{venn circle/.style={draw,ellipse,minimum width=7cm, minimum height=3cm,fill=#1,opacity=0.4}}

  \node [venn circle = yellow] (A) at (-0.5,0) { };
  \node [venn circle = link] (C) at (0:3cm) {};
  \node[text width=2.8cm] at (barycentric cs:A=3/7,C=2/3 ) {\footnotesize{Set of solutions of  the  scalar TV problem } };  
    \node[text width=2.5cm] at (-1.7,0) {\footnotesize{Set of solutions of the weighted TV problem \eqref{weighted_rof}}}; 
        \node[text width=2.8cm] at (0:4.7cm) {\footnotesize{Set of solutions of the weighted fidelity problem \eqref{weight_fid}} };    
\end{tikzpicture} 
\caption{Data function $f(x)=\lambda x$: The solutions that can be obtained by both problems \eqref{weighted_rof} and \eqref{weight_fid} are exactly those that can be obtained by the classical scalar total variation minimisation.}
\label{fig:venn}
\end{figure}
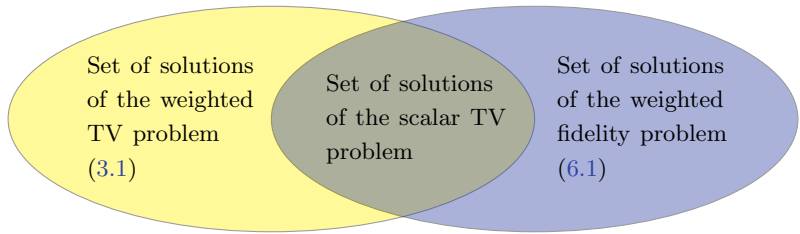

In Section \ref{exact_recon} we examined the capability of the weighted $\tv$ model to recover exactly piecewise constant data, provided the weight function $\alpha$ vanishes at the jump points of the data. On the contrary, the weighted fidelity model \eqref{weight_fid} cannot recover piecewise constant functions even in the absence of noise, regardless of the weight function $w$. Indeed, it can be easily checked using the optimality conditions  \eqref{fid_opt1}--\eqref{fid_opt2} that only constant functions can be recovered exactly using \eqref{weight_fid}. However, as the next proposition shows, a  piecewise constant function can be recovered as a limit of a sequence of solutions of \eqref{weight_fid} using suitable weight functions, which have mass concentrating at the jump points of the data.

\newtheorem{exact_pc_w}[alpha_sgn]{Proposition}
\begin{exact_pc_w}\label{lbl:exact_pc_w}
Let $f_{0}\in\bv(\om)$ be a piecewise constant function as defined  in \eqref{f0} and let 
\[f=f_{0}+\eta,\] 
where $\eta$ is a continuous  function which vanishes at the jump points of $f_{0}$ (not necessarily satisfying \eqref{noise}). Define the sequence of weight functions 
\[w_{n}(x)=\sum_{i=1}^{N-1}n^{2}\mathcal{X}_{B(x_{i},\frac{1}{n})},\]
where $x_{i}$, $i=1,\ldots,N-1$ are the jump points of $f_{0}$.
Then for any sequence $(u_{n})_{n\in\NN}$ such that
\[u_{n}\in\underset{u\in\bv(\om)}{\operatorname{argmin}}\; \frac{1}{2}\int_{\om}w_{n}(f-u)^{2}dx+|Du|(\om),\]
we have 
\[\|f_{0}-u_{n}\|_{\infty}\to 0\quad \text{as }n\to\infty.\]
\end{exact_pc_w}

\begin{proof}
We can assume that $n$ is large enough so that the balls $B(x_{i},\frac{1}{n})$ are all inside $\om$. Note as well that since each $u_{n}$ is optimal, we have for every $n\in\NN$
\[\frac{1}{2}\int_{\om} w_{n}(f-u_{n})^{2}dx\le \frac{1}{2}\int_{\om} w_{n}(f-u_{n})^{2}dx+ |Du_{n}|(\om)\le |Df|(\om).\]
Using the above estimate, we find
\begin{align*}
\frac{1}{|B(x_{i},\frac{1}{n})|} \int_{B(x_{i},\frac{1}{n})} (f-u_{n})^{2}dx
&=  \frac{1}{n^{2}|B(x_{i},\frac{1}{n})|}  \int_{B(x_{i},\frac{1}{n})}    w_{n} (f-u_{n})^{2}dx
\le \frac{2n}{n^{2}} |Df|(\om),
\end{align*}
and hence 
\begin{equation}\label{lp_like}
\frac{1}{|B(x_{i},\frac{1}{n})|} \int_{B(x_{i},\frac{1}{n})} (f-u_{n})^{2}dx\to 0 \quad \text{as }n\to \infty,
\end{equation}
for every $x_{i}$, $i=1,\ldots,N-1$. Note that \eqref{lp_like} in fact implies
\begin{equation}\label{lp_like_unif}
\sup_{x\in (x_{i},x_{i}+\frac{1}{n})} \left(f(x)-u_{n}(x)\right)^{2} \to 0 \quad\text{and}\quad \sup_{x\in (x_{i}-\frac{1}{n},x_{i})} \left(f(x)-u_{n}(x)\right)^{2} \to 0  \quad  \text{as}\quad n\to \infty,
\end{equation}
for every $i=1,\ldots,N-1$.
Indeed, suppose that \eqref{lp_like_unif} does not hold, say that the first limit fails for some $x_{i}$. Then there exists an $\epsilon>0$ such that for every $k\in\NN$ there exists an $n_{k}\in\NN$, $n_{k}>k$, and a $x_{n_{k}}\in (x_{i},x_{i}+\frac{1}{n_{k}})$ such that
\begin{equation}\label{lp_tocontra}
\left ( f(x_{n_{k}})-u_{n_{k}}(x_{n_{k}}) \right )^{2}\ge \epsilon.
\end{equation}
Since $f$ is right continuous in $x_{i}$, we can pick a large enough $k\in\NN$ such that $|f(x_{i+})-f(x)|< \sqrt{\epsilon}/4$ for every $x\in (x_{i},x_{i}+\frac{1}{n_{k}})$. Then from \eqref{lp_tocontra} and Proposition \ref{lbl:weight_fid_opt} it follows that $u_{n_{k}}$ is constant in $(x_{i},x_{i}+\frac{1}{n_{k}})$, i.e., $u_{n_{k}}=c_{k}$ and $|f(x)-c_{k}|\ge \sqrt{\epsilon}/2$ for every $x\in (x_{i},x_{i}+\frac{1}{n_{k}})$. However, this violates \eqref{lp_like} as for large $k$ we would have
\begin{align*}
\frac{1}{|B(x_{i},\frac{1}{n_{k}})|} \int_{B(x_{i},\frac{1}{n_{k}})} (f-u_{n_{k}})^{2}dx&\ge \frac{2n_{k}}{n_{k}} \inf_{x\in(x_{i},x_{i}+\frac{1}{n_{k}})} (f(x)-u_{n_{k}}(x))^{2}\ge \epsilon.
\end{align*}
Thus \eqref{lp_like_unif} holds. Notice also that it  holds as well if we replace the squares with absolute values. Note now that from Proposition \ref{lbl:weight_fid_discont}, $u_{n}$ is continuous in the intervals $(x_{i},x_{i}+\frac{1}{n})$ and 
$(x_{i}-\frac{1}{n},x_{i})$. Then it can be readily checked that (any solution) $u_{n}$ will be just a monotonic interpolation between $u_{n}^{l}(x_{i}+\frac{1}{n})$ and $u_{n}^{r}(x_{i+1}-\frac{1}{n})$, for $i=1,\ldots, N-1$ and constant in $(a,x_{1}-\frac{1}{n})$ and $(x_{N-1}+\frac{1}{n},b)$ taking the values $u_{n}^{r}(x_{1}-\frac{1}{n})$ and $u_{n}^{l}(x_{N-1}+\frac{1}{n})$ respectively in those intervals. Since $f_{0}=f+\eta$ with $\eta$ being continuous and vanishing at every $x_{i}$, using \eqref{lp_like_unif} we have  that 
\begin{equation}\label{lp_like_unif2}
\sup_{x\in (x_{i},x_{i}+\frac{1}{n})} |f_{0}(x)-u_{n}(x)| \to 0 \quad\text{and}\quad \sup_{x\in (x_{i}-\frac{1}{n},x_{i})} |f_{0}(x)-u_{n}(x)| \to 0  \quad  \text{as}\quad n\to \infty.
\end{equation}
It is then straightforward that $u_{n}$ converges to $f_{0}$ uniformly.
\end{proof}

While the result of  Proposition \ref{lbl:exact_pc_w}  could be considered of theoretical value only, note however that using similar techniques as in Section \ref{sec:discontinuities} one can show that for the solution $u$ of \eqref{weight_fid} it holds
\[\underline{f}(x)\le\underline{u}(x)\le \overline{u}(x)\le\overline{f}(x),\quad \text{for every } x\in J_{f}.\]
Thus, if the noise function $\eta$ changes the values of $f_{0}$ at its jump points one cannot expect to recover $f_{0}$ using the limiting process of Proposition \ref{lbl:exact_pc_w}. It might happen for instance that $f=f_{0}+\eta$ satisfies $\underline{f_{0}}(x)<\underline{f}(x)< \overline{f}(x)<\overline{f_{0}}(x)$ at a jump point $x$ and thus any solution $u$ of \eqref{weight_fid} would satify
\[\underline{f_{0}}(x)<\underline{f}(x)\le \underline{u}(x)\le  \overline{u}(x)\le \overline{f}(x)<\overline{f_{0}}(x).\]

Let us remark also here that it is essential that the values of $w_{n}$ grow at a high rate  around $x_{i}$. For instance, if $\|w_{n}\|_{L^{1}(\om)}\to 0$,  we have that the solutions $u_{n}$ tend to a constant, up to a subsequence. Indeed, in that case we have
\[|Du_{n}|(\om)\le \frac{1}{2}\int_{\om}w_{n}(f-u_{n})^{2}dx+|Du_{n}|(\om)\le\frac{1}{2} \int_{\om}w_{n}f^{2}\,dx\le \frac{1}{2}\|f^{2}\|_{\infty}\|w_{n}\|_{L^{1}(\om)}\to 0,\quad \text{as }n\to\infty.\]
Denoting by $u_{\om}$ the mean value of $u$ in $\om$, from the Poincar\'e inequality we get 
\[|(u_{\om})_{n}|\le C(\|u_{n}\|_{L^{1}(\om)}+\|u_{n}-(u_{\om})_{n}\|_{L^{1}(\om)})\le C(1+|Du_{n}|(\om))<C,\]
for a generic constant $C$.
Thus, there exists a subsequence $(u_{\om})_{n_{k}}\to c$ as $k\to\infty$. Then
\begin{align*}
\|u_{n_{k}}-c\|_{L^{1}(\om)}&\le \|(u_{\om})_{n_{k}}-c\|_{L^{1}(\om)}+\|u_{n_{k}}-(u_{\om})_{n_{k}}\|_{L^{1}(\om)}\\
&\le |\om||(u_{\om})_{n_{k}}-c|+C|Du_{n_{k}}|(\om)\to 0,\quad\text{as }k\to\infty.
\end{align*}

\begin{figure}[t!]
\begin{center}
\begin{subfigure}[t]{0.35\textwidth}
	\centering
	\captionsetup{width=.9\linewidth}
	\includegraphics[width=0.9\textwidth]{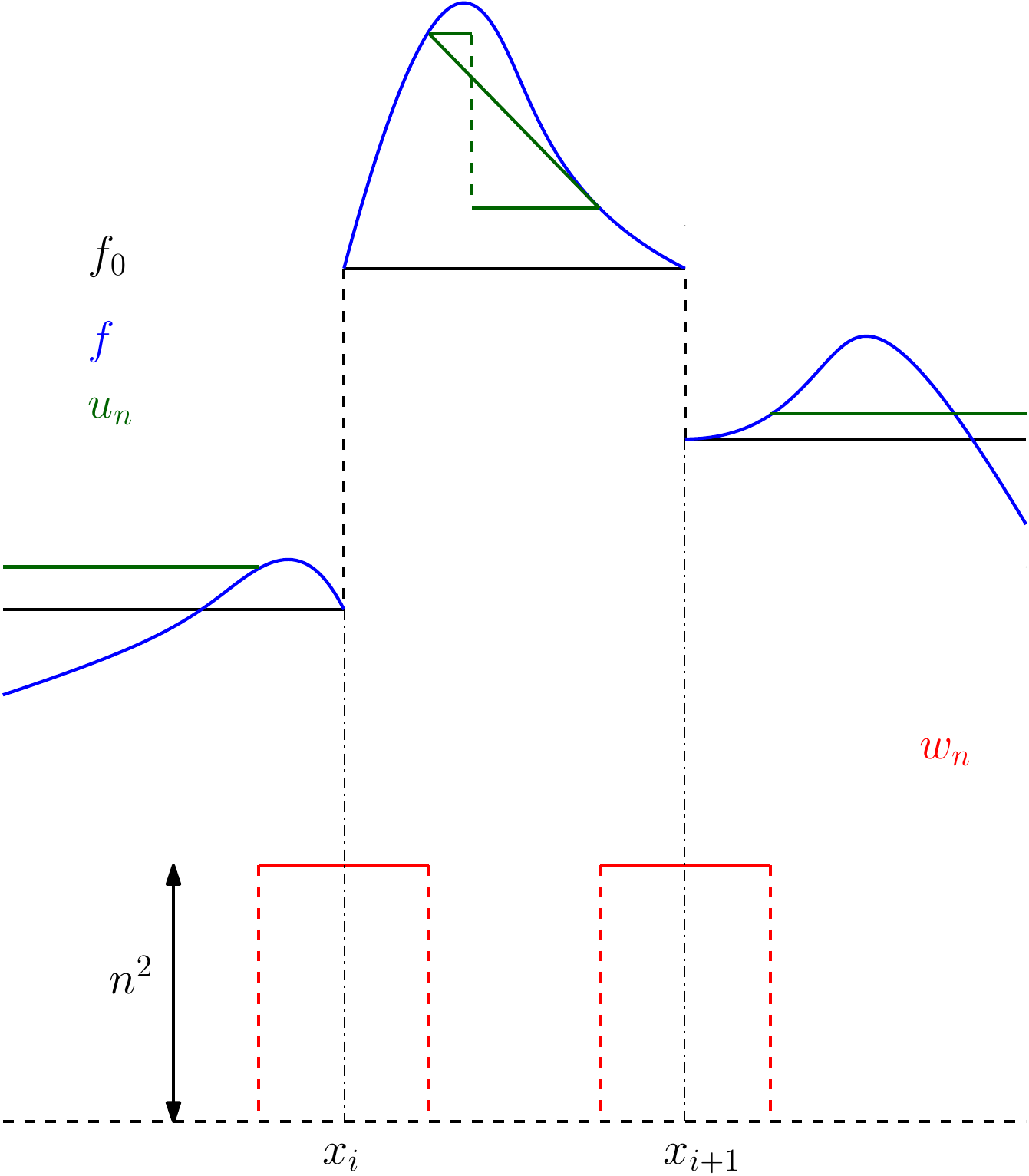}
	\caption{}
	\label{fig:highw1}
\end{subfigure}\hspace{0.05\textwidth}
\begin{subfigure}[t]{0.35\textwidth}
	\centering
	\captionsetup{width=.9\linewidth}
	\includegraphics[width=0.9\textwidth]{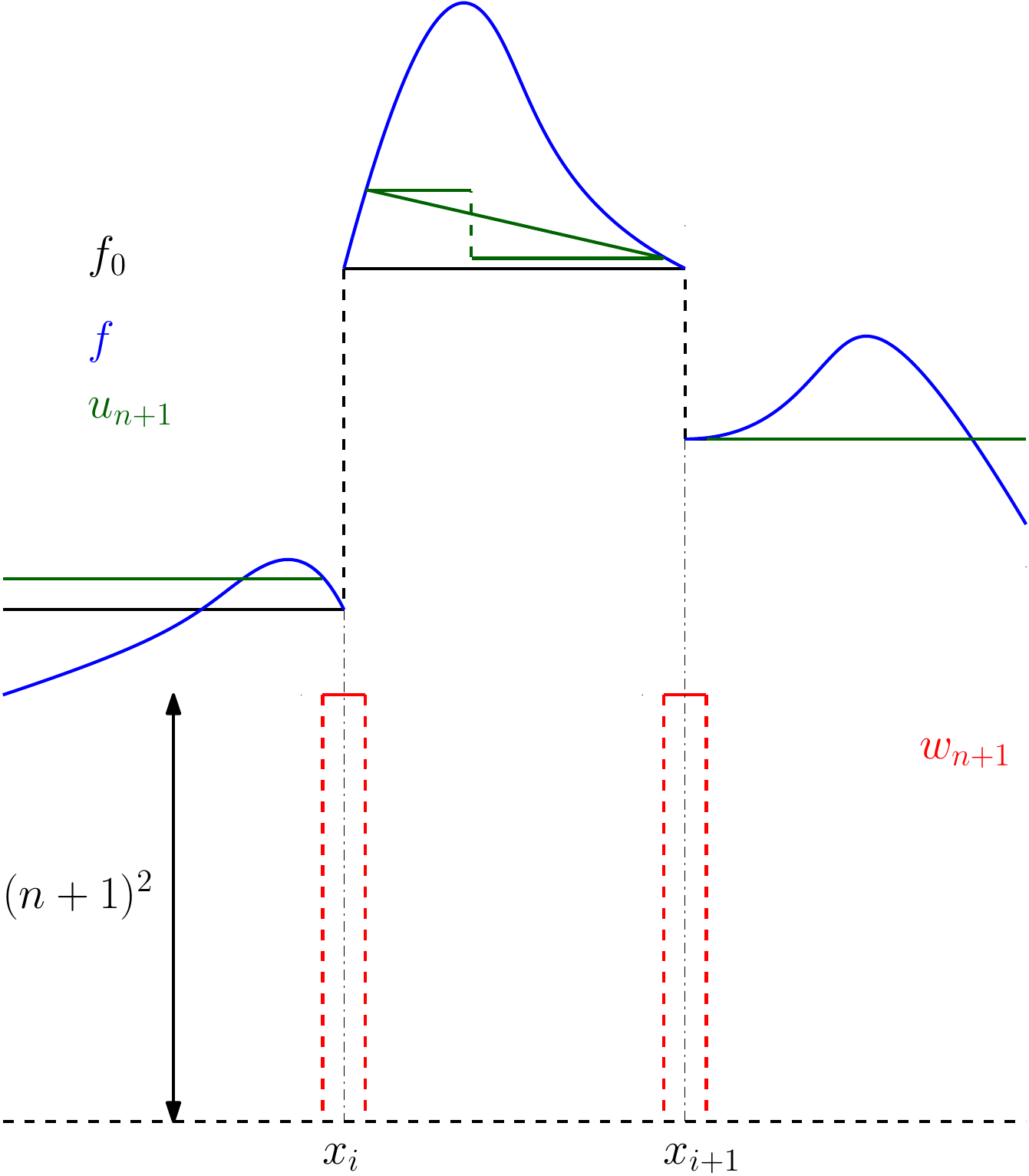}
	\caption{}
	\label{fig:highw2}
\end{subfigure}
\caption{Illustration of the result of Proposition \ref{lbl:exact_pc_w}. While  a piecewise constant function cannot be recovered exactly using the weighted fidelity total variation regularisation \eqref{weight_fid},   this can be achieved in the limit,  as the mass of the weight functions $w_{n}$ is concentrated with a high rate at the jump points of the data $f$}
\label{fig:highw}
\end{center}
\end{figure}

Figure \ref{fig:highw} describes the limiting process of Proposition \ref{lbl:exact_pc_w}. As $n$ tends to infinity and the mass of the weight $w$ is concentrated at the jump points $x_{i}$, the solution $u_{n}$ of \eqref{weight_fid} approaches $f$ (and hence $f_{0}$) near these jump points, while in the areas where $w_{n}$ is zero, $u_{n}$ is a monotonic interpolation. Note that this interpolation is not unique as depicted in the middle part of the Figures \ref{fig:highw1}--\ref{fig:highw2}.

Finally for the reader's convenience, we summarise in Table \ref{differences_table} the differences between the weighted total variation \eqref{weighted_rof}   and the weighted fidelity  \eqref{weight_fid} models.

{\footnotesize
\setlength\extrarowheight{3pt}
\begin {table}[t]
\begin{tabular}{| c | c | c |}
\cline{2-3}
    \multicolumn{1}{c|}{} &\textbf{Weighted $\boldsymbol{\tv}$ model }	& \textbf{Weighted fidelity model }	\\[3pt]\cline{2-3} \cline{1-1}
Creation of new discontinuities &Potentially    & 	No 	\\[3pt]\hline
Solution $u$ is constant when $u\ne f$ &Not necessarily   &	Yes \\[3pt]\hline
\begin{tabular}{@{}c@{}}Exact recovery of piecewise  \\[-3pt] constant functions\\[2pt]\end{tabular}&Yes, setting $\alpha=0$ at jump points & Only in the limit $w\to\infty$ at jump points	\\[3pt]\hline
\end{tabular}
\vspace{10pt}
\caption{Summary of the differences between the weighted total variation \eqref{weighted_rof}   and the weighted fidelity  \eqref{weight_fid} models}
\label{differences_table}
\end{table}
}

\section{Conclusions}\label{conclusions}

We have performed a thorough analysis of the weighted total variation regularisation problem in dimension one. We studied conditions under which discontinuities are created in the solution, at points where the data function is continuous. A variety of analytical and numerical results was provided. Moreover,  in contrast to the standard scalar total variation regularisation only a partial version of the parameter semigroup property holds in the weighted case. We further computed some analytic solutions for simple data and weight functions. We were able to infer a bound of the total variation of the solution by the total variation of the data, an estimate which was used to show the well-posedness of the weighted total variation minimisation problem even in the case of vanishing weight function. It was shown, that by using vanishing weights one can recover exactly piecewise constant data. Finally, the total variation problem with weighted fidelity term was considered and it was shown that even for very simple data, its solutions can be very different to those of the weighted total variation problem. 

Even though a few of our results hold in arbitrary dimension, the main statements of this paper are in dimension one and their proofs rely on a fine scale one dimensional analysis. Extension of these results in higher dimensions is a subject of future research. We note that a higher dimension study poses some extra challenges, regarding for example the derivation of the predual problem, a characterisation of the subdifferential of the weighted $\tv$, as well as deriving a bound on the total variation of the solution as we have done in Section \ref{sec:boundTV}.

\subsection*{Acknowledgements}This research was partially carried out in the framework of {\sc Matheon} supported by the Einstein Foundation Berlin within the ECMath projects OT1, SE5 and SE15 as well as by the DFG under grant no.~HI 1466/7-1 ``Free Boundary Problems and Level Set Methods'' and SFB/TRR154. KP acknowledges the financial support of Alexander von Humboldt Foundation. A large part of this work was done while  KP was at the Institute for Mathematics, Humboldt University of Berlin.

\bibliographystyle{amsalpha}
\bibliography{kostasbib}

\end{document}